\documentclass[11pt,a4paper]{article}

\usepackage{amsmath} 
\usepackage{amssymb}  
\usepackage{graphics}
\usepackage{graphicx}
\usepackage{verbatim}
\usepackage{epsfig}
\usepackage{amsfonts}
\usepackage{mathrsfs}
\usepackage{hyperref}
\hypersetup{
	colorlinks=true,
	linkcolor=black,
	filecolor=black,
	urlcolor=cyan,
	citecolor=black,
}

\usepackage{bbm}
\usepackage{color}
\usepackage{makeidx}
\usepackage{color}
\usepackage{subfigure}
\usepackage{amsfonts}%
\usepackage{algorithm}
\usepackage{algorithmic}
\usepackage{amsthm}
\usepackage{multicol}
\usepackage{enumitem}
\usepackage{multirow}

\def\opi{{\mathrm{i}\mkern1mu}}

\def\eps{\epsilon}

\newcommand{\bR}{\mathbb{R}}
\newcommand{\bC}{\mathbb{C}}

\newcommand{\real}{\operatorname*{Re}}
\newcommand{\imag}{\operatorname*{Im}}
\newcommand{\iu}{\mathrm{i}}
\newcommand{\tol}{tol}
\newcommand{\Id}{\mathrm{I}}
\newcommand{\e}{\mathrm{e}}
\renewcommand{\epsilon}{\varepsilon}

\newlist{aims}{enumerate}{1}
\setlist[aims,1]{
	label={Aim~\arabic*},
	leftmargin=*,
	align=left,
	labelsep=2mm,
}

\newtheorem{theorem}{Theorem}
\newtheorem{lemma}[theorem]{Lemma}

\newtheorem{remark}[theorem]{Remark}

\makeindex

\title{Pseudospectral roaming contour integral methods for convection-diffusion equations}
\author{Nicola Guglielmi\footnotemark[1] \and Mar\'ia L\'opez-Fern\'andez\footnotemark[2] \and Mattia Manucci\footnotemark[3]}

\begin{document}

\maketitle
\renewcommand{\thefootnote}{\fnsymbol{footnote}}
\footnotetext[1]{Gran Sasso Science Institute,
                 via Crispi 7,
                 L'Aquila, Italy. Email: {\tt nicola.guglielmi@gssi.it}}
\footnotetext[2]{Departamento de An\'alisis Matem\'atico, Estad\'istica e I.O. y Matem\'atica Aplicada.
Facultad de Ciencias. Universidad de M\'alaga.
Bulevar Louis Pasteur, 31
29010  M\'alaga,  Spain. Email: {\tt maria.lopezf@uma.es}}
\footnotetext[3]{Gran Sasso Science Institute,
                 via Crispi 7,
                 L'Aquila, Italy. Email: {\tt mattia.manucci@gssi.it}}

\renewcommand{\thefootnote}{\arabic{footnote}}

\begin{abstract}
We generalize ideas in the recent literature and develop new ones in order to propose a general class of
contour integral methods for linear convection--diffusion PDEs and in particular for those arising in finance.
These methods aim to provide a numerical approximation of the solution by computing its
inverse Laplace transform. The choice of the integration contour is
determined by the computation of a few suitably weighted pseudo-spectral level sets of the leading operator of the equation.
Parabolic and hyperbolic profiles proposed in the literature are investigated and compared
to the elliptic contour originally proposed by Guglielmi, L\'opez-Fern\'andez and Nino 2020, see \cite{GLN20}.
In summary, the article
\begin{aims}
\item[(i) ]   provides a comparison among three different integration profiles;
\item[(ii) ]  proposes a new fast pseudospectral roaming method;
\item[(iii) ] optimizes the selection of time windows on which one may arbitrarily approximate the
solution by no extra computational cost with respect to the case of a fixed time instant;
\item[(iv) ]  focuses extensively on computational aspects and it is the reference of the MATLAB code  \cite{GLMcode}, where all algorithms described here are implemented.
\end{aims}

\end{abstract}

{\bf Keywords:} Contour integral methods, weighted pseudospectra, inverse Laplace transform, convection-diffussion equations, elliptic contour, parabolic contour, hyperbolic contour, quadrature for analytic integrands.

{\bf AMS subject classifications:} 65L05, 65R10, 65J10,65M20, 91-08.

\pagestyle{myheadings}
\thispagestyle{plain}
\markboth{N.~Guglielmi, M.~L\'opez-Fern\'andez and M.~Manucci}{Pseudospectral roaming contour integral methods for convection-diffusion equations}

\section{Introduction}

We consider convection diffusion PDEs of the form
\begin{eqnarray}
&& \frac{\partial U}{\partial t}(x,t) = {\cal{A}} (x) U(x,t) + f(x,t),
\nonumber
\\
&& + \mathrm{B.C.}
\label{eq:pde}
\\
&& U(x,0) = U_0(x)
\nonumber
\end{eqnarray}
with ${\cal{A}}$ a linear second order elliptic operator. After discretizing the problem in space, we study efficient numerical integrators for the Cauchy problem
\begin{equation}\label{eq:cauchy}
\dot{u} = A u + b(t), \qquad u(0)=u_0, \qquad t > 0,
\end{equation}
with $A$ representing a suitable discretization of the  elliptic operator $\cal{A}$ and $b$ is a source term which possibly includes boundary contributions. We are particularly interested in equations arising in mathematical finance, such as Black--Scholes, Heston or Heston-Hull-White equations \cite{BS,He,Hu}, but our approach is by no means restricted to them.

In order to approximate the solution $u(t)$ to \eqref{eq:cauchy} one may use Runge-Kutta methods, multistep integrators as well as splitting schemes.
The drawback of these time-stepping schemes is that in order to approximate the solution at a certain time $T=t_n$, one needs to compute an approximation
of the solution at grid points $0<t_1<t_2<\ldots<t_n$, which would be particularly demanding if $T$ is large.
As an alternative, it is possible to derive methods based on the Laplace transform and its numerical inversion, which do not advance on a grid.
In the literature this approach has been widely studied for pure diffusion equations (see e.g.  \cite{GaMa,LP,LPS,SST}) and for convection diffusion equations recently in \cite{GLN20}. An important case is when the time $T$ at which one is interested to determine the solution is not known exactly but is uncertain although it belongs to a certain time window
of moderate size. In such case it would be convenient to develop methods which do not require substantial additional computations with respect to the model case when $T$ is fixed a priori. This is another goal of this article, i.e. to discuss and analyze methods able to approximate the solution on suitable time windows.

In the sequel of the article - when not indicated differently - the considered norm is the spectral one, that is the matrix norm induced by the vector Euclidean norm.

The magnitude of the resolvent norm $\left\|\left(zI-A\right)^{-1}\right\|$ has a crucial role in the rate of convergence of any contour integral method based on Laplace transformation. Due to this, the choice and parametrization of the integration contour is of main importance. In a recent paper \cite{GLN20}, an elliptic profile has been proposed, in connection to the knowledge of the $\eps$-pseudospectrum of $A$ (see \cite{TreE})
\begin{eqnarray}
\sigma_\eps(A) & = & \left\{ z \in \bC \ : \left\|\left(z\Id-A\right)^{-1}\right\| \le \frac{1}{\epsilon}
\right\}, \label{eq:epsps}
\end{eqnarray}
for suitable values $\eps > 0$.
Since $A$ is in general non-normal, due to the convection terms in the operator $\mathcal{A}$, the pseudospectrum may increase fast around the spectrum of $A$, making the problem challenging.
We assume the existence of the Laplace transform of $b$ and that it admits a bounded analytic extension to a suitable region of the complex plane outside the spectrum of $A$.
We then apply the Laplace transform to \eqref{eq:cauchy}, which yields, for the Laplace transform of $u$, $\hat{u}={\cal L}(u)$,
\begin{equation}\label{eq:LT}
\hat{u}(z)=\left(zI-A\right)^{-1}\left(u_0+\hat{b}(z)\right)\,,
\end{equation}
where $\hat{b}={\cal L}(b)$ and $I$ stands for the identity matrix.

After solving \eqref{eq:LT}, we reobtain $u$ by considering the inverse Laplace transform
\begin{equation}\label{eq:bromwich}
u(t)=\frac{1}{2\pi \opi}\int_{\Gamma}\e^{zt}\hat{u}(z)\,dz,
\end{equation}
being the contour $\Gamma$ an open piecewise smooth curve running from $-\opi\infty$ to $+\opi\infty$ surrounding all singularities of $\hat{u}$.
To approximate the Bromwich integral \eqref{eq:bromwich}, we parameterize the integration contour $\Gamma$ by $z=z(x)$, $x\in \mathbb{R}$, for a suitable mapping $z(x)$, so that
\[
\int_{\Gamma}\e^{zt}\hat{u}(z)\,dz = \int\limits_{\mathbb{R}} G(x) dx,
\]
with $G$ appropriately defined. Since we are interested in approximating $u(t)$ within precision $\tol$, we will only consider the portion of the Bromwich integral parameterized in $[-c\pi,c\pi]$, this is,
\[
I = \int\limits_{\bR} G(x) dx \approx \int\limits_{-c \pi}^{c \pi} G(x) dx,
\]
for a certain truncation parameter $c \in (0,c_{\text{max}})$, which we determine by the estimate
\[
| G(c \pi) | = \tol
\]
for $\tol$ the desired accuracy.
Finally, the application of a quadrature formula to approximate \eqref{eq:bromwich} provides a numerical approximation of $u$, for a given time $t$,
or even time windows of the form $[t_0, \Lambda t_0]$, $\Lambda > 1$, without need of
computing it at intermediate time instants.
An application of the trapezoidal rule
\begin{equation} \label{eq:tr}
	I_N=\frac{2 c \pi}{N}\sum_{j=1}^{N-1}G(\xi_j)\, \quad \mbox{with } \ \xi_j=-c \pi + j \frac{2 c \pi }{N},\quad j=1,\ldots,N-1.
\end{equation}
provides the desired approximation $I_N$ of $I$.
Note that the computation of each term in the summation \eqref{eq:tr} requires the solution of a shifted linear system $A - z(\xi_j) I$.
An advantage of the method we propose is that these computations can be done in parallel. Furthermore, if the integrand is conjugate symmetric, the number of addends and thus, the number of linear systems, can be halved. This is often the case in applications, such as the ones we consider here.

Assuming that the Laplace transform can be analytically extended to the left half of the complex plane and that this extension is properly bounded with respect to $z$, several authors have proposed different contour profiles and parametrizations for $\Gamma$.  We refer the reader to the recent article \cite{GLN20} for a detailed review of the
literature concerning the crucial choice of the profile $\Gamma$.
In this paper we extend the results of \cite{GLN20} by considering not only elliptic but also parabolic and hyperbolic profiles, which we compare.
The parametrization of all contours is optimized by using the knowledge of the pseudospectrum of $A$ on a region of the complex plane surrounding the spectrum of $A$. A novel Newton iteration is developed to obtain the required knowledge of pseudospectral level sets. In this way we are able to determine more specifically, accurately and efficiently the required pseudospectral level curves and avoid the use of the software {\tt eigtool} \cite{eigtool} as it was done in \cite{GLN20}.
We notice that since the exponential factor in \eqref{eq:bromwich} reduces the norm of the integrand function when $\real(z)$ is sufficiently large and negative, we have to control
the pseudospectrum of $A$ only in a vertical strip of the complex plane.
A main advantage of the method we discuss is that it provides an approximation of the solution by a prescribed accuracy $\tol$, simply increasing the number of quadrature
points on the integration contour, without changing the integration profile and taking advantage of previous computations.
The paper is organized as follows.
In Section \ref{sec:contour} we describe the three contours we consider (elliptic, parabolic and hyperbolic).
In Section \ref{sec:pseudospec} we present a new method to obtain approximations of pseudospectral level sets, which does not require making use of Eigtool.
In Section \ref{sec:BS} we study in full detail the pseudospectra of the 1D Black and Scholes operator.
In Section \ref{sec:param} we provide the determination of the parameters characterizing the whole procedure based on sharp error estimates.
In Section \ref{sec:compar} we compare the profiles and present some numerical illustrations.
Finally, in Section \ref{sec:code} we focus our attention on implementation issues and present a Matlab code aimed to approximate the solution of the problem by means of any of the considered methods.

\section{The integration contours} \label{sec:contour}

We propose contours $\Gamma$ in \eqref{eq:bromwich} which are either elliptic, parabolic or hyperbolic arcs, possibly linked to half-lines.

\subsection{Elliptic profile: a review from \cite{GLN20}}
In \cite{GLN20}, $\Gamma$ is parameterized by
\begin{equation}\label{eq:Gamma}
z(x)=\left\{\begin{array}{ll}
\ell_1(x), \quad    & x \le -\frac{\pi}{2}, \\[1mm]
z(x), \quad  & -\frac{\pi}{2} \le x \le \frac{\pi}{2}, \\[1mm]
\ell_2(x), \quad  & x\ge \frac{\pi}{2}, \\
\end{array}\right.
\end{equation}
where, for constant parameters $A_1, A_2, A_3$ to be determined,
\[
z(x) = A_1\cos x+\opi A_2\sin x+A_3
\]
parametrizes an elliptic arc and
\[
\ell_{1}(x) = A_3 + x+\frac{\pi}{2}-\opi\left(A_2 - d\left(x+\frac{\pi}{2}\right)\right),\quad
\ell_2(x) = A_3-x+\frac{\pi}{2}+\opi\left(A_2+d\left(x-\frac{\pi}{2}\right)\right)
\]
parametrize two half-lines, see Figure~\ref{plot_profile1}.
\begin{figure}\label{plot_profile1}
\begin{center}
\includegraphics[viewport= 0 100 600 350, width=.75\textwidth]{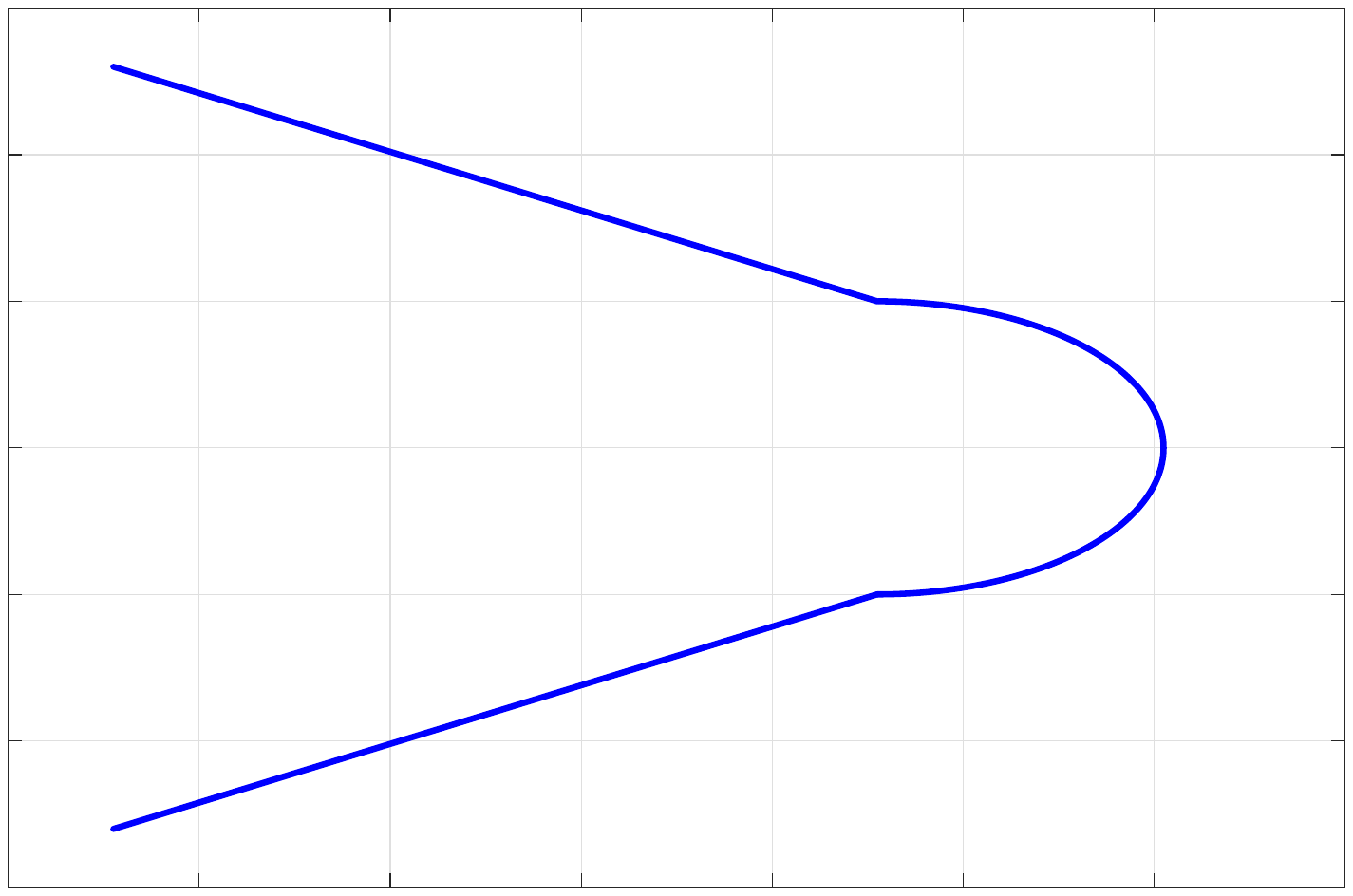}
\end{center}
\caption{The integration profile $\Gamma$ from \cite{GLN20}.}
\end{figure}
The choice of the parameters $A_1, A_2$ ad $A_3$ is discussed on \cite{GLN20} and is fundamentally
based on the knowledge of the $\eps$-pseudospectrum of $A$ in a rectangular region surrounding the
rightmost section of the spectrum.

The only section of the integration contour $\Gamma$ \eqref{eq:Gamma} that it is actually  used in practise is the arc of ellipse parameterized by $z$. In order to improve the performance of the quadrature, $z$ is extended to a rectangle in the complex plane by
\begin{equation}\label{eq:mapping}
	z(x+iy)=A_1(y)\cos x+iA_2(y)\sin x+A_3(y), \quad x\in \left[-\frac{\pi}{2},\frac{\pi}{2}\right], y\in[-a,a],
\end{equation}
for a certain parameter $a>0$ to be determined. We require \eqref{eq:mapping} to be holomorphic in the rectangle
\[
R = [-\pi/2,\pi/2] \times [-\iu a, \iu a]
\]
and thus impose the Cauchy-Riemann equations. In this way we obtain that $A_3$ is necessarily a constant,
\begin{eqnarray}
\label{eq:A1}
A_1(y)&=& a_1\e^y+a_2\e^{-y},\\
\label{eq:A2}
A_2(y)&=& a_2\e^{-y}-a_1\e^y,
\end{eqnarray}
with $a_1$ and $a_2$ real constants. The resulting mapping turns out to be entire. The holomorphy of $z$ in the rectangle leads to the exponential convergence of the trapezoidal rule when it is applied to the integral resulting after parametrizing the elliptic contour by $z(x)$, being the rate of convergence increasing with $a$, see \cite{GLN20}.
\begin{figure}[h!]
\begin{center}
\includegraphics[viewport=69   296   525   544, scale=0.5]{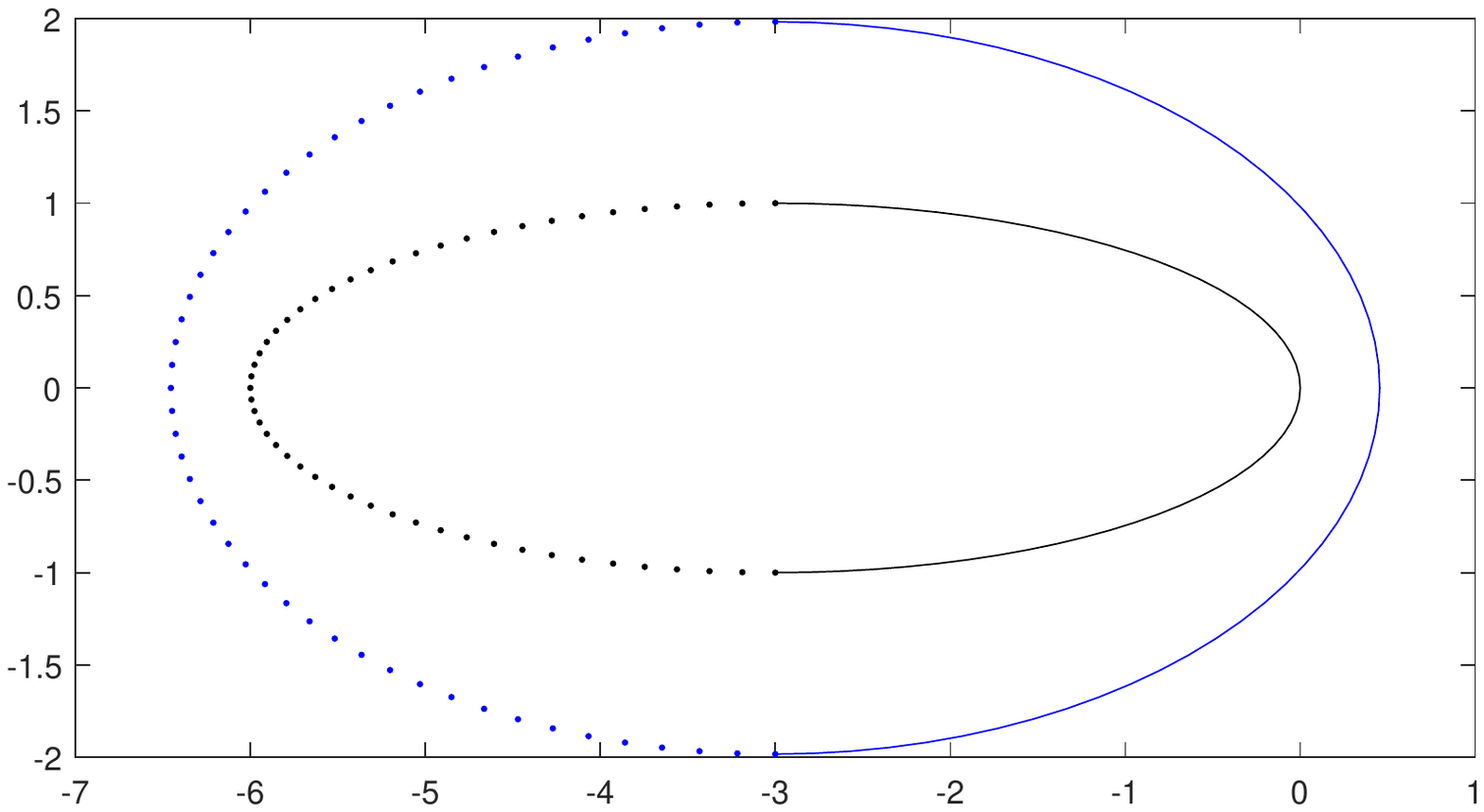}
\caption{The ellipse $\Gamma_{left}$ (in black) and the integration profile (in blue). \label{fig_ell}}
\end{center}
\end{figure}
The rectangle $R$ is mapped into an elliptic ring-shaped region. In particular the upper horizontal side of the rectangle is mapped into the inner ellipse $\Gamma_{left}$ (black) in Figure \ref{fig_ell} and is selected in a way that its rightmost section (continuous line) is external to the $\eps$-pseudospectrum for a suitable value $\eps$.
In order to select $\Gamma_{left}$, we fix the center of the ellipse $z^L$, its right intersection with the real axis ${z^R}$ and one interpolation point ${z^B}$.
In particular $z^L$ is such that ${\rm e}^{z^L t} \approx {\rm eps}$ (the working precision) and $z^R$ is the rightmost intersection point of the $\eps$-pseudospectrum of $A$ and the real axis. The interpolation point ${z^B}=d + \iu r$ is chosen in such a way that the ellipse encloses the $\eps$-pseudospectrum of $A$ for a suitable $\eps$, as well as the possible singularities of $\hat{b}$.
Next the half-elliptic integration profile
\begin{equation} \label{eq:Gamma1}
\Gamma: \quad z(x)=(a_1+a_2)\cos x+\opi(a_2-a_1)\sin x+A_3,
\end{equation}
is determined, with coefficients $a_1,a_2,A_3$ depending on the unique free parameter $a$ (see \cite{GLN20} for the details).
Indeed, imposing the ellipse $\Gamma_{left}$ to be centered at $z^L$ and to pass through the points $z^R$ and $z^B=d+\opi r $, we get
\begin{equation}\label{eq1}
a_1\e^a+a_2\e^{-a}=z^R-z^L,\quad
a_2\e^{-a}-a_1\e^a=\frac{r}{\sin(\theta)},\quad
A_3=z^L,
\end{equation}
where
\begin{equation*} 
\theta=\arccos\left(\frac{d-z^L}{z^R-z^L}\right)\,.
\end{equation*}

Solving \eqref{eq1} for $a_1,a_2, A_3$ yields
\begin{eqnarray}\label{a1ell}
a_1 & = & \frac{\e^{-a}}{2}\left(z^R-z^L-\frac{r}{\sin(\theta)}\right) \\
\label{a2ell}
a_2 & = & \frac{\e^{ a}}{2}\left(z^R-z^L+\frac{r}{\sin(\theta)}\right) \\
A_3 & = & z^L \nonumber
\end{eqnarray}
which only depend on the real parameter $a$.
\subsection{Parabolic profile}
For $y$ fixed, the mapping
\begin{equation} \label{eq:par}
z(x + \iu y) = - x^2 - 2 \iu x A_2(y) + A_1(y), \quad x\in \bR,
\end{equation}
defines a parabola symmetric with respect to the real axis in the complex plane. In order to obtain a holomorphic parametrization of the parabola, we impose the Cauchy-Riemann equations to determine $A_1$ and $A_2$. This yields
\begin{eqnarray}
\nonumber
A_1(y) & = & y^2 + 2 a_1 y + a_2, \\[2mm]
\nonumber
A_2(y) & = & y + a_1,
\end{eqnarray}
with real constants $a_1$ and $a_2$. Since the parabolic profile is symmetric with respect to the real axis we only need of two points in order to determine it uniquely.
We proceed by constructing the conformal mapping
\begin{equation} \nonumber
z:\bR\times[-\opi a,\opi a]\rightarrow \mathbb{C}
\end{equation}
for a certain positive $a$.
Proceeding analogously as for the elliptic profile, we call $\Gamma_{left} = z(\bR + \opi a)$, this is, the parabola limiting to the left the image by $z$ of the horizontal strip $|\imag y| \le a$. We impose the vertex of $\Gamma_{left}$ to be the rightmost intersection point of the $\eps$-pseudospectrum of $A$, that we call $z^R$, and a further
interpolation control point $z^B = d + \iu r$. In this way we obtain the uniparametric family of values
\begin{eqnarray}
a_1 & = & -\frac{r}{2 \sqrt{ z^R - d}} - a,
\label{a1para}\\[2mm]
a_2 & = & z^R - a^2 - 2 a a_1 \label{a2para},
\end{eqnarray}
depending on the free parameter $a$, which is the band width of the analyticity domain of $z$.
\subsection{Hyperbolic profile}
As it was done in \cite{LP}, we set $w=x+iy$ and notice that the function {of $w$}
\[{
-\sin(\alpha+\opi w) = \phi+\opi \psi, \quad \mbox{with }\ \phi, \psi \in \bR,
}\]
maps the horizontal line $\imag w=y$ into the left branch of the hyperbola
\[{
\left(\frac{\phi}{\sin(\alpha-y)} \right)^2- \left(\frac{\psi}{\cos(\alpha-y)} \right)^2 =1,
}\]
whose asymptotes are given by {$\phi + \opi \psi$} with
\[{
\psi = \pm\cot(\alpha-y) \phi.
}\]

In order to better control the position of the center, the length of the horizontal semi axis and the angle of the asymptotes, we start here from the more general expression
\begin{equation}\label{genhyp}
z(x+\opi y) = A_3(y) - A_2(y)\sin(A_1(y) -y + \opi x ).
\end{equation}
After imposing the Cauchy-Riemann equations we actually obtain that $A_1$, $A_2$ and $A_3$ must be constant with respect to $y$, leaving a map of the form
\begin{equation}\label{defhyp}
z(x+\opi y) = a_3 - a_2 \sin(a_1-y)\cosh x - \opi a_2 \cos(a_1-y) \sinh x.
\end{equation}
To obtain optimal error estimates, we need to control the image of the horizontal strip $|\imag y|\le a$ under $z$, which is the region in the complex planed limited by the two branches of hyperbola $z(x-\opi a)$ (the one closest and furthest to the left) and $z(x+\opi a)$, $x\in \bR$. Similarly to what we do for the other type of contours, we prescribe $z(x-\opi a)$, $x\in \bR$, to be an appropriate critical hyperbola $\Gamma_{left}$, with vertex at ${z^R}$, center at ${z^C}$ and passing through a third point $d+\opi r$. In this way we obtain
\begin{align}
a_3&= {z^C} \label{a3hyp}\\
a_3-a_2 \sin(a_1+a) &= z^R  \quad \Rightarrow \quad a_2=\frac{z^C-z^R}{\sin(a_1+a)} \label{a2hyp}\\
a_3 - a_2 \sin(a_1+a)\cosh x &=d  \quad \Rightarrow \quad  z^C-(z^C-z^R) \cosh x = d \nonumber\\
-a_2 \cos(a_1+a) \sinh x &=r \quad \Rightarrow \quad  -(z^C-z^R)\cot(a_1+a) \sinh x = r, \nonumber
\end{align}
so that
\[
\tan(a_1+a) = \frac{z^R-z^C}{r}\sqrt{\cosh^2 x-1}=\frac{z^R-z^C}{r}\sqrt{\left(\frac{d-z^C}{z^R-z^C} \right)^2 -1}
\]
and
\begin{equation}\label{a1hyp}
a_1= \arctan\left( \frac{1}{r} \sqrt{(d-z^C)^2-(z^R-z^C)^2} \right)-a.
\end{equation}
Equations \eqref{a1hyp}, \eqref{a2hyp} and \eqref{a3hyp} give a uniparametric family of solutions for the parameters $a_1, a_2$ and $a_3$, respectively, depending on the width $a$ of the strip of analyticity of the mapping $z$.
We notice that the orientation we need to invert the Laplace transform is actually the opposite to the one in \eqref{defhyp}, as $x$ runs from $-\infty$ to $\infty$. This is resolved by simply taking the conjugate of \eqref{defhyp} as parametrization.

\section{Roaming pseudospectral sets}
\label{sec:pseudospec}
We consider here the case of a parabolic contour. 
The idea can be extended in a straightforward way to the elliptic and the hyperbolic contour.
We start from an initial internal parabola uniquely identified by a prescribed interpolation point $w=d+\iu r$ (together with the vertex $z^R$, which we consider
fixed). {Since the parabola is symmetric with respect to the real axis, we only work with its upper half, with positive imaginary part, and consider a set of $M$ points $z_k = \phi_k +\opi \psi_k$, $k=1,\ldots,M,$ on this curve, increasingly ordered according to their real parts $\phi_k$.}
{The parametric form of the {\em inner} parabola is
\begin{equation} \label{eq:innerp}{
z(x) = -x^2 + z^R + \frac{\iu r x}{\sqrt{z^R-d}}, \qquad x \in \bR.
}
\end{equation}
{Setting as in the previous section $z(x)=\phi+\opi \psi$, with $\psi>0$, this means that if we fix the abscissa $\phi=\real{(z)}$ we obtain
\[
\phi = z^R - x^2 \qquad \Longrightarrow \qquad x = \sqrt{z^R -{\phi}}
\]
}
which uniquely defines the argument of the parametrization $x > 0$, and consequently
\[
{\psi} = \frac{r x}{\sqrt{z^R-d}},
\]
which depends on $r$ and $d$. We easily obtain
{\begin{eqnarray}
\frac{\partial \psi}{\partial d} & = & \frac{x  r}{2 (z^R-d)^{3/2}}
\nonumber
\\
\frac{\partial \psi}{\partial r} & = & \frac{x }{\sqrt{z^R-d}}.
\nonumber
\end{eqnarray}}}
We discuss two possible approaches and after numerical simulations we have adopted the second in our code. The idea is that of modifying the interpolation point $w = d + \iu r$, which determines the parabola (together with the vertex $z^R$, which we consider fixed) by using variational results for simple singular values. In the sequel we shall make use of the following classical result on the derivative of a simple eigenvalue (see e.g. \cite{K,HJ}).
\begin{lemma}
	\label{lemma:dsv}
	Let $D(t)$ be a differentiable matrix-valued function in a neighborhood of $t_0$.
	Let
	\begin{equation}
	D(t) = U(t) \Sigma(t) V(t)^{*} = \sum\limits_{i} u_i(t) \sigma_i(t) v_i(t)^*
	\label{eq:svd}
	\end{equation}
	be a smooth (with respect to $t$) singular value decomposition of the matrix $D(t)$ and
	$\sigma(t)$ be a certain singular value of $D(t)$ converging to a simple singular value  $\hat \sigma \ne 0$ of $D_0 = D(t_0)$.
	
	If $\hat u, \hat v$ are the associated left and right singular vectors, respectively,
	the function $\sigma(t)$  is differentiable near $t=t_0$ with
	\begin{equation} \label{der:sigma}
	\dot{\sigma}(t_0) = \real \bigl( \hat u^{*} \dot{D}_0 \hat v \bigr) \qquad \mbox{\rm with} \ \dot{D}_0 = \dot{D}(t_0) .
	\end{equation}
\end{lemma}
\begin{proof}
	We have that $\sigma(t)^2$ is an eigenvalue of $D(t) D(t)^* = U(t) \Sigma(t)^2 U(t)^*$.
	
	At $t=t_0$ the left and right eigenvectors associated to $\hat \sigma^2$ coincide and are equal to $\hat u$, having unit norm.
	Note that $\hat u$ is a certain column of $U(t_0)$ determined by the position of $\hat \sigma^2$ in the diagonal matrix $\Sigma(t_0)^2$.
	
	Then - omitting the ubiquitous dependence on $t$ - by \cite[Theorem 6.3.12]{HJ} we get for $t=t_0$
	\begin{equation*}
	\frac{d}{dt} \sigma^2(t) \Big|_{t=t_0} = 2 \hat\sigma \dot{\sigma}(t_0) = \frac{\hat u^* \bigl( \dot{D}_0 D^*_0 + D_0 \dot D^*_0 \bigr) \hat u}{\hat u^* \hat u} =
	2 \real \bigl( \hat u^* \dot{D}_0 D^*_0 \hat u \bigr).
	\end{equation*}
	Now using the fact (see \eqref{eq:svd}) that $D^*_0 \hat u = \hat v \hat \sigma$ we get \eqref{der:sigma}.
\end{proof}

\subsection{An optimal although expensive approach}
{Differently from what proposed in \cite{GLN20}, where both pseudospectral and a weighted version of pseudospectral level sets were considered, here we will only consider the weighted ones. This choice is motivated by the fact that we do not relay anymore on Eigtool for the computation of pseudospectral level sets. Instead, we will directly compute the value of a weighted version of the pseudospectrum, which is the one that really matters in the error estimate, see \cite[Theorem 2]{GLN20}. More precisely, we define the weighted $\varepsilon$-pseudospectrum as the set
\begin{equation}
	\sigma_{\varepsilon,t}(A)=\left\{z\in\mathbb{C} : \e^{\real(z)t}\Big\|(zI-A)^{-1}\Big\|\ge\frac{1}{\varepsilon}\right\},
\end{equation}
which is equivalent to
\begin{equation}
\sigma_{\varepsilon,t}(A)=\left\{z\in\mathbb{C} : \e^{-\real(z)t}\sigma_{\min}\Big(zI-A\Big)\le\varepsilon\right\}.
\end{equation}
{with $\sigma_{\min}\Big(zI-A\Big)$ denoting the smallest singular value of $zI-A$}. In particular, we are interested on the boundary of the weighted $\varepsilon$-pseudospectral level set, this is}
\begin{equation}\label{pslevelset}
\partial \sigma_{\varepsilon,t}(A)=\left\{z\in\mathbb{C} : \e^{\real(z)t}\Big\|(zI-A)^{-1}\Big\|=\frac{1}{\varepsilon}\right\}.
\end{equation}
Note that for $t=0$ we recover the standard definition of pseudospectrum and $\varepsilon$-pseudospectrum.

For a fixed target $\eps$, we impose the set in \eqref{pslevelset} to lay internal to the parabola. A natural approach to achieve this goal is defining the functional
\[{
\mathcal{F}(d,r) = \frac12 \sum\limits_{k=1}^{{M}} \Bigl(\eps - \tilde{\sigma}_k(d,r) \Bigr)_{+}^2
}\]
where
\begin{equation}\label{eq:sigk}
	{\tilde{\sigma}_k(d,r)=\e^{-\real(z_k)t}\sigma_{\min}\Big(A - z_k \Id\Big)},
\end{equation}
{with $z_k$ depending on $d$ and $r$, and $(\tau)_+ = \max\{ \tau,0 \}$}. In this way, values of $\tilde{\sigma}_k(d,r)$ which are larger than $\eps$ do not contribute to the functional.
The goal is to compute a solution $(d,r)$ to
\begin{eqnarray*}
	\min\limits_{d,r} && \mathcal{F}(d,r)
	\\
	{\rm s. \ t.} && \min_{k=1, \ldots, M} \tilde{\sigma}_k(d,r) = \eps
\end{eqnarray*}
where - for numerical convenience - the constraint may be
treated as a penalization term.
This can be done by computing the gradient of $\mathcal{F}$, where we use Lemma \ref{lemma:dsv}.
The gradient is continuous and has the form
\begin{equation}{
G(d,r) = \sum_{k=1}^{{M}}  (\eps-\tilde{\sigma}_k(d,r))_{+}\, \real(\iu u_k^* v_k)
\left( \begin{array}{cc} \displaystyle{\frac{x_k  r}{2 (z^R-d)^{3/2}}} \\[3mm] \displaystyle{\frac{x_k}{\sqrt{z^R-d}}} \end{array} \right),}
\label{eq:gradp}
\end{equation}
with {$x_k$ such that $z(x_k)=z_k$. Note that $\real(z_k)$ does not depend on $d$ and $r$, see \eqref{eq:innerp}}.
We add to the functional the penalization term
\begin{equation*}
P(d,r) = \frac12 \left( \min_{k=1,\ldots,M} \tilde{\sigma}_k(d,r) - \eps \right)^2,
\end{equation*}
whose gradient is obtained in a straightforward way. Then to compute a solution we can apply any gradient based method for unconstrained optimization to the functional
\[
\mathcal{F}(d,r)  + c P(d,r),
\]
for a sufficiently large $c$ (in the context of a penalization methodology). The method turns out to be effective but appears to be computationally expensive due to the fact that at every step of the gradient descent method we have to compute several singular values and the associated singular vectors.
%
\subsection{Selecting points internal to the weighted $\eps$-pseudospectrum}
{A cheaper (and preferred) alternative to the previous method, is obtained by treating the points $\{ z_k \}$ singularly, one after the other. In this way, we start by considering the first point $z_1 = \phi_1 + \iu \psi_1$ and compute}
\begin{equation*}
\left( A - z_1 \Id \right) = U_1 \Sigma_1 V_1,
\end{equation*}
setting
\begin{equation*}
\tilde{\sigma}_1 = \e^{-\real{(z_1)}t}\min {\rm diag} \left(\Sigma_1 \right),
\end{equation*}
according to the definition in \eqref{eq:sigk}. Then we check the difference $\delta \eps = \eps - \tilde{\sigma}_1$. If $\delta\epsilon\le 0$ we proceed by considering $z_2$ and repeating the same steps, otherwise it means that we are inside the  weighted $\eps$-pseudospectrum and therefore we need to update the internal parabola. If we do not find any $z_k$ such that $\tilde{\sigma}_k < \eps$, we may consider the point $z \in \{ z_k \}_{k=1}^{M}$ for which the corresponding $\tilde{\sigma}$ is minimal and proceed in order to find a closer parabola to the weighted $\eps$-pseudospectral level set.

The algorithm we adopt tunes the interpolation point $w=d+\iu r$, so that the updated curve is external weighted $\eps$-pseudospectrum of $A$ at $z$. We indicate by $p = p(d,r)$ a selected point {$z_k=z(x_k)=\phi_k+\opi \psi_k$ laying in the wrong {{weighted}} pseudospectral level set} and write $\tilde{\sigma}_k$ as
\begin{equation}
	\tilde{\sigma}(d,r)=\e^{-\real(p(d,r))t}\sigma_{\min}\Big(A-p(d,r)\Id\Big).
\end{equation}
In principle we want to solve the equation
\begin{equation}
\tilde{\sigma}(d,r) - \epsilon = 0 \qquad \mbox{w.r.t.} \ r.
\label{eq:sigmard}
\end{equation}
It seems natural to fix the parameter $d$ as the mean of the abscissas of the support points $z_k$, i.e.
\begin{equation*}
d = \frac{1}{M} \sum\limits_{k=1}^{M} \phi_k
\end{equation*}
and solve the scalar equation
\begin{equation}
\tilde{\sigma}(d, r) - \epsilon = 0 \qquad \mbox{w.r.t.} \ r.
\label{eq:sigmar}
\end{equation}
Applying Lemma \ref{lemma:dsv} to $\tilde{\sigma}(d,r)$- with $u$ and $v$ left and right associated singular vectors - we get
\begin{equation*}
\frac{d}{d r} \tilde{\sigma} \left(d,r\right) = -\e^{-\real(p(d,r))t}\real(\iu u^* v)\,g
\end{equation*}
with
\[
g = \frac{x_k }{\sqrt{z^R-d}}.
\]
In order to accurately compute $r$ such that $\tilde{\sigma}(d,r)= \eps$ we make a few (say $m$) Newton iterations
\begin{equation}\label{eqNI}
r^{\ell + 1} = r^{\ell} + \frac{\e^{-\real(p(d,r^\ell))t} \sigma_{\min}\left( A-p(d,r^\ell)\Id \right)-\eps}{\e^{-\real(p(d,r^\ell))t}\real\Bigl(\iu (u^{\ell})^* v^{\ell}\Bigr)\,g}, \quad \ell=1,\ldots,m-1
\end{equation}
with $u^{\ell}$ and $v^{\ell}$ singular vectors associated to $\sigma_{\min}(A- p(d, r^{\ell}) \Id)$ and $r^{\ell}$ the actual ordinate of the interpolation point $w$.

Then we compute a new parabola, which interpolates $ d + \iu r^m$, reparametrize it and compute a new set of points. Iterating a few times this procedure we compute the desired parabolic profile.
\begin{remark}
Since we have $2$ free real parameters to determine, $d$ and $r$, we may consider at the same time two points $z_1$ and $z_2$ to which correspond values of $\tilde{\sigma}(d,r)$ smaller than the
target value $\eps$.
This would provide a simple variant to the method described above.
We would first determine two points $z_1$ and $z_2$ such that $\tilde{\sigma}_j < \eps$, for $j=1,2$ and then solve equations  (in analogy to \eqref{eq:sigmar})
\begin{eqnarray*}
&& \tilde{\sigma}_1(d,r) - \epsilon = 0 \\[2mm]
&& \tilde{\sigma}_2(d,r) - \epsilon = 0,
\end{eqnarray*}
with respect to $r$ and $d$
by Newton method. It would be natural to expect that this method would result into a fewer number of iterations.
\end{remark}
\section{A case study: the 1D Black and Scholes equation}
\label{sec:BS}
The well known (deterministic) Black-Scholes equation \cite{BS} has the following form:
\begin{equation}\label{eq:BS}
	\frac{\partial u}{\partial \tau}=\frac{1}{2}\sigma^2s^2\frac{\partial^2u}{\partial s^2}+rs\frac{\partial u}{\partial s}-ru,\;\;s>L,\;\;0<\tau\le t,
\end{equation}
for $L$, $t$ given, where the unknown function $u(s,\tau)$ stands for the fair price of the option when the corresponding asset price at time $t-\tau$ is $s$ and $t$ is the maturity time of the option. Moreover, $r\ge0$, $\sigma>0$ are given constants (representing the interest rate and the volatility, respectively). In practice, for the sake of numerical approximation, we consider a bounded spatial domain, setting
\begin{equation*}
	L<s<S
\end{equation*}
for a sufficiently large $S$. We take \eqref{eq:BS} together with the following conditions, typical for the European call option, cf.~\cite{ITHW}:
\begin{equation}\label{ibcBS}
	\begin{array}{l}
		u(s,0)=\max(0,s-K),\\
		u(L,\tau)=0,\;0\le\tau\le t,\\
		u(S,\tau)=S-\e^{-r\tau}K,\;\;0\le \tau \le t,
	\end{array}
\end{equation}
being $K$ the reference strike price.
In this Section we extend the theory developed in \cite{RT} for the 1D convection-diffusion operator
\[
\mathcal{L}u = u_{xx}+u_x.
\]
to equation \eqref{eq:BS}. In this way we are able to theoretically determine a region in the complex plane where the norm of the resolvent of the Black-Scholes differential operator grows exponentially. This knowledge allows us to use \eqref{eq:BS} as a benchmark problem to test the new pseudospectral roaming strategy in Section~\ref{sec:pseudospec}.
Our goal is to solve \eqref{eq:BS} with \eqref{ibcBS} by applying the Laplace transform method. To do this we first transform the problem to an equivalent one with homogeneous boundary conditions. This is easily achieved by considering
\[
v(s,\tau)=u(s,\tau) - y(s,\tau),
\]
with
\[
y(s,\tau)= \frac{s-L}{S-L}\left( S - \e^{-r\tau}K \right).
\]
The differential equation for $v$ reads
\begin{equation}\label{eq:v}
	\frac{\partial v}{\partial \tau}=\frac{1}{2}\sigma^2s^2\frac{\partial^2v}{\partial s^2}+rs\frac{\partial v}{\partial s}-rv - {\frac{s}{S-L}}r\e^{-r\tau} K+{\frac{Lr}{S-L}S},\qquad s>L,\;\;0<\tau\le t,
\end{equation}
with initial and boundary data
\begin{equation}\label{ibcv}
	\begin{array}{l}
		v(s,0) = \max(0,s-K)-\frac{s-L}{S-L}\left( S - K \right) =: v_0(s),\\
		v(L,\tau)=0,\;0\le\tau\le t,\\
		v(S,\tau)=0,\;\;0\le \tau \le t,
	\end{array}
\end{equation}
We can now apply the Laplace transform to both sides of \eqref{eq:v} with respect to $\tau$. This leads to the following equation for $V(s,z)$, the Laplace transform of $v(s,\tau)$:
\[
V(s,z) = (zI-\mathcal{L})^{-1} \left( v_0(s) - \frac{srK}{S-L}\frac{1}{z+r}+\frac{Lr}{S-L}\frac{1}{z} \right),
\]
with $\mathcal{L}$ the differential operator for the Black-Scholes problem with homogeneous boundary conditions.

\subsection{Pseudospectra of the Black-Scholes equation}
For our analysis we set $L=1$, which is reasonable if $S, K>>1$ in \eqref{ibcBS} and allows to apply the change of coordinates $x=\log(s)$ while keeping the domain bounded. After this change of variable we obtain the evolution problem $u_t=\mathcal{L}u$, with
\begin{equation}\label{defL}
	\mathcal{L}u=\frac{1}{2}\sigma^2 u_{xx}+ \left(r-\frac{1}{2}\sigma^2 \right)u_x-ru,\qquad 0\le x\le\log(S),
\end{equation}
a second order diffusion-convection-reaction differential operator with constant coefficients on a bounded domain with homogeneous boundary conditions of Dirichlet type. We thus can compute explicitly the eigenvalues and eigenfunctions of $\mathcal{L}$ by applying it to a mapping of the form $\varphi(x)=\e^{\alpha x}$. In this way, we obtain
\begin{equation}
	\mathcal{L}\varphi =(\nu\alpha^2+(r-\nu)\alpha-r)\e^{\alpha x}=\lambda \varphi,\qquad \text{with}\;\;\nu=\frac{1}{2}\sigma^2
	\label{eq5}
\end{equation}
and
\begin{equation}\label{eigL}
	\lambda = \nu\alpha^2+(r-\nu)\alpha-r.
\end{equation}
Then, for each $\lambda$ real we have two associated values of $\alpha$, namely
\begin{equation}
	\alpha_{\pm}=\frac{-(r-\nu)\pm\sqrt{(r+\nu)^2+4\lambda\nu}}{2\nu}.
	\label{eq3}
\end{equation}

For any $\lambda$ and corresponding $\alpha_+$ and $\alpha_-$, the function
\begin{equation}
	\phi(x)=\frac{\e^{\alpha_+x}-\e^{\alpha_-x}}{\alpha_+-\alpha_-}
	\label{eq6}
\end{equation}
satisfies (\ref{eq5}) in the interior of $[0,\;\log(S)]$ and the boundary condition at $x=0$. It also satisfies the boundary condition at $x=\log (S)$ provided $\e^{\alpha_+ \log (S)}=\e^{\alpha_- \log (S)}$, that is, $(\alpha_+-\alpha_-)\log(S)=2\pi n \opi$ for some nonzero $n\in\mathbb{Z}$. By (\ref{eq3}), this amounts to the condition $(\log(S)/\nu)\sqrt{(r+\nu)^2+4\lambda\nu}=2\pi n \opi $, and upon squaring we obtain the following eigenvalues
\begin{equation}
	\lambda_n=-\left(\frac{r+\nu}{2}\right)^2\frac{1}{\nu}-\frac{\pi^2n^2\nu}{\log(S)^2},\;\;\;\;n=1,2,3,...
	\label{eq4}
\end{equation}
Thus $\Lambda(\mathcal{L})$ is a discrete set of negative real numbers in the interval $(-\infty,-\frac{1}{4})$.
Note that, for our problem, there are choices of $\lambda$ for which both $\alpha_+$ and $\alpha_-$ lie in the left half-plane and thus both $\e^{\alpha_+x}$ and $\e^{\alpha_-x}$ are decreasing functions. For the eigenfunctions associated to (\ref{eq4}), this occurs with $\real(\alpha_+)=\real(\alpha_-)=-\frac{(r-\nu)}{2\nu}$ under the assumption $r>\nu$. More generally, it occurs if and only if $\alpha$ belongs to the strip $B=\{\alpha\in\mathbb{C}:\;-\frac{(r-\nu)}{\nu}\le\real(\alpha)\le0\}$, since if $\alpha$ is one solution of (\ref{eq5}), the other is $-\frac{(r-\nu)}{\nu}-\alpha$. The corresponding region in the $\lambda$-plane is the image of $B$ under the function $\lambda=\nu\alpha^2+(r-\nu)\alpha-r$, which we denote by $\Pi$:
\begin{equation}
	\Pi=\{\lambda\in\mathbb{C}:\;\lambda=\nu\alpha^2+(r-\nu)\alpha-r,\;-\frac{(r-\nu)}{\nu}\le\text{Re}(\alpha)\le0\}.
\end{equation}
The  ``critical parabola'' that bounds $\Pi$ is the image of the boundary of $B$ under the same function, which we can simply represent by
\begin{equation}
	P=\{\lambda\in\mathbb{C}:\;\lambda=\nu\alpha^2+(r-\nu)\alpha-r,\;\real(\alpha)=0\}
\end{equation}
since $\text{Re}(\alpha)=-\frac{(r-\nu)}{\nu}$ maps onto the same parabola as $\text{Re}(\alpha)=0$.
Suppose now that $\lambda$ is any complex number in the interior of $\Pi$ so that $\real (\alpha_+)<0$ and $\real(\alpha_-)<0$. Then $\phi(x)$ decreases exponentially with $x$, so if $\log(S)$ is reasonably large, the boundary condition $u(\log(S))=0$ is \textit{nearly} satisfied, with an error of order $\e^{\mu \log(S)}=S^{\mu}$, where $\mu=\max\{\real(\alpha_+),\real(\alpha_-)\}$. Thus $\phi(x)$ is \textit{nearly an eigenfunction} of $\mathcal{L}$, though $\lambda$ may be far from any of the exact eigenvalues.
Then we can just repeat the arguments and passages of \cite{RT} to get their results. The main difference is that, in our case,  we also consider a reaction term, anyway results of \cite{RT} for the convection-diffusion operator can be extended with their same procedure to the convection-diffusion-reaction operator.
We now state our version of Theorem 5 of \cite{RT} which deal with the Black-Scholes differential operator.
\begin{theorem}
	Let $\lambda$ be an arbitrary point in the interior of $\Pi$, with $\alpha_{\pm}$ and $\phi(x)$ defined by (\ref{eq3}) and (\ref{eq6}), and write $\alpha_+-\alpha_-=(1/\nu)\sqrt{(r+\nu)^2+4\lambda\nu}=\beta+\opi\tau$. Then
	\begin{equation}
		\|(\lambda I-\mathcal{L})^{-1}\|\sim\frac{\|\phi\|^2}{\phi(\log (S))}
	\end{equation}
	If in addition $\lambda\not\in(-\infty,\;-\left(\frac{r+\nu}{2}\right)^2\frac{1}{\nu}]$, then $\phi(\log (S))\sim S^{\mu}/|\alpha_+-\alpha_-|$ and therefore
	\begin{equation}
		\|(\lambda I-\mathcal{L})^{-1}\|\sim{S^{-\mu}(\beta^2+\tau^2)^{1/2}\|\phi\|^2},
		\label{eq7}
	\end{equation}
	where $\mu=\max\{\real(\alpha_+),\real(\alpha_-)\}<0$.
\end{theorem}
This result tells us that the resolvent norm changes exponentially along any vertical line inside $\Pi$. Indeed we know, by construction, that the critical parabola is the curve such that $\real(\alpha)=0$ and therefore $S^{-\mu}=1$, while on the real axis, for $\real(\lambda)$ sufficiently small the real part of $\alpha_-$ and $\alpha_+$ are the same and equal to $-\frac{r-\nu}{2\nu}$, that corresponds to the case where $S^{-\mu}$ is maximized.
\subsection{Symmetrizability and a further estimate}
As done in \cite{RT} we can explicitly symmetrize the BS operator $\mathcal{L}$.
First let's define $\rho=\frac{r-\nu}{2\nu}$ and $u(x)=\e^{-\rho x}v(x)$, which implies
\begin{align*}
	&u'=\e^{-\rho x}\left(-\rho v+v'\right),\\
	&u''=\e^{-\rho x}\left(\rho^2v-2\rho v'+v''\right),
\end{align*}
and therefore
\begin{align*}
	\mathcal{L}u=&\nu u''+(r-\nu)u'-ru=\\
	=&\e^{-\rho x}\left[(\rho^2v-2\rho v'+v'')\nu+(-\rho v+v')(r-\nu)-rv\right]=\\
	=&\e^{-\rho x}\left[\rho^2\nu v-(r-\nu)v'+\nu v''-\rho(r-\nu)v+(r-\nu)v'-rv\right]=\\
	=&\e^{-\rho x}\left[v''\nu+v(\rho^2\nu-\rho(r-\nu)-r)\right]=\\
	=&\e^{-\rho x}\left[\nu v''-\left(\frac{(r-\nu)^2}{4\nu}-r\right)v\right].
\end{align*}
Thus if we define $\mathcal{K}v=\nu v''-\left(\frac{(r-\nu)^2}{4\nu}-r\right)v$, $\mathcal{M}v=\e^{-\rho x}v(x)$, then we have
\begin{equation}
	\mathcal{L}=\mathcal{MKM}^{-1}.
\end{equation}
Here $\mathcal{K}$ is a self-adjoint operator and $\mathcal{M}$ is a diagonal operator with $\|\mathcal{M}\|=1$,
$\|\mathcal{M}^{-1}\|=\e^{\rho\log(S)}$, and consequently
\begin{equation}
	\kappa(\mathcal{M})=\|\mathcal{M}\|\|\mathcal{M}^{-1}\|=S^{\frac{r-\nu}{2\nu}}.
\end{equation}
From here, we follow the analysis in \cite{RT} and derive the following bound for the resolvent norm of Black-Scholes operator.
\begin{theorem}
	For any $d>0$, $r>\nu$ and $\lambda\in\mathbb{C}$,
	\begin{equation}
		\|(\lambda I-\mathcal{L})^{-1}\|\le\frac{S^{\frac{r-\nu}{2\nu}}}{\mathrm{dist}(\lambda,P)}\le\frac{S^{\frac{r-\nu}{2\nu}}}{|\imag( \lambda)|}.
	\end{equation}
	\label{teo4.2}
\end{theorem}
The discussion done in the previous subsection this theorem suggested that the pseudospectra level curves of $\mathcal{L}$ are bounded approximately by parabolas. On the light of Theorem \ref{teo4.2} we can say that the exponential bound by parabola does not hold as $|\lambda|\rightarrow0$ but, for any fixed $\epsilon$ and $S$, $\Lambda_{\epsilon}(\mathcal{L})$ the $\epsilon$-pseudospectra is contained in a strip of finite (though typically large) width:
\begin{equation}
	\Lambda_{\epsilon}(\mathcal{L})\subset\{\lambda\in\mathbb{C}:\;|\imag(\lambda)|\le\eps\; S^{\frac{r-\nu}{2\nu}}\}.
\end{equation}
\subsection{Numerical validation}
We have to keep in mind that the results previously exposed hold for the continuous operator.
Since our aim is to have a numerical validation we have to deal with the discrete version
$\mathcal{L}_h$. The way properties of the pseudospectrum of the discrete operator converge
to the properties of the continuous one is not treated here and, to our knowledge, is an open
research problem.
Nevertheless it is reasonable to expect for small $h$ a behaviour close to the one exhibited
by the differential operator. Indeed this is what we observe in Figure \ref{fig4.1} where we
set $S=200$ and used $2000$ points for the space discretization; the resolvent norm is plotted
in a subset of the $\Pi$ region.
The comparison between the magnitude of the resolvent norm estimate (\ref{eq7}) and the resolvent
norm of the discrete operator indicates a very similar behaviour of the two operators.
We clearly see that both decrease when approaching the critical parabola and are maximal close to
the real axis. 
\begin{figure}
	\centering{
		\subfigure{
			\includegraphics[width=0.46\textwidth]{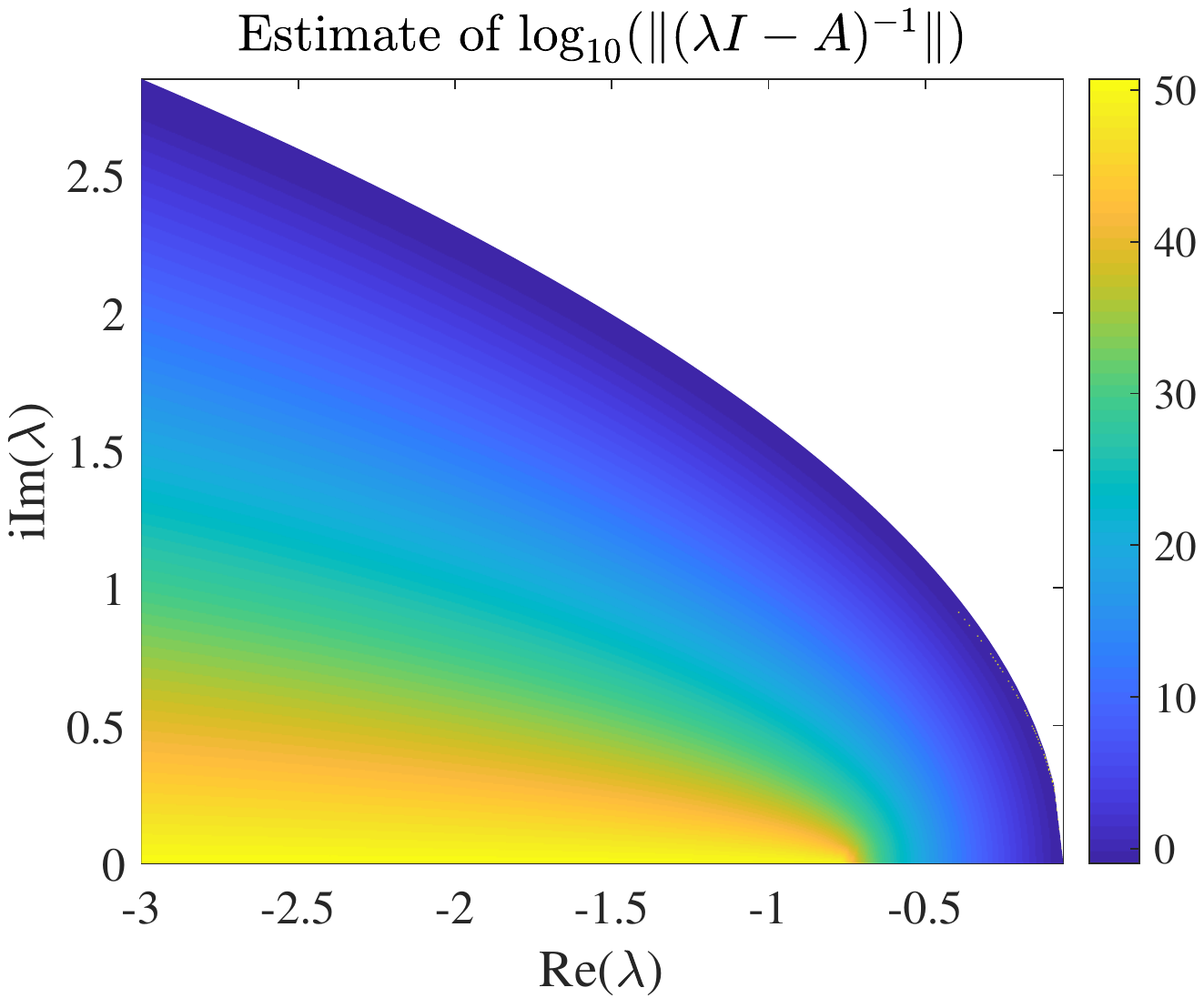}}
		\subfigure{
			\includegraphics[width=0.46\textwidth]{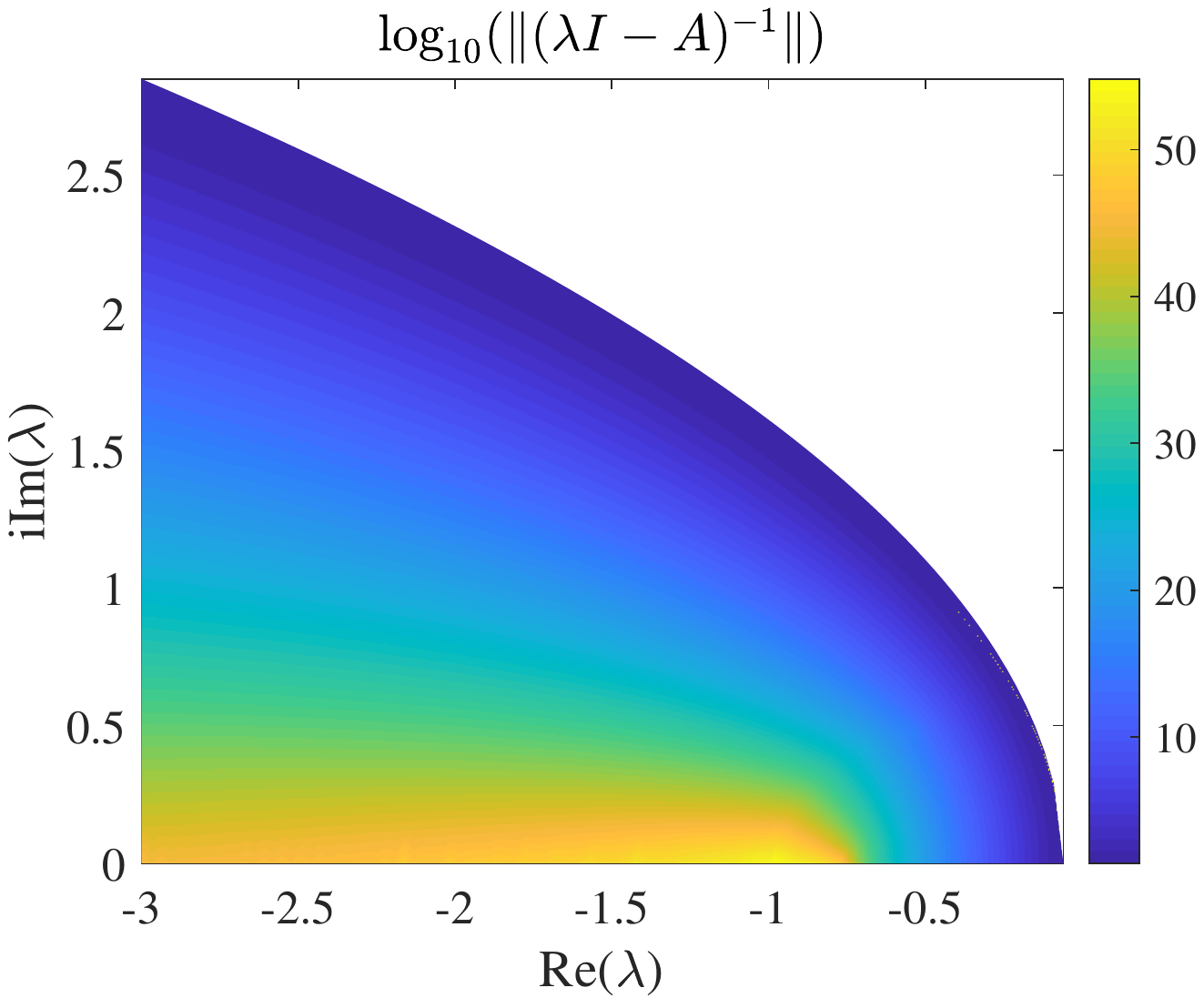}}
	}	
	\caption{Magnitude of the estimate of the resolvent norm (\ref{eq7}) (left) and magnitude of the computed resolvent norm (right).}
	\label{fig4.1}
\end{figure}

\section{Choice of the parameters}\label{sec:param}
In \cite{GLN20} error estimates are developed to provide a practical strategy to
optimize the elliptical integration contour and minimize the required number $N$
of quadrature nodes, for a prescribed target accuracy.
However, the main results in \cite{GLN20}, namely Theorems 1 and 2, do not depend
on the specific choice of an ellipse
and apply in a straight forward way to the parabolic and hyperbolic contours we
have described in Section~\ref{sec:contour}.
The steps to determine the integration contour are thus common for the three types of
contours under study and are the following:
\begin{enumerate}
\item Compute $z^L$ from $\e^{z^L t} = \varepsilon$, with $\varepsilon$ the working precision.
\item Compute the critical curve $\Gamma_{left}$ according to Section~\ref{sec:pseudospec}.
This procedure provides an explicit parametrization of $\Gamma_{left}$, of the form $\psi(x)$, $x\in \bR$.
\item Compute  $c_{\max}$ as the unique value that satisfies ${\real(\Gamma(c_{\max} \pi) )= z^L}$.
This gives, according to the discussion in Section~\ref{sec:contour}:
    \begin{equation}\label{cmax}
    c_{\max}(a) = \left\{
    \begin{array}{l}
    \displaystyle \frac{1}{2}, \qquad \mbox{ for the ellipse},\\[1em]
    \displaystyle {\frac{1}{\pi}\sqrt{z^R-z^L+a^2+\frac{ar}{\sqrt{z^R-d}}}},\qquad \mbox{ for the parabola}, \\[1em]
    \displaystyle  {\frac{1}{\pi} \log \left(b+\sqrt{b^2-1 } \right)},\;\text{with}\;b=\frac{(z^C-z^L)\sin(a_1+a)}{(z^R-z^L)\sin(a_1)},\;
		\mbox{ for the hyperbola}.
    \end{array}
    \right.
    \end{equation}
\item Compute $a$ by following the same steps as in \cite[Section 3.2]{GLN20}.
This requires the identification of the right-most point of the external ellipse/parabola/hyperbola, this is
    \begin{equation}\label{D}
    D(a) = \left\{
    \begin{array}{l}
    \displaystyle z(-\opi a) = z^L+ \cosh(2a)(z^R-z^L)+\sinh(2a) \frac{r}{\sin(\theta)}, \qquad \mbox{ for the ellipse},\\[1em]
    \displaystyle z(-\opi a) = a^2-2 a_1(a)a +a_2(a),\qquad \mbox{ for the parabola}, \\[1em]
    \displaystyle  z(\opi a) = z^C-a_2(a)\sin(a_1(a) -a),\qquad \mbox{ for the hyperbola},
    \end{array}
    \right.
    \end{equation}
where $r$ in the first expression above is the imaginary part of the control point $d+\opi r$ from Section~\ref{sec:contour} and
$\theta=\arccos\left(\frac{d-z^L}{z^R-z^L}\right)$, as in \eqref{a1ell} and \eqref{a2ell}.
In order to have a more robust estimate we take into account the constant $M_{left}$, named $M_{+}$ in \cite[equation (23)]{GLN20}, which is defined as
\begin{equation}
	M_{left}=\frac{1}{2\pi}\max_{z\in\tilde{\Gamma}_{left}}\Big\|\e^{zt}(zI-A)^{-1}\left(u_0+\hat{b}(z)\right)z'\Big\|,
\end{equation}
where $\tilde{\Gamma}_{left}$ is a suitable restriction of $\Gamma_{left}$.

We also consider a different estimate for $M_{right}$, named $M_{-}$ in \cite[equation (24)]{GLN20}, which is defined as
\begin{equation}
	M_{right}=\frac{1}{2\pi}\max_{z\in\tilde{\Gamma}_{right}}\Big\|\e^{zt}(zI-A)^{-1}\left(u_0+\hat{b}(z)\right)z'\Big\|,
\end{equation}	
where $\tilde{\Gamma}_{right}$ is a suitable restriction of $\Gamma_{right}$. Note that $M_{left}$ can be bounded as
\begin{equation}	M_{left}\le\frac{1}{2\pi}\max_{z\in\tilde{\Gamma}_{left}}\e^{\real(z)t}\Big\|(zI-A)^{-1}\Big\|\Big\|\left(u_0+\hat{b}(z)\right)\Big\||z'|=\tilde{M}_{left},
\end{equation}
and we can take advantage of the computation done in step $2$ for the weighted $\eps$-pseudospectrum to approximate $\tilde{M}_{left}$.
Concerning $M_{right}$ we first assume that the maximum is reached on $z_v(a)$, the vertex of $\tilde{\Gamma}_{right}$; in this way we have
\begin{equation}\label{eq:Mright}
M_{right}\le\frac{1}{2\pi}\e^{D(a)t}\Big\|(z_v(a)I-A)^{-1}\Big\|\left\|u_0+\hat{b}\left(z_v(a)\right)\right\||z'_v(a)|=\tilde{M}_{right}.
\end{equation}
Therefore repeating the calculations of \cite[Section $3.2$]{GLN20} and including $M_{left}$ and $M_{right}$ we have,
for a fixed target accuracy $\tol$,
\begin{align}\label{eq:minimiza}
N=&\frac{c}{a}\Big(\log\left(2\pi c\tilde{M}_{right}+\pi \tilde{M}_{left}\right)-\log\left(\tol\right)\Big)\nonumber\\
\leq&\frac{c_{\max}(a)}{a}\Big(\log\left(2\pi c_{\max}(a)\tilde{M}_{right}+\pi \tilde{M}_{left} \right)-\log\left(\tol\right)\Big).
\end{align}
\begin{algorithm}[t]
	\caption{\texttt{Numerical algorithm for computing $a$.}}\label{A2}
	\hspace*{\algorithmicindent} \textbf{Input:} $a_{max}$, $\Gamma_{left}$, $\tilde{M}_{left}$, $\tilde{M}_{right}=1$, $a^{1}=a_{max}$, $err=1$ and $j=0$\\
	\hspace*{\algorithmicindent} \textbf{Output:} $a$
	\begin{algorithmic}[1]
		\WHILE {$err\ge prec$}
		\STATE Find $a^{j+1}\in[0, a_{max}]$ such that $a^{j+1}=\arg \min f(a)$, see (\ref{eq:Nbis});
		\STATE $err=\frac{|a^{j+1}-a^{j}|}{a^{j+1}}$;
		\IF {$err\ge prec$;}
		\STATE Compute $\tilde{M}_{right}$ according (\ref{eq:Mright}) and then update $f(a)$;
		\ENDIF
		\STATE $j=j+1$;
		\ENDWHILE
	\end{algorithmic}
\end{algorithm}

The value of $a$ is then computed by minimizing the right hand side of \eqref{eq:minimiza}
which requires an interval of the form $a\in [a_{\min},a_{\max}]$, with prescribed
bounds $a_{\min}$ and $a_{\max}$.
Clearly $a_{\min}$ has to be set equal to $0$ while to determine $a_{\max}$ we need
a further discussion.
The procedure described above does not take into account the numerical error due to
the conditioning of $zI-A$, whose effect is also amplified by the multiplication with
the exponential term. This contribution becomes relevant when one wants to reach high
accuracies or has to deal with ill conditioned systems.
To keep under control this error we estimate it on the vertex of the integration
profile defined by $a=a_{\max}$. If this is higher than the required tolerance we
reduce $a_{\max}$ and we check again until the estimate is below the required
tolerance.
Reducing $a_{\max}$ has the effect of bringing  the vertex profile closer to the imaginary
axis, which reduces the amplification effect of the exponential.
We remand to Subsection \ref{subsec:TW} for the description on how we estimate the
numerical error due to the solution of linear systems.
Two aspects are remarkable.

\renewcommand{\labelenumii}{(\roman{enumii})}
\begin{enumerate}
\item It may be too restrictive to select the first acceptable
$a_{\max}$, since we may exclude some acceptable $a$. Therefore, we select the value computed
in the iteration before convergence and add a penalization term to (\ref{eq:minimiza}). This penalization term is based on the estimation of the numerical error due to the conditioning of $zI-A$ on the
vertex of the integration profile. The new function to minimize becomes
\begin{equation}\label{eq:Nbis}
	f(a)=\frac{c_{\max}(a)}{a}\Big(\log\left(2\pi c_{\max}(a)\tilde{M}_{right}+\pi \tilde{M}_{left} \right)-\log\left(\tol\right)\Big)+wp(a),
\end{equation}
where $w$ is a positive scalar and
\begin{equation}
	p(a)=\begin{cases}0 & err^{num}_N(a)  < \tol\\
		1 &err^{num}_N(a)  \ge \tol\end{cases}
\end{equation}

\item We note that evaluating (\ref{eq:Mright}) for every $a\in[a_{\min},a_{\max}]$ is computationally
expensive due to the presence of the resolvent norm. We thus implement the iteration procedure
described in Algorithm \ref{A2}, which we have observed to convergence in very few iterations
(from $2$ to $6$ depending on which type of contour we use) in all the numerical tests done.

\end{enumerate}

\item Compute a truncation parameter $c\le c_{\max}$. The actual numerical integration is performed for $x\in[-c\pi, c\pi]$. The computation  of $c$ requires an iterative procedure resumed in Algorithm \ref{A1} and follows precisely \cite[Section 3.5]{GLN20}, for the three types of integration contours.

\item Set 
\begin{equation}\label{eq:N}
N= \left \lceil \frac{c}{a}\Big(\log\left(2\pi c\tilde{M}_{right}+\pi\tilde{M}_{left}\right)-\log\left(\tol\right)\Big)\right\rceil.
\end{equation}

\end{enumerate}

\begin{algorithm}[t]
	\caption{\texttt{Numerical algorithm for approximating $c,K$.}}\label{A1}
	\hspace*{\algorithmicindent} \textbf{Input:} $K^{(1)}$ given, $K^{(0)}=K^{(1)}-2prec$, $j=0$\\
	\hspace*{\algorithmicindent} \textbf{Output:} $c$
	\begin{algorithmic}[1]
		\WHILE {$|K^{(j+1)}-K^{(j)}|\ge prec$}
		\STATE Find $c^{(j)}$ such that $\real(z(c^{(j)}\pi))=\frac{1}{t}\log\left(\frac{\tol}{K^{(j)}}\right)$;
		\STATE $K^{(j+1)}=\frac{1}{2\pi}\|\hat{u}(z(c^{(j)}\pi))z^{'}(c^{(j)}\pi))\|$;
		\STATE $j=j+1$;
		\ENDWHILE
	\end{algorithmic}
\end{algorithm}

\section{Comparison of the integration profiles: numerical illustrations} \label{sec:compar}
In this section we apply the considered contour integral methods to two illustrative problems
arising from finance. The first problem is the Black-Scholes equation, while the second is
the Heston equation. Both models are the same as in \cite{GLN20}.

We show the absolute error rather than the relative error in order to check the match with
the target accuracy $\tol$, that is the accuracy we want to reach in the approximation of (\ref{eq:bromwich}).
{Similarly to what has been done in other publications presented in the literature, we compute
the time approximation error for a semidiscretization in space of the PDE. We do not address here specific estimates of the spatial discretization error, but rely on the referred literature for every example.
We measure the error in time against a reference solution computed by using the MATLAB function {\tt expmv}, see \cite[Algorithm 3.2]{AH}.}

We also notice that in all our tests we construct the inner curves as explained in Section
\ref{sec:pseudospec}, taking the weighted $\eps$-pseudospectral level set with $\epsilon=10^{-7}$.

\subsection{Black-Scholes equation}
\label{subsec:BS}
Following the same strategy adopted in \cite{ITHW2,GLN20}, we discretize \eqref{eq:BS} in space
on a uniform space grid of $N_h=2000$ points in $[L,S]$ for $L=0$, $S=200$, by using the classical
centered finite difference scheme. {The error associated to the spatial discretization is then $\mathcal{O}(\Delta x^2)$, with $\Delta x$ the diameter of the spatial grid, being around $0.001$ for our choice of $N_h$. This
can be observed in Figure \ref{fig5.0}, where we consider a reference solution computed with $N_h=2\cdot10^4$ spatial nodes and approximated in time by using {\tt expmv}}.

We set $r=0.06$, $\sigma=0.05$, and $K=80$. We plot the error for a selection of tolerances for
the cases $t=1$ (Figure \ref{fig5.1}) and $t=10$ (Figure \ref{fig5.2}). In Figure \ref{fig5.5}
we show the selected profiles of integration for the tolerance $\tol=5\cdot10^{-5}$ at time $t=1$
and $t=10$. {In the pictures we also highlight, for each profile, the estimated number
of quadrature nodes $N$ according to (\ref{eq:N}), to reach the target accuracy.}
\begin{figure}[t]
	\centering{
		
		\includegraphics[width=0.48\textwidth]{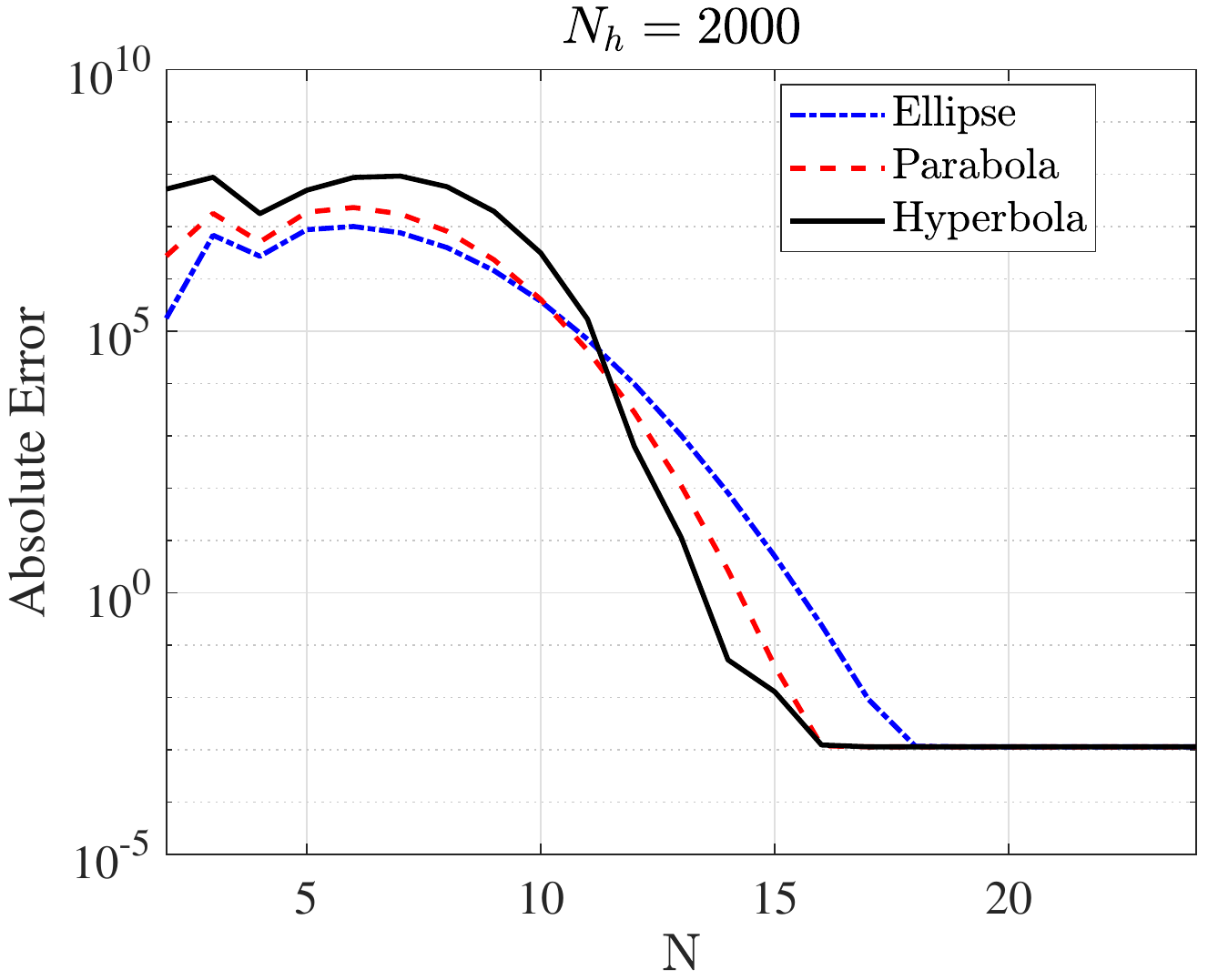}
		
	}	
	
	\caption{Error vs number of nodes for Black-Scholes, $t=1$. The asymptotic value reached corresponds to the error in space when {$N_h=2000$}.}
	\label{fig5.0}
\end{figure}
\begin{figure}[t]
	\centering{
		\subfigure{
			\includegraphics[width=0.48\textwidth]{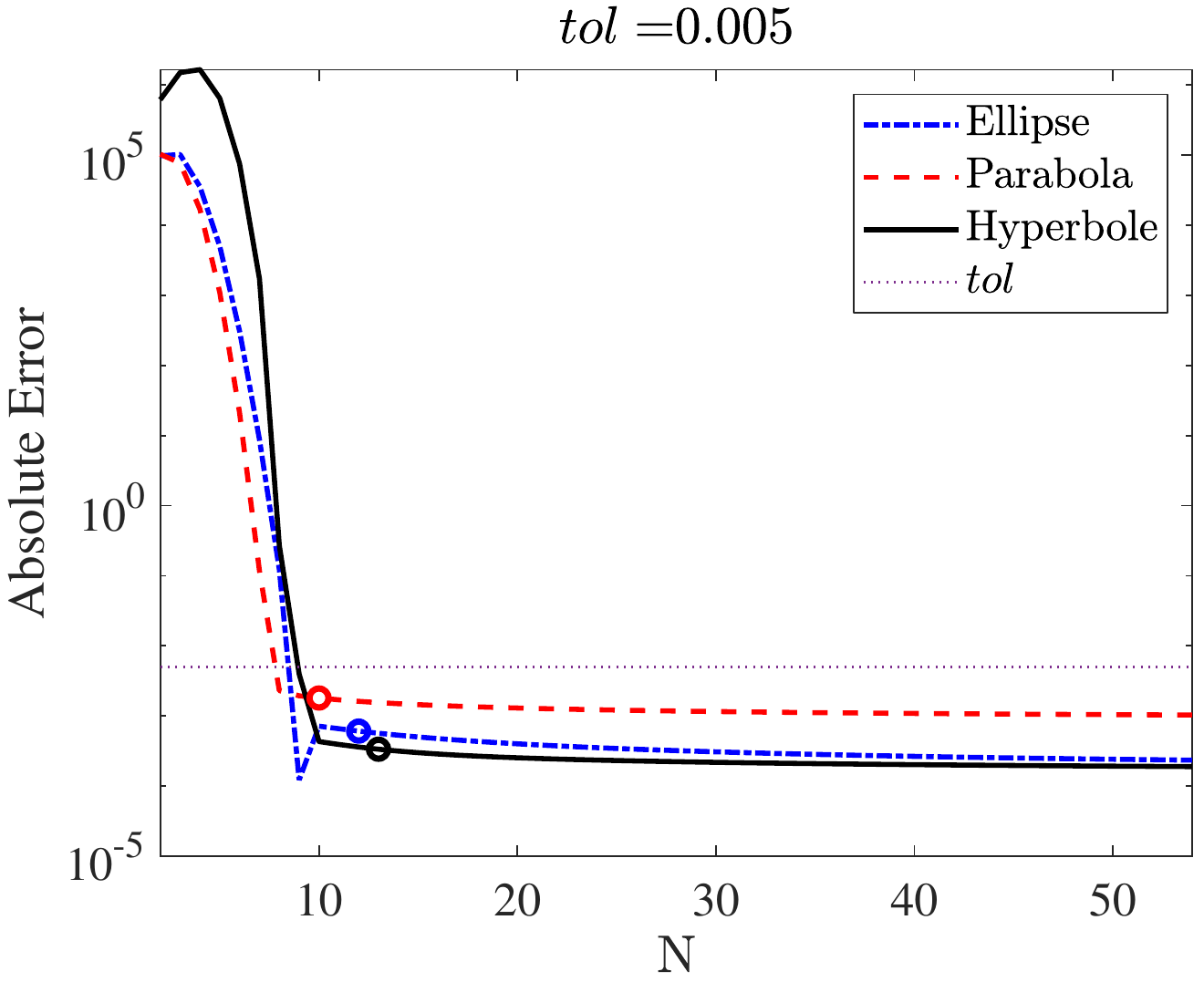}}
		\subfigure{
			\includegraphics[width=0.48\textwidth]{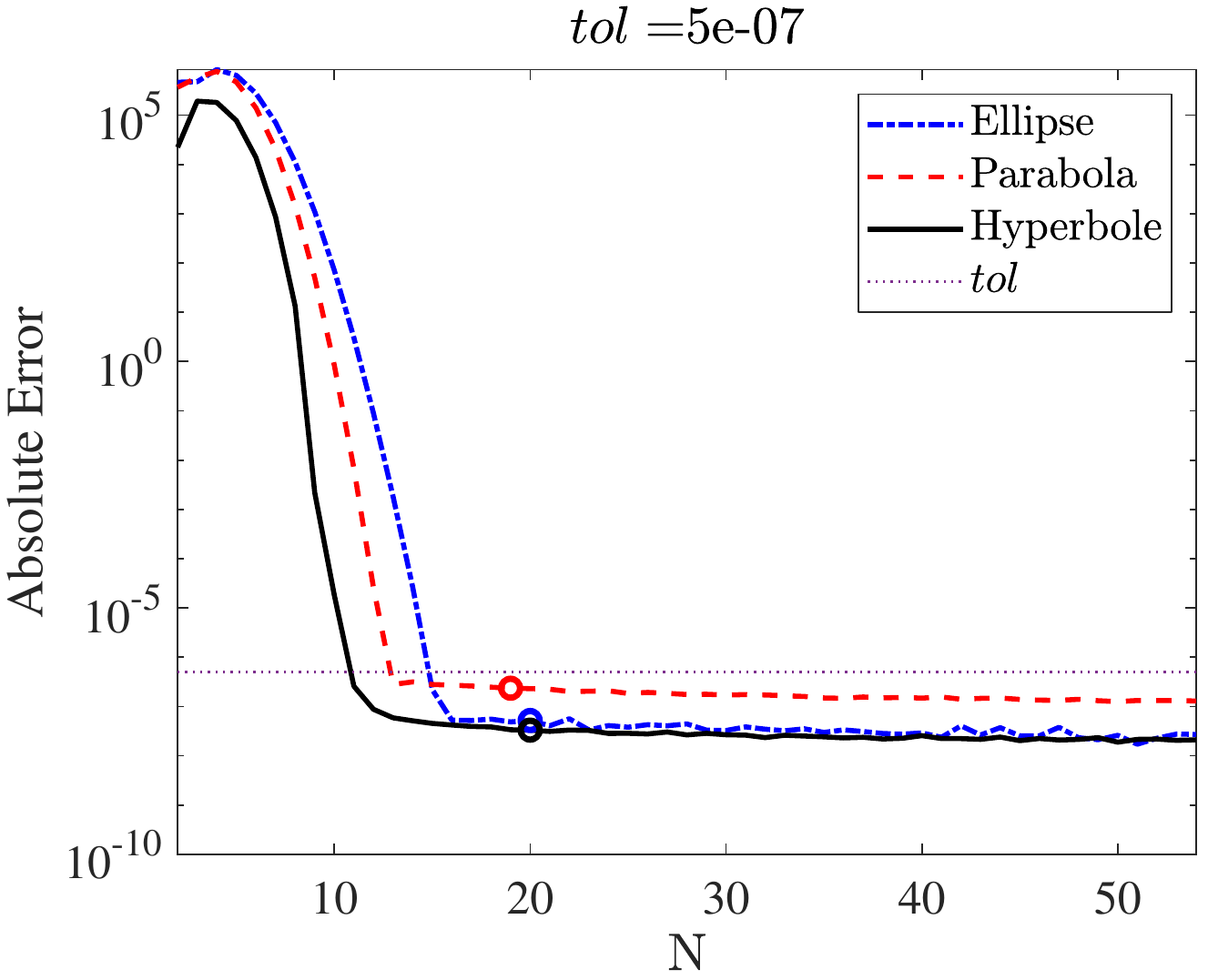}}
	}	
	\centering{
		\subfigure{
			\includegraphics[width=0.48\textwidth]{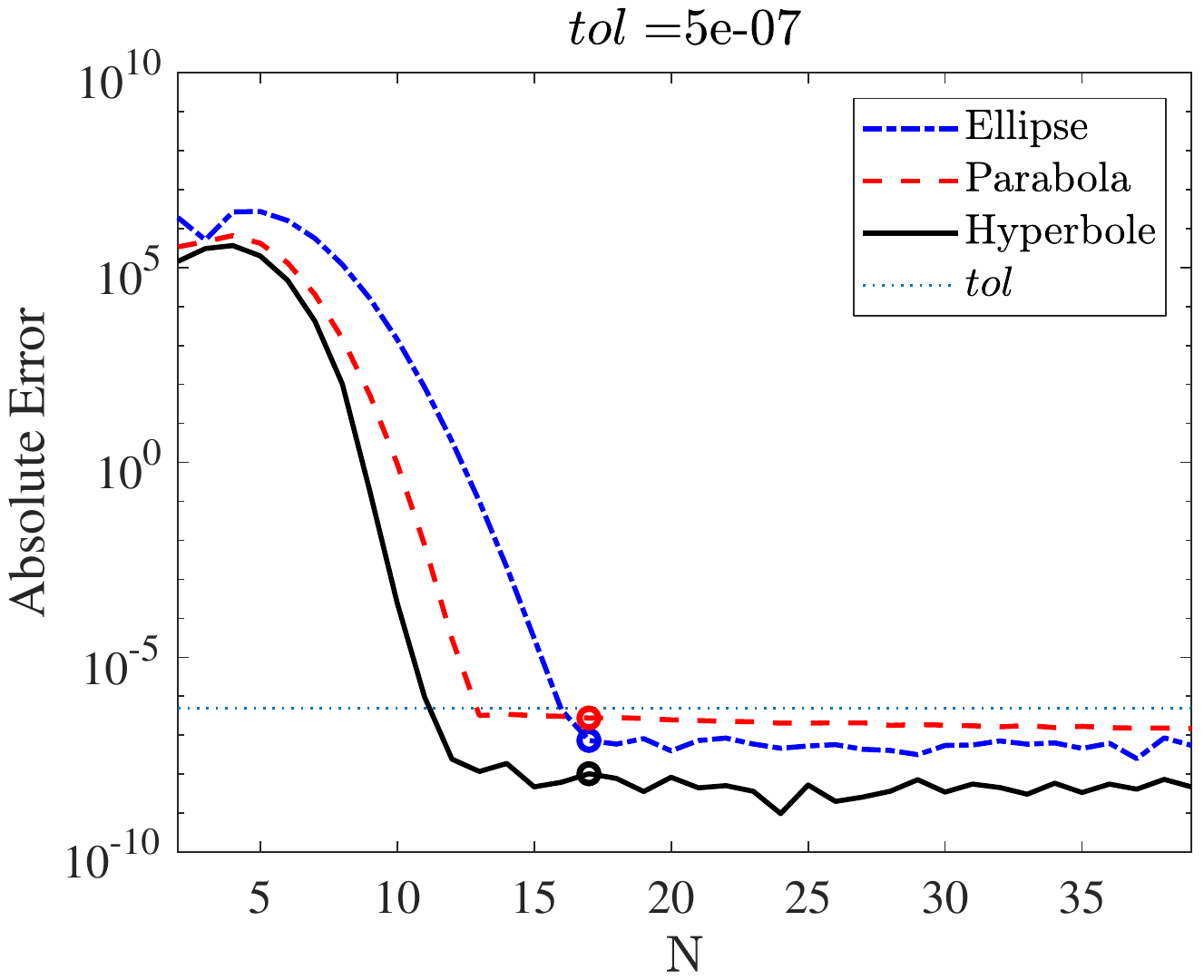}}
		\subfigure{
			\includegraphics[width=0.48\textwidth]{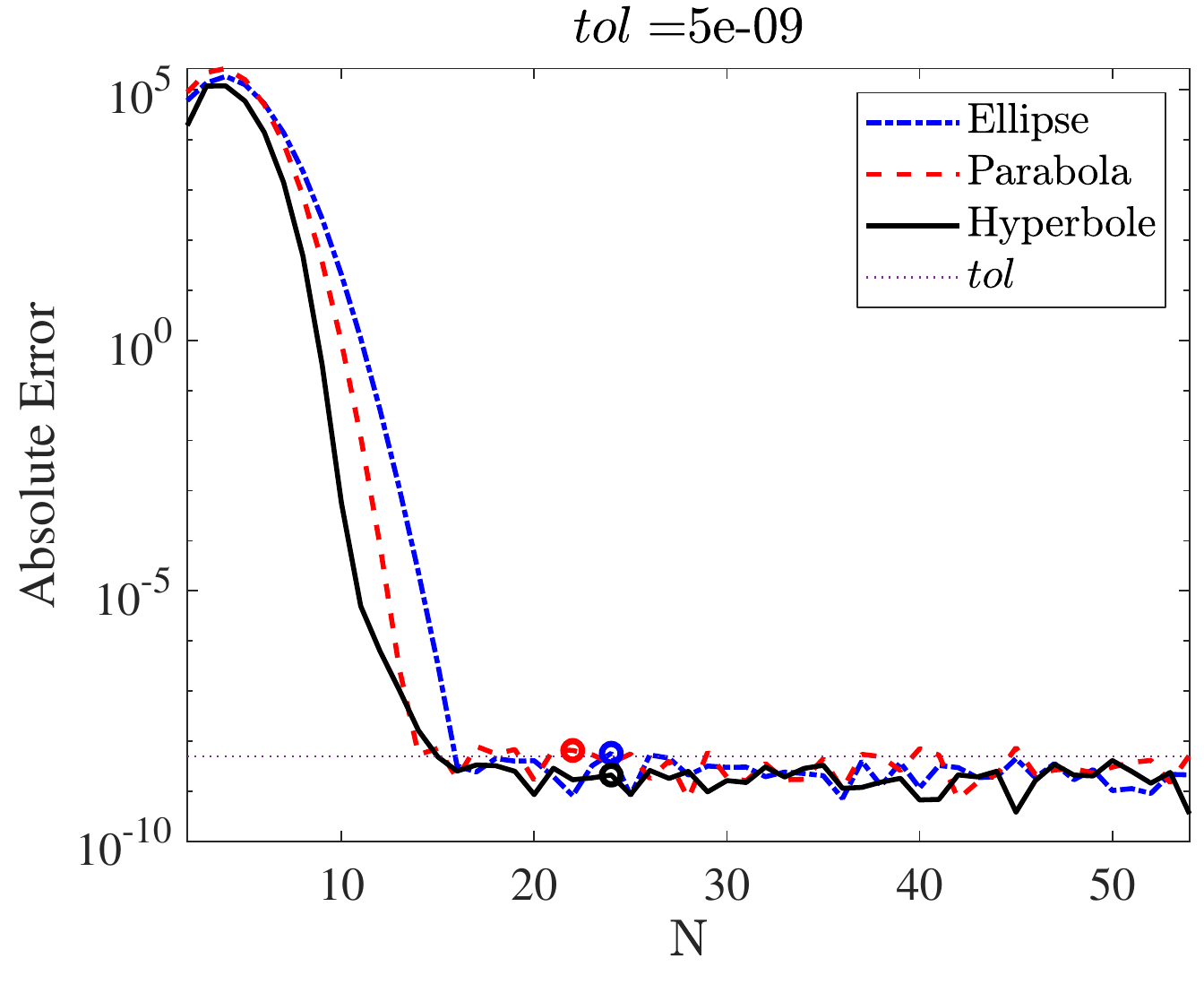}}
	}			
	
	\caption{Error vs number of nodes for Black-Scholes, $t=1$. Comparison for different values of the tolerance among the elliptic, parabolic and hyperbolic contours.}
	\label{fig5.1}
\end{figure}
\begin{figure}[t]
	\centering{
		\subfigure{
			\includegraphics[width=0.48\textwidth]{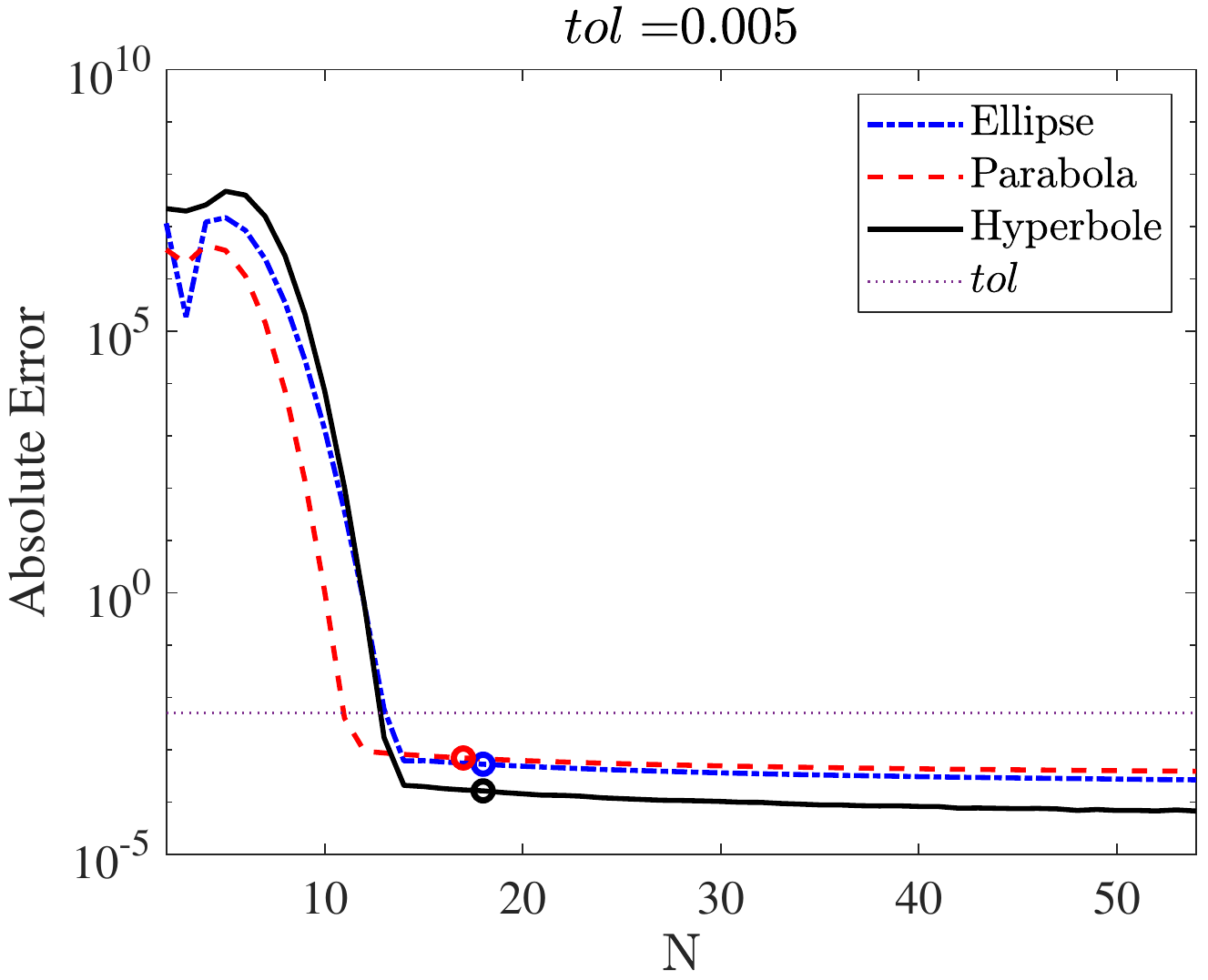}}
		\subfigure{
			\includegraphics[width=0.48\textwidth]{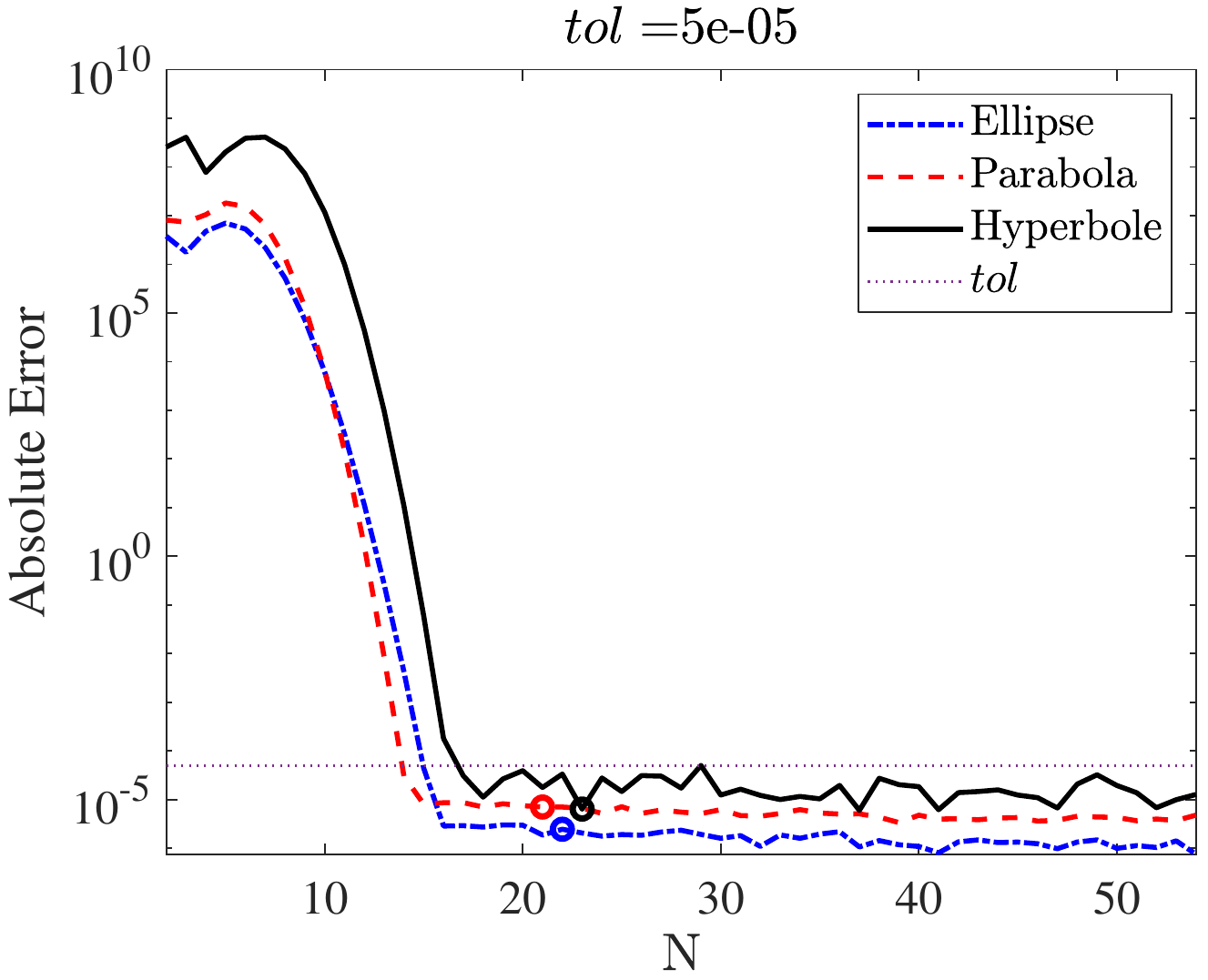}}
	}	
	\centering{
		\subfigure{
			\includegraphics[width=0.48\textwidth]{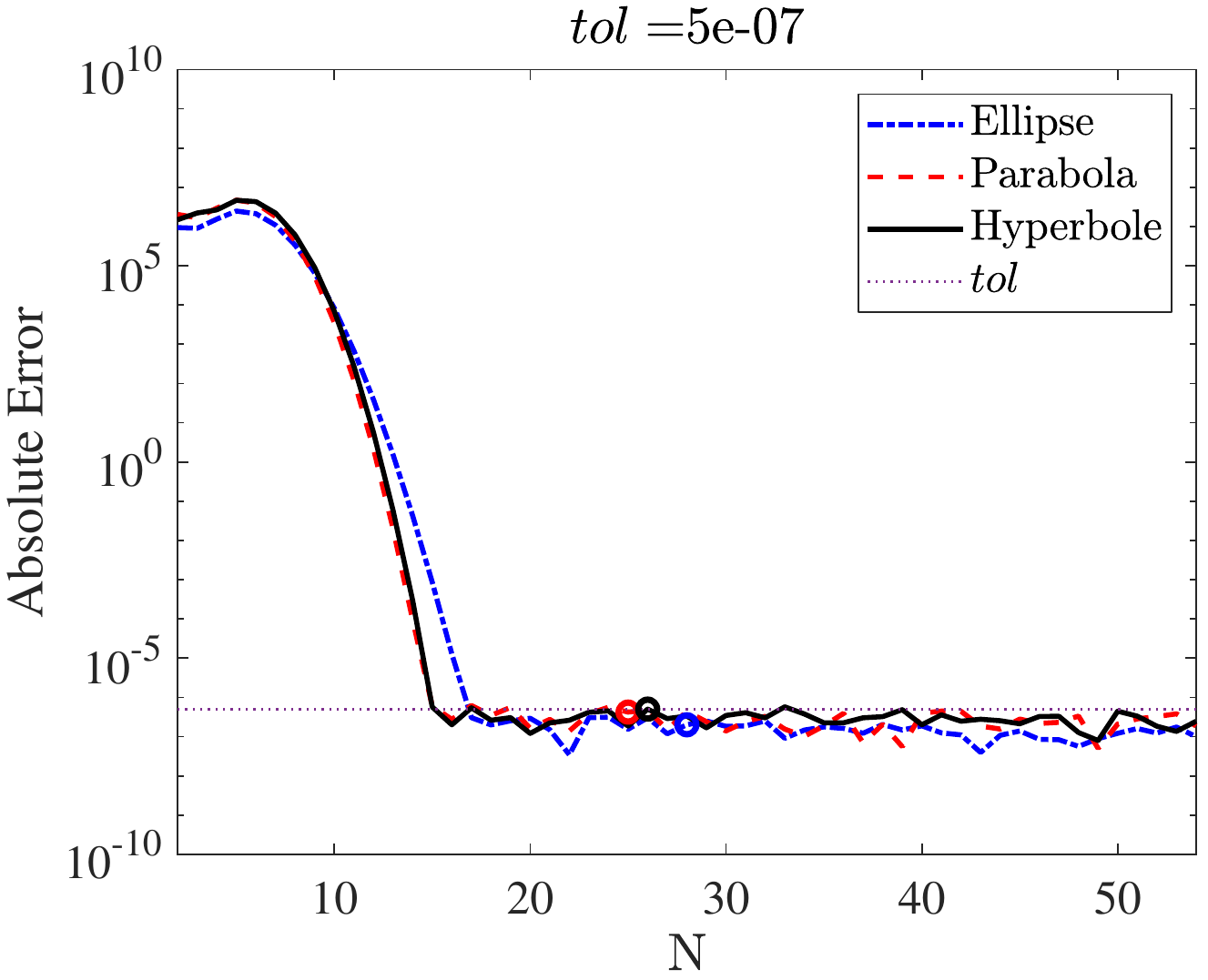}}
		\subfigure{
			\includegraphics[width=0.48\textwidth]{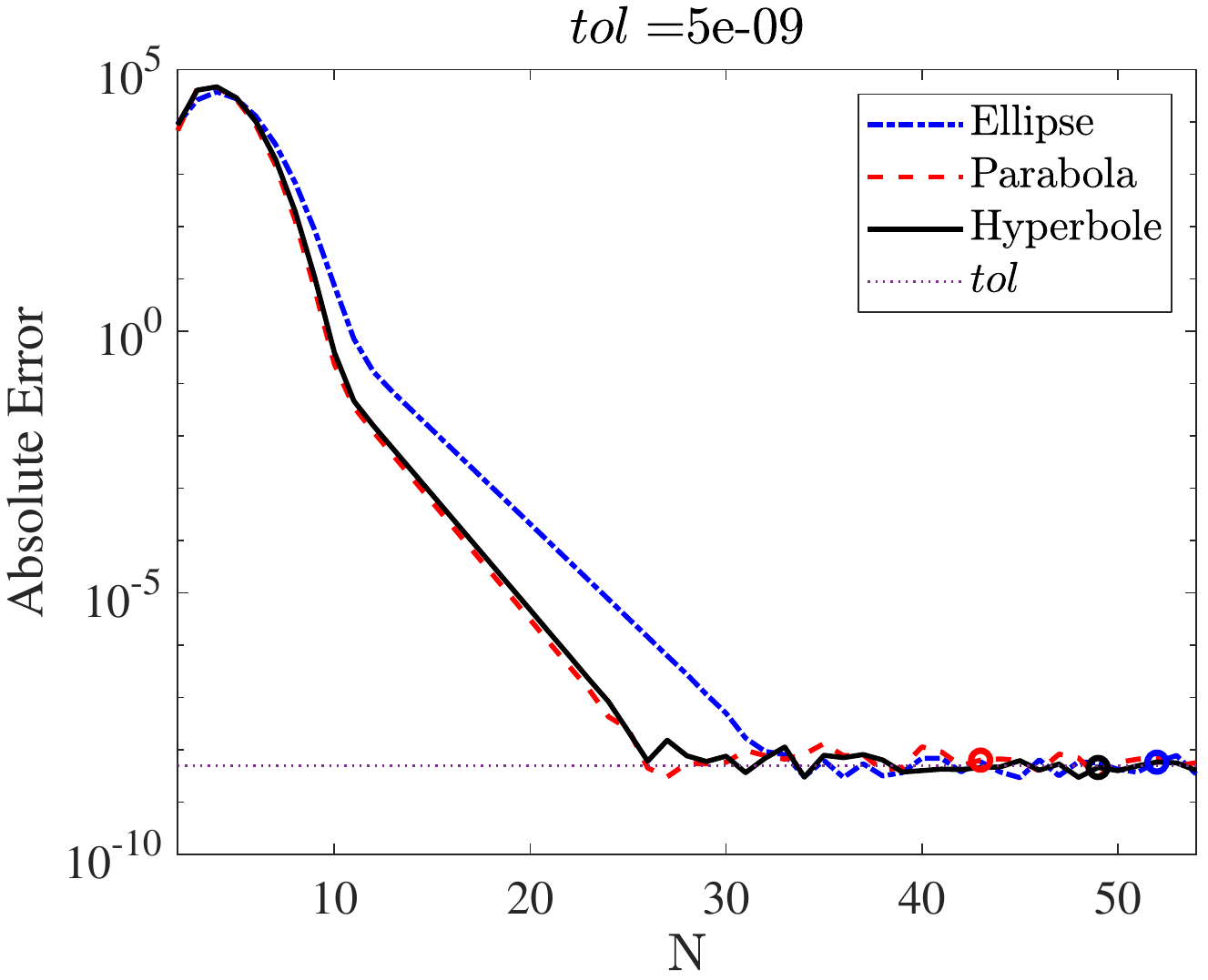}}
	}			
	
	\caption{Error vs number of nodes for Black-Scholes, $t=10$. Comparison for different values of the tolerance among the elliptic, parabolic and hyperbolic contours.}
	\label{fig5.2}
\end{figure}
\begin{figure}[]
	
		\centering{
		\subfigure{
			\includegraphics[width=0.48\textwidth]{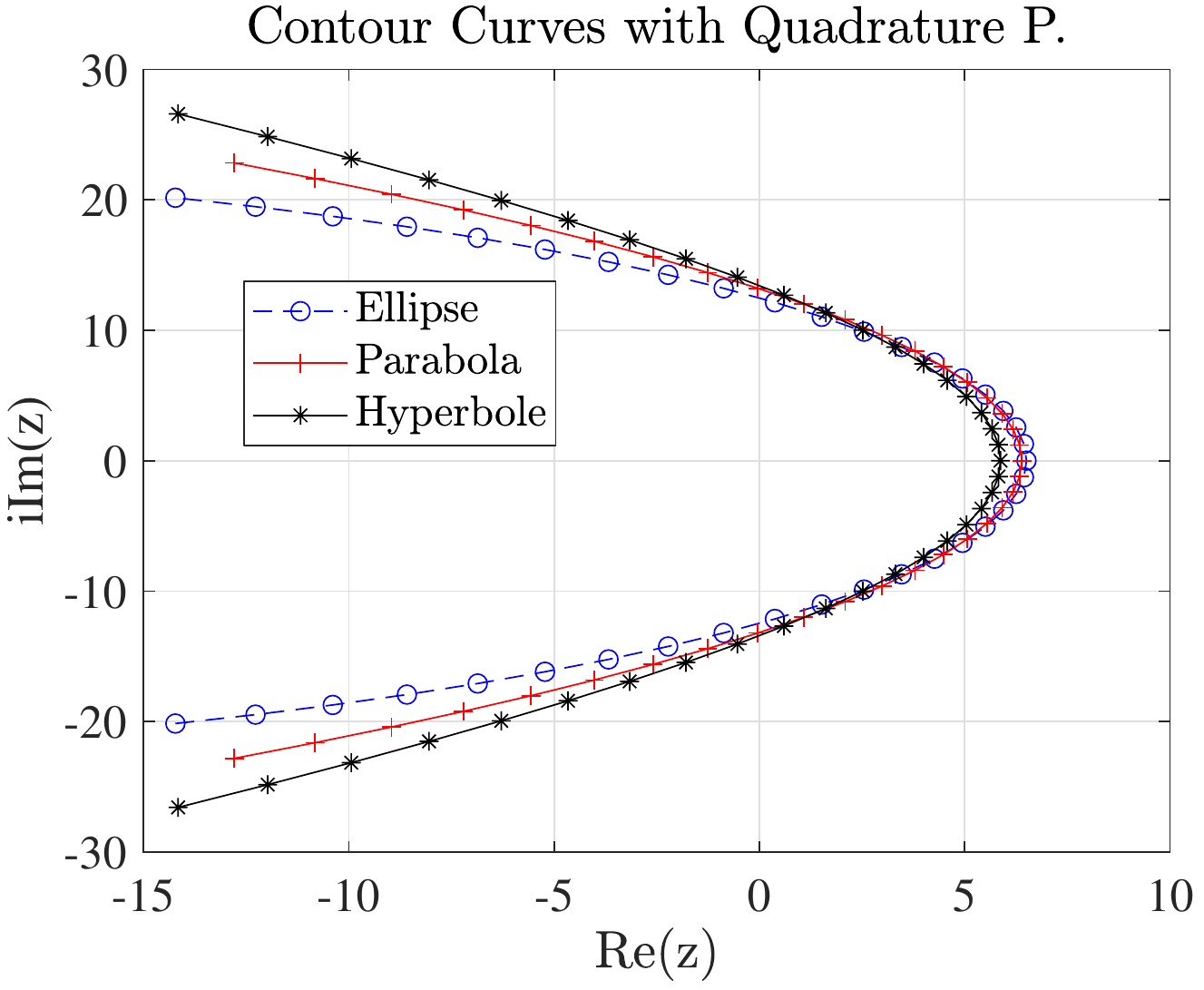}}
		\subfigure{
			\includegraphics[width=0.48\textwidth]{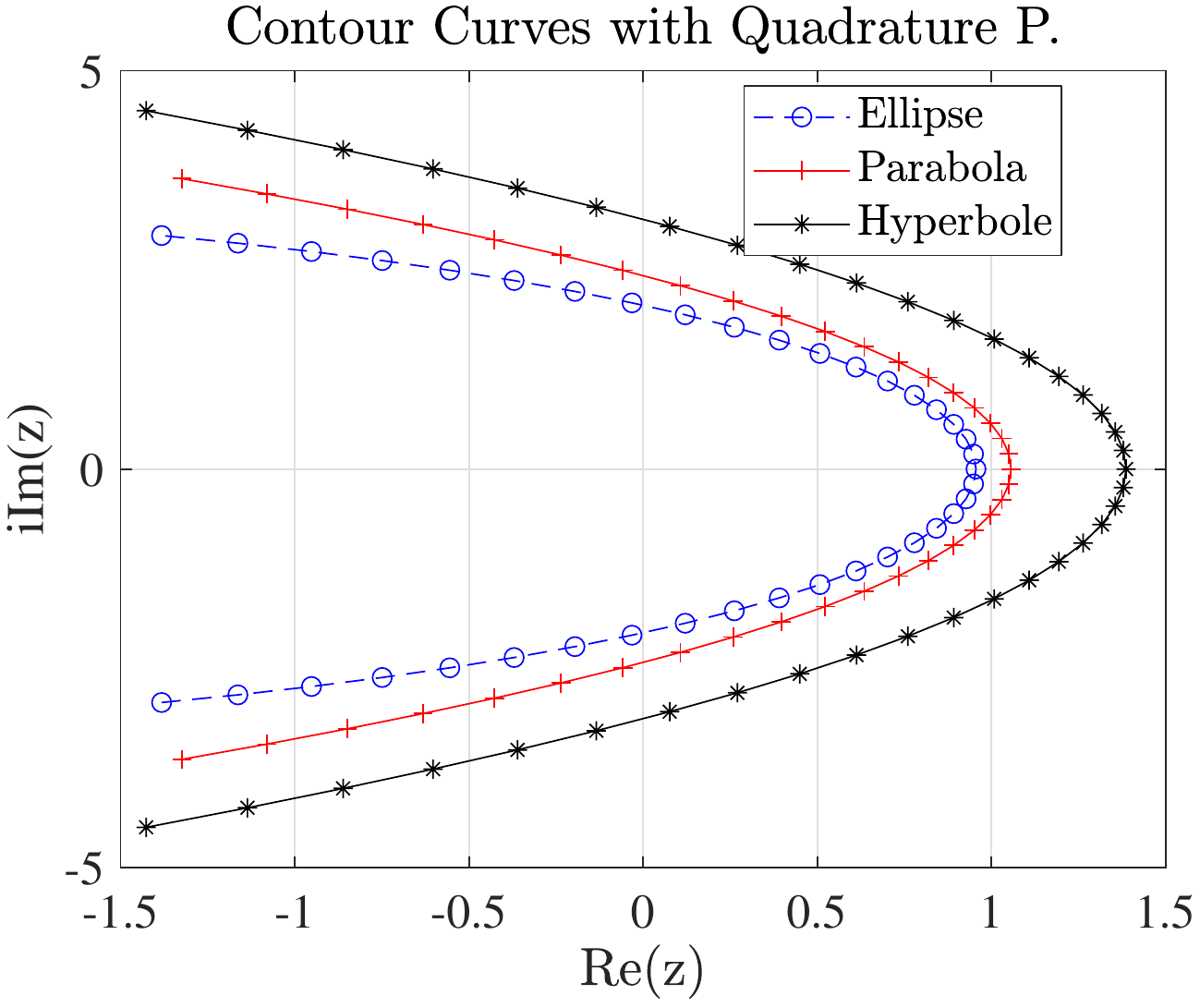}}
	}
			
	\caption{Example of integration profiles for the Black-Scholes problem for tolerance $\tol=5\cdot10^{-5}$ at time $t=1$ (left) and $t=10$ (right).}
	\label{fig5.5}
\end{figure}
All contour methods are effective when applied to the Black-Scholes test problem.
The hyperbolic contour is slightly faster in reaching the target
accuracy for $t=1$, while for $t=10$ the parabolic contour seems to provide the
best results. In Table \ref{T1} we report the values of $N$ estimated by (\ref{eq:N}), which are
similar for the three types of contour.
Also, as expected, the estimate returns larger $N$ when a higher accuracy is required.
{We note that for $t=10$ and $\tol=5\cdot 10^{-9}$ the number of
estimated quadrature points increases significantly. This is due to the fact that when
considering larger times, the exponential term may considerably amplify the numerical
error introduced by the solution of linear systems.
To avoid this, one has to select integration profiles which lie in a region with small
positive real part; this results in lower value of $a$ and thus in a larger number of
quadrature nodes in order to reach the prescribed accuracy (see (\ref{eq:minimiza})).}
\begin{table}
	\centering
	\begin{tabular}{|c|ccc|ccc|}
		\hline
		\multirow{2}{*}{$\tol$}&\multicolumn{3}{c|}{$N$ at $t=1$}&\multicolumn{3}{c|}{$N$ at $t=10$}\\
		\cline{2-7}
		& Ellipse & Parabola  & Hyperbola & Ellipse & Parabola & Hyperbola\\
		\hline
		\hline
		$5\cdot10^{-3}$   & $12$    &$11$& $13$&$18$& $17$&$18$ \\
		
		$5\cdot10^{-5}$    & $16$ &$15$& $16$&$22$&$21$&$23$\\
		
		$5\cdot10^{-7}$  & $20$ &$19$& $20$&$28$&$25$&$26$\\
		
		$5\cdot10^{-9}$  & $24$ &$22$& $24$&$52$&$43$&$49$\\
		
		\hline
	\end{tabular}
	\caption{Estimate (\ref{eq:N}) for the Black-Scholes test problem at time $t=1$ and $t=10$. Results for the elliptic, parabolic and hyperbolic profiles.}
	\label{T1}
\end{table}
\subsection{Heston equation}
The Heston equation \cite{He} is given by
\begin{equation}
\frac{\partial u}{\partial \tau}=\frac{1}{2}s^2v\frac{\partial^2u}{\partial s^2}+\rho\sigma sv\frac{\partial^2 u}{\partial s \partial v}+\frac{1}{2}\sigma^2v\frac{\partial^2u}{\partial v^2}+(r_d-r_f)s\frac{\partial u}{\partial s}+\kappa(\eta-v)\frac{\partial u}{\partial v}-r_du.
\label{5.2}
\end{equation}
The unknown function $u(s, v,\tau)$ represents the price of a European option when at time $t-\tau$ the corresponding asset price is equal to $s$ and its variance is $v$. We consider the equation on the unbounded domain
\begin{equation*}
0\le\tau\le t,\;\;s>0,\;v>0,
\end{equation*}
where the time $t$ is fixed. The parameters $\kappa>0$, $\sigma>0$, and $\rho\in[-1,1]$ are given. Moreover equation (\ref{5.2}) is usually considered under the condition $2\kappa\eta>\sigma^2$ that is known as the Feller condition (see \cite{JKWW}). We take equation (\ref{5.2}) together with the initial condition
\begin{equation*}
u(s,v,0)=\max(0,s-K),
\end{equation*}
where $K>0$ is fixed a priori (and represents the strike price of the option), and boundary condition
\begin{equation*}
u(L,v,\tau)=0,\;\;0\le\tau\le t.
\end{equation*}
For the numerical solution of (\ref{5.2}), we need to choose a bounded domain of integration, we follow \cite{ITHF} for this issue. In particular, we fix two positive constants $S$, $V$ and we let the two variables $s$, $v$ vary in the set
\begin{equation*}
0\le s\le S,\;\; 0\le v\le V.
\end{equation*}
On the new boundary, we need to add two more conditions (specific for the European call option),
\begin{align*}
&\frac{\partial u}{\partial s}(S,v,\tau)=\e^{-r_f\tau},\;\;0\le\tau\le t,\\
&u(s,V,\tau)=s\e^{-r_f\tau},\;\;0\le\tau\le t,
\end{align*}
which are treated analogously yo the boundary condition in (\ref{ibcBS}).
We use the spatial discretization proposed in \cite{ITHF} for $k=1.5$, $\eta=0.04$, $\sigma=0.3$, $\rho=-0.9$, $r_d=0.025$, $r_f=0$, $K=100$, $L=0$, $S=8K$, $V=5$. In Figure~\ref{fig5.3}, we plot the error for a selection of tolerances for $t=1$ and in Figure~\ref{fig5.4} we do the same for $t=10$. In Figure~\ref{fig5.6} we show the selected profiles of integration for $\tol=5\cdot10^{-5}$ at $t=1$ and $t=10$.
\begin{figure}
	\centering{
		\subfigure{
			\includegraphics[width=0.48\textwidth]{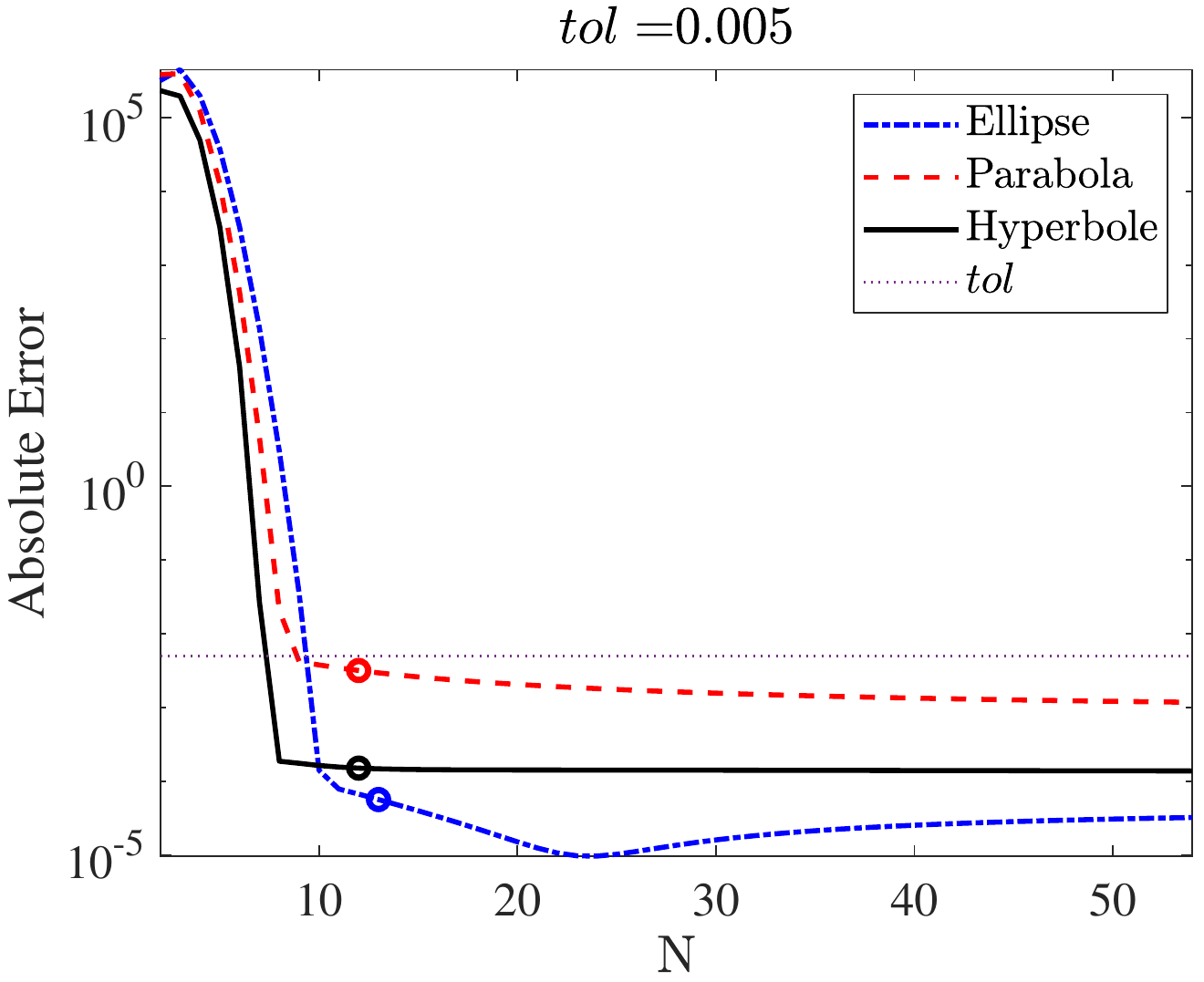}}
		\subfigure{
			\includegraphics[width=0.48\textwidth]{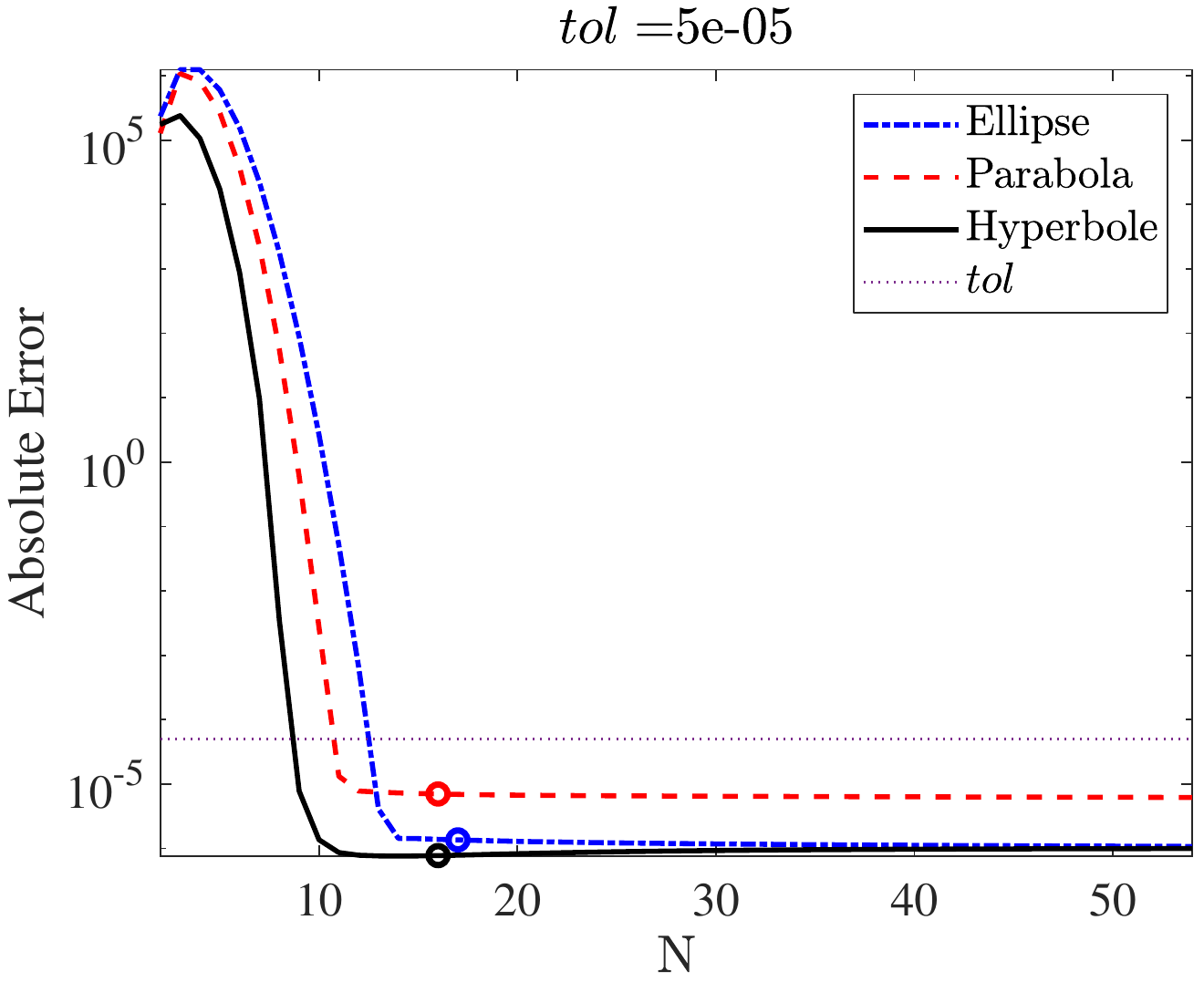}}
	}	
	\centering{
		\subfigure{
			\includegraphics[width=0.48\textwidth]{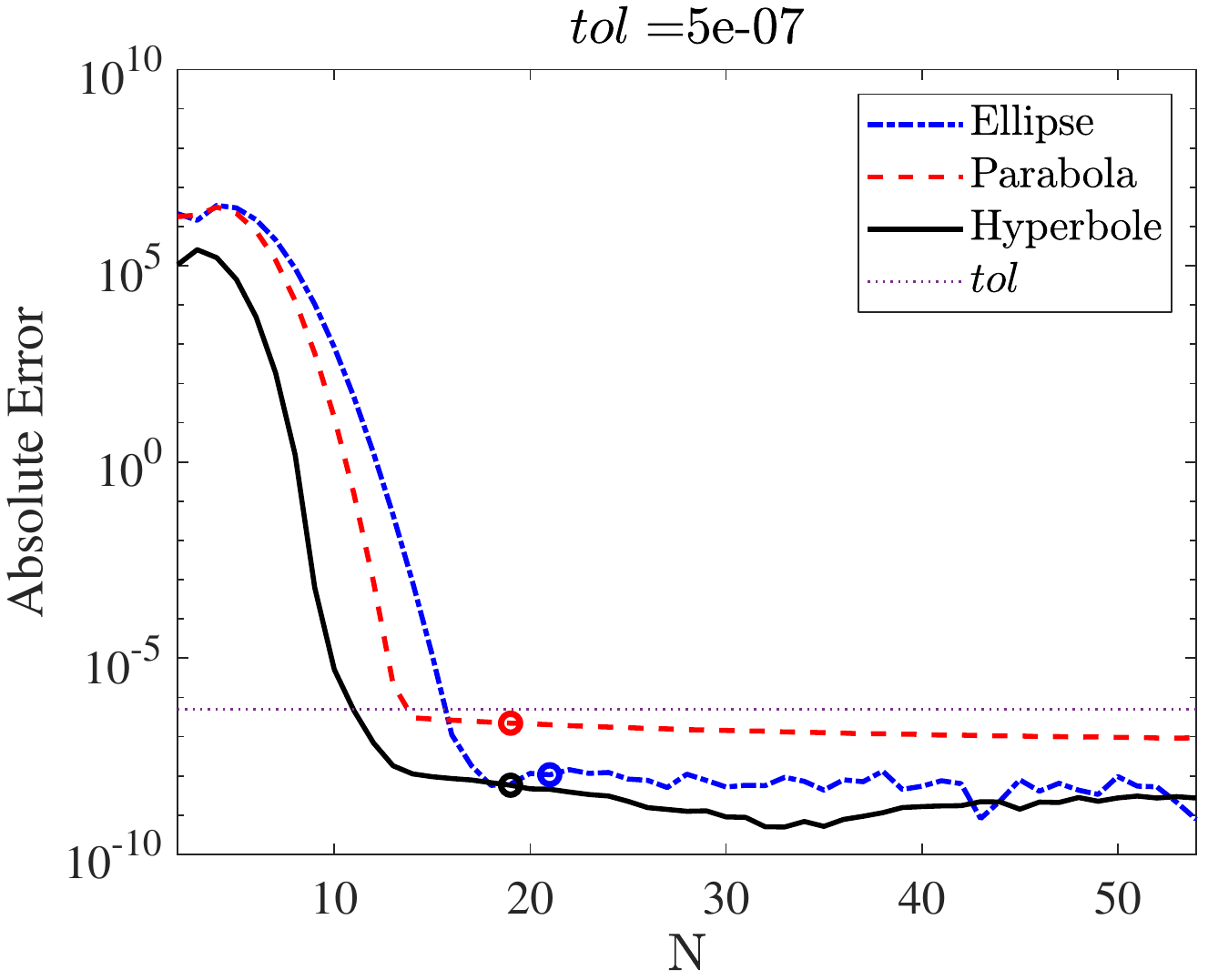}}
		\subfigure{
			\includegraphics[width=0.48\textwidth]{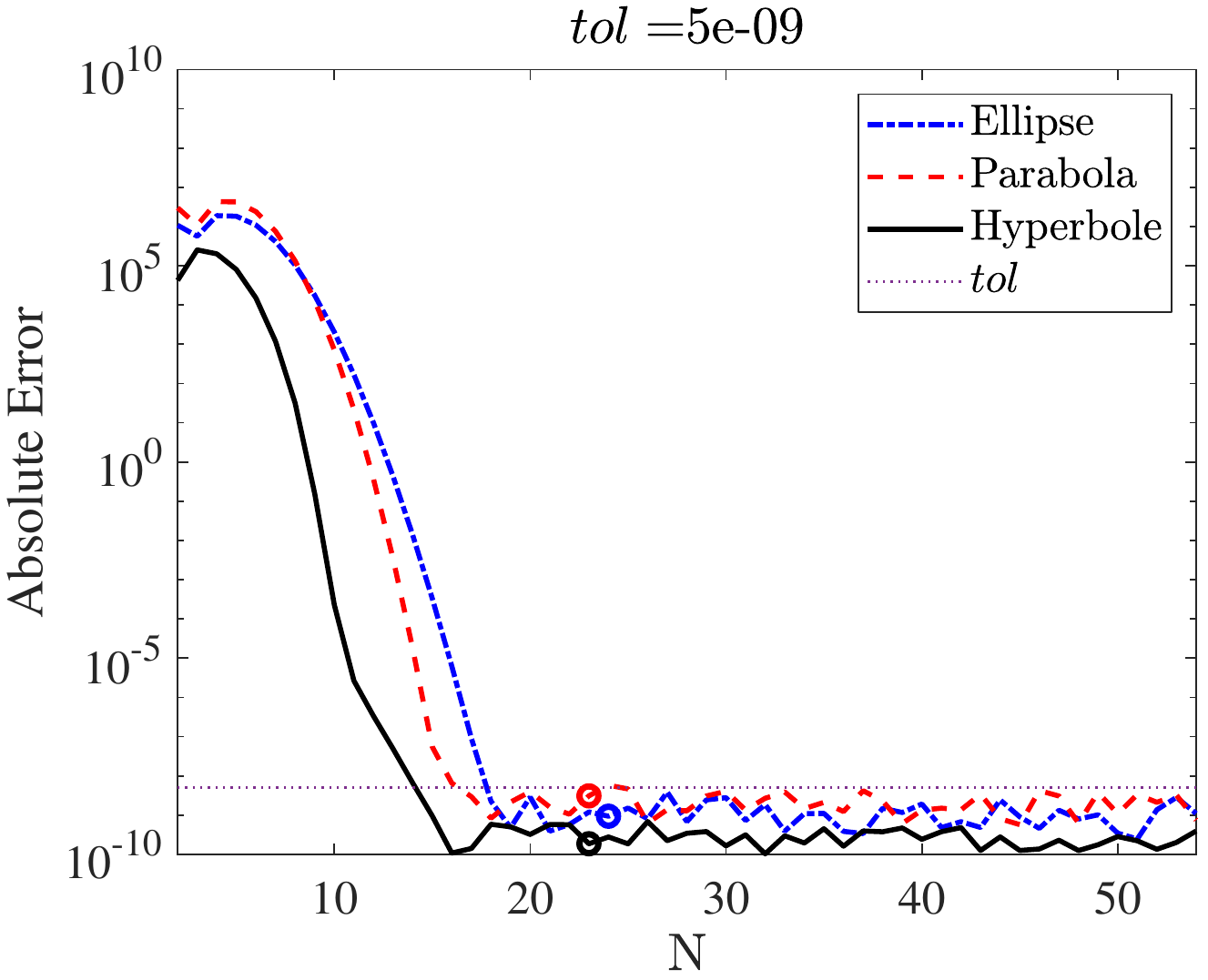}}
	}			
	
	\caption{Error vs number of nodes for Heston, $t=1$. Comparison for different values of the tolerance among the elliptic, parabolic and hyperbolic contours.}
	\label{fig5.3}
\end{figure}
\begin{figure}
	\centering{
		\subfigure{
			\includegraphics[width=0.48\textwidth]{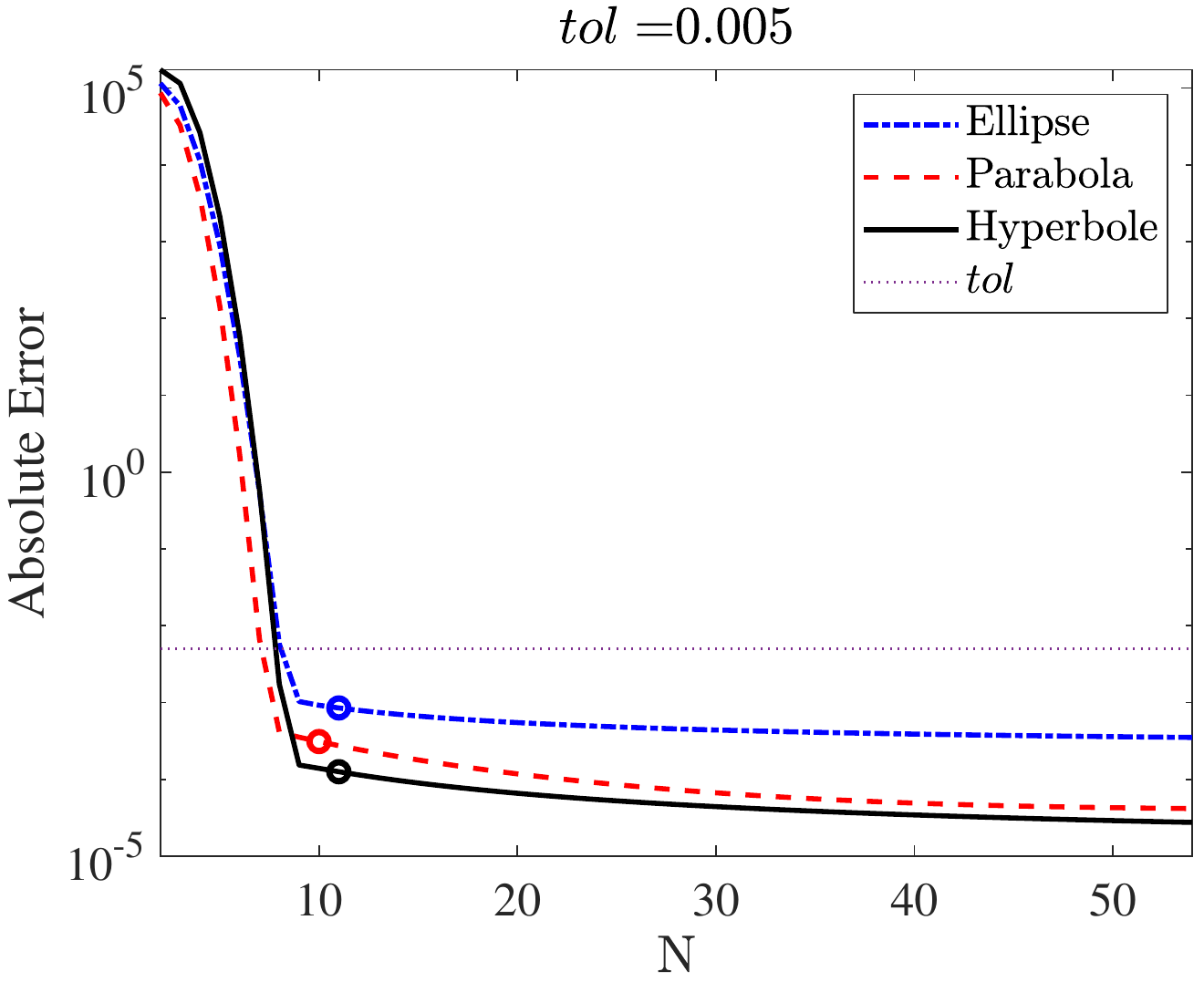}}
		\subfigure{
			\includegraphics[width=0.48\textwidth]{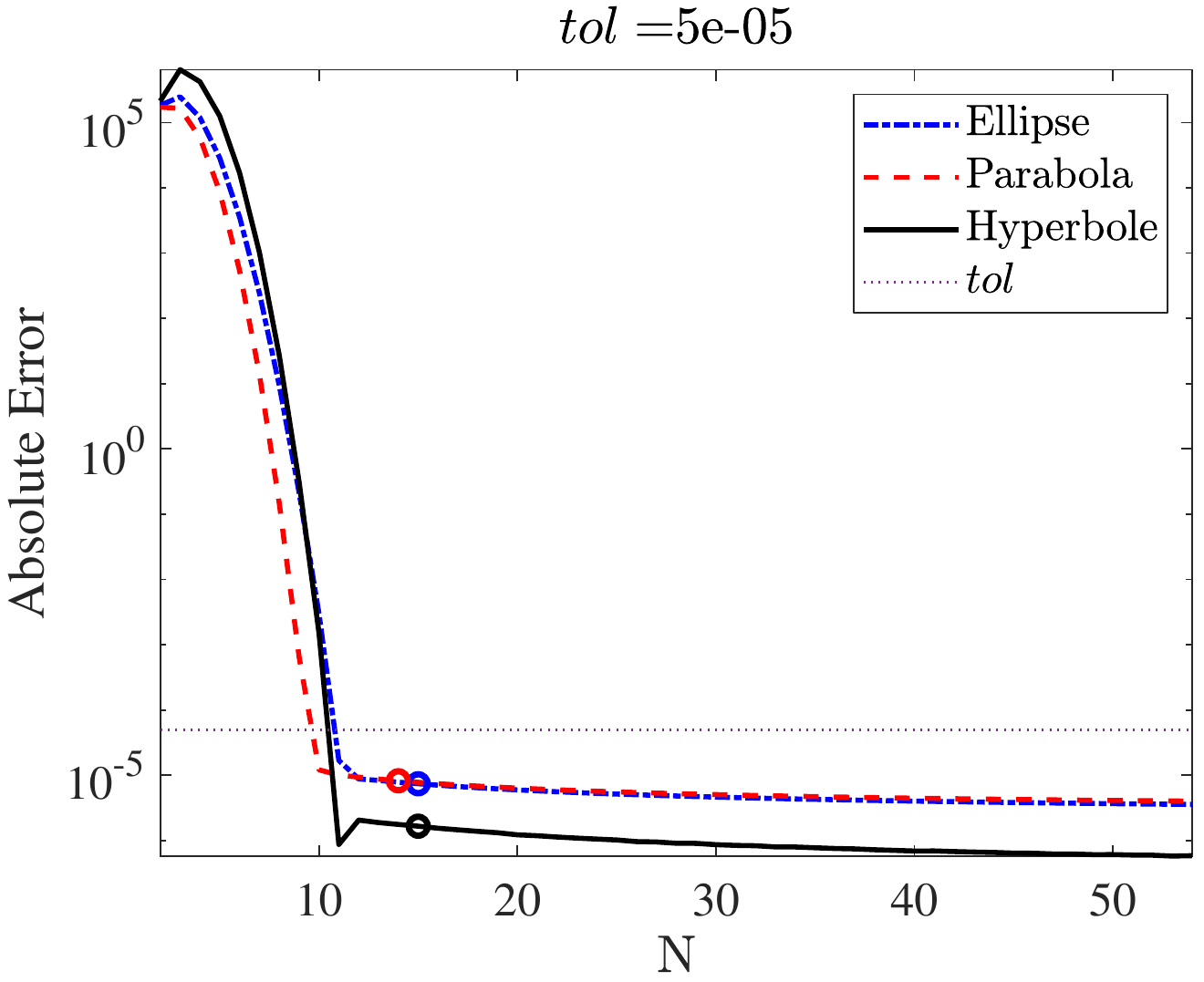}}
	}	
	\centering{
		\subfigure{
			\includegraphics[width=0.48\textwidth]{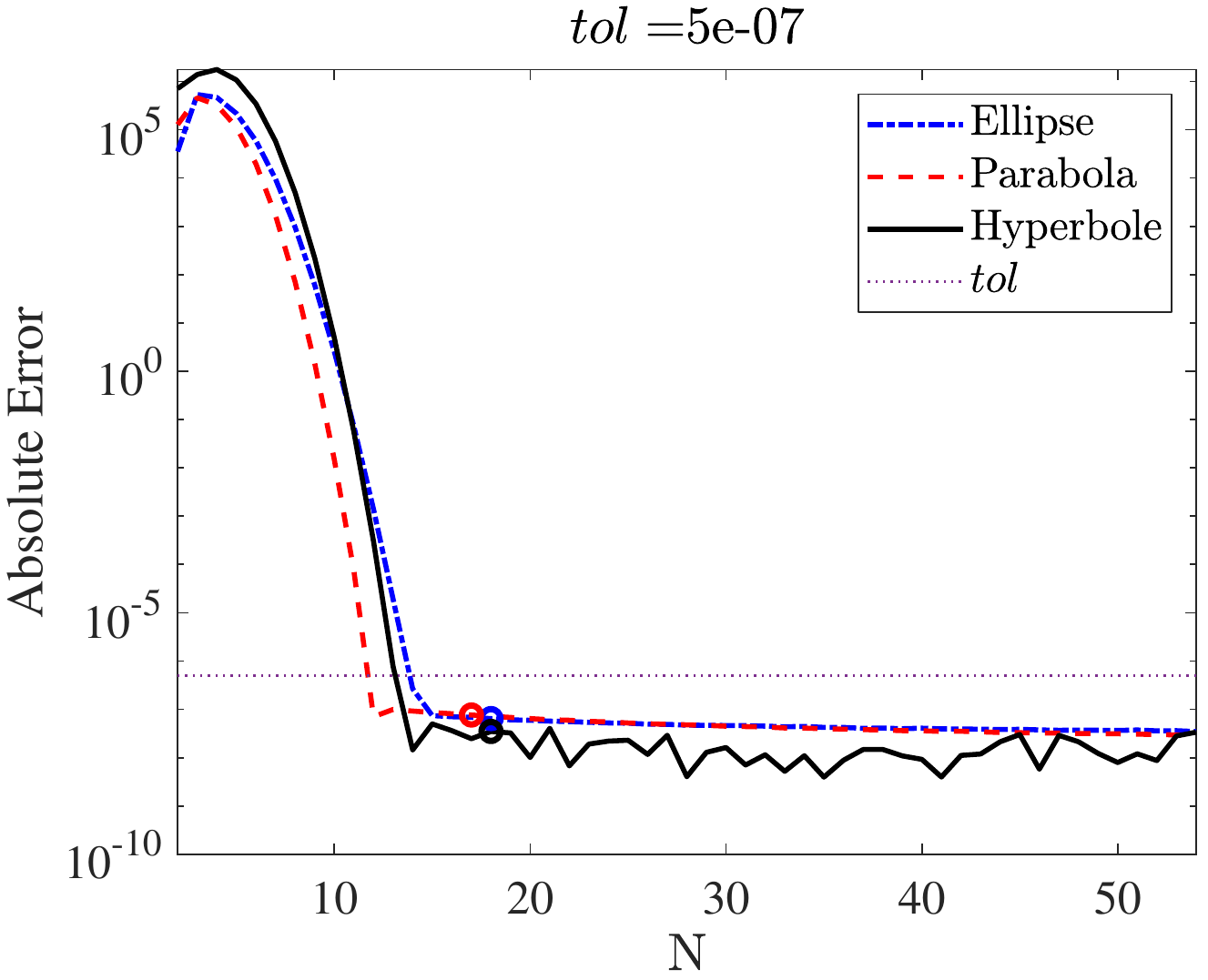}}
		\subfigure{
			\includegraphics[width=0.48\textwidth]{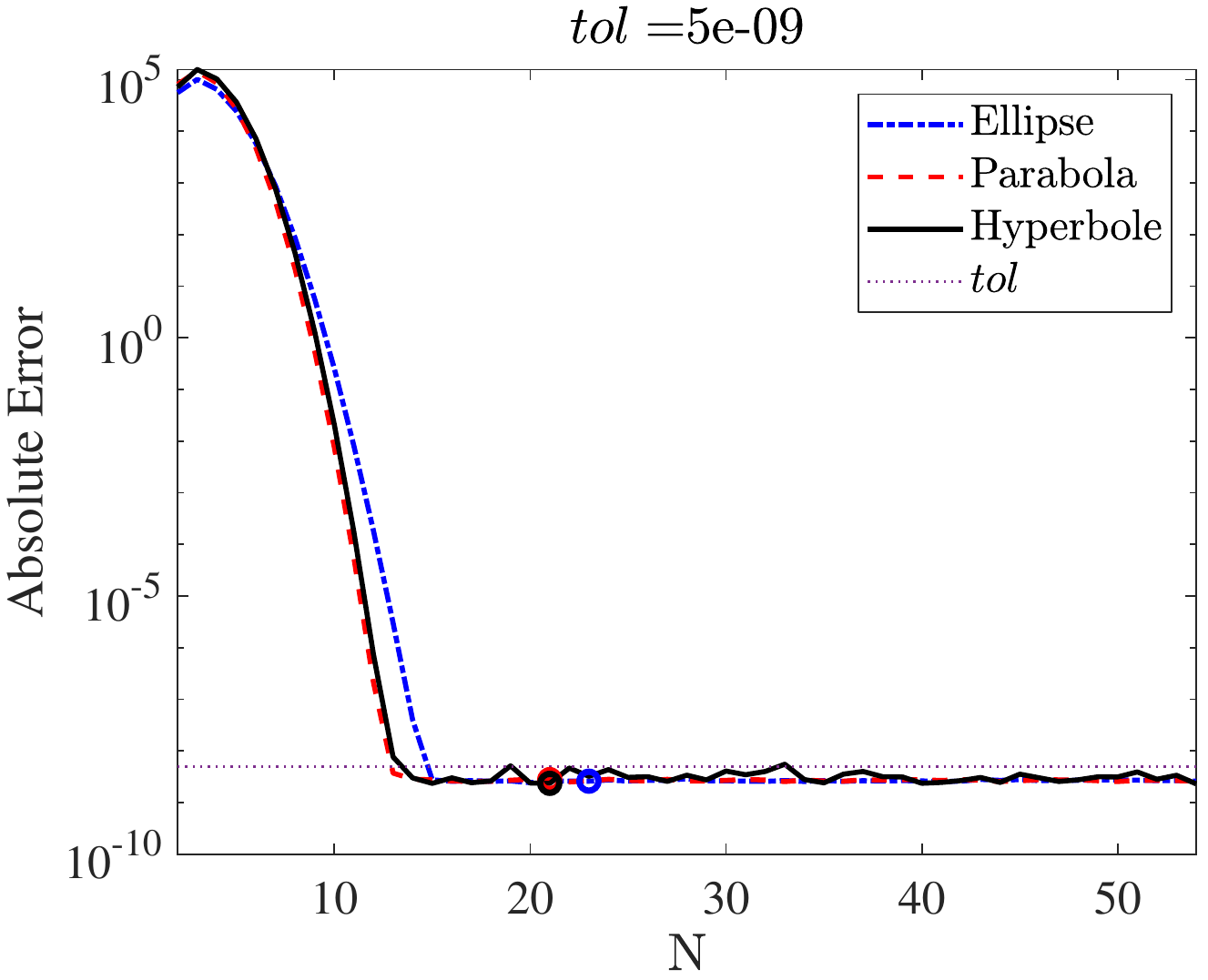}}
	}			
	
	\caption{Error vs number of nodes for Heston, $t=10$. Comparison for different values of the tolerance among the elliptic, parabolic and hyperbolic contours.}
	\label{fig5.4}
\end{figure}
\begin{figure}[]

		\centering{
	\subfigure{
		\includegraphics[width=0.48\textwidth]{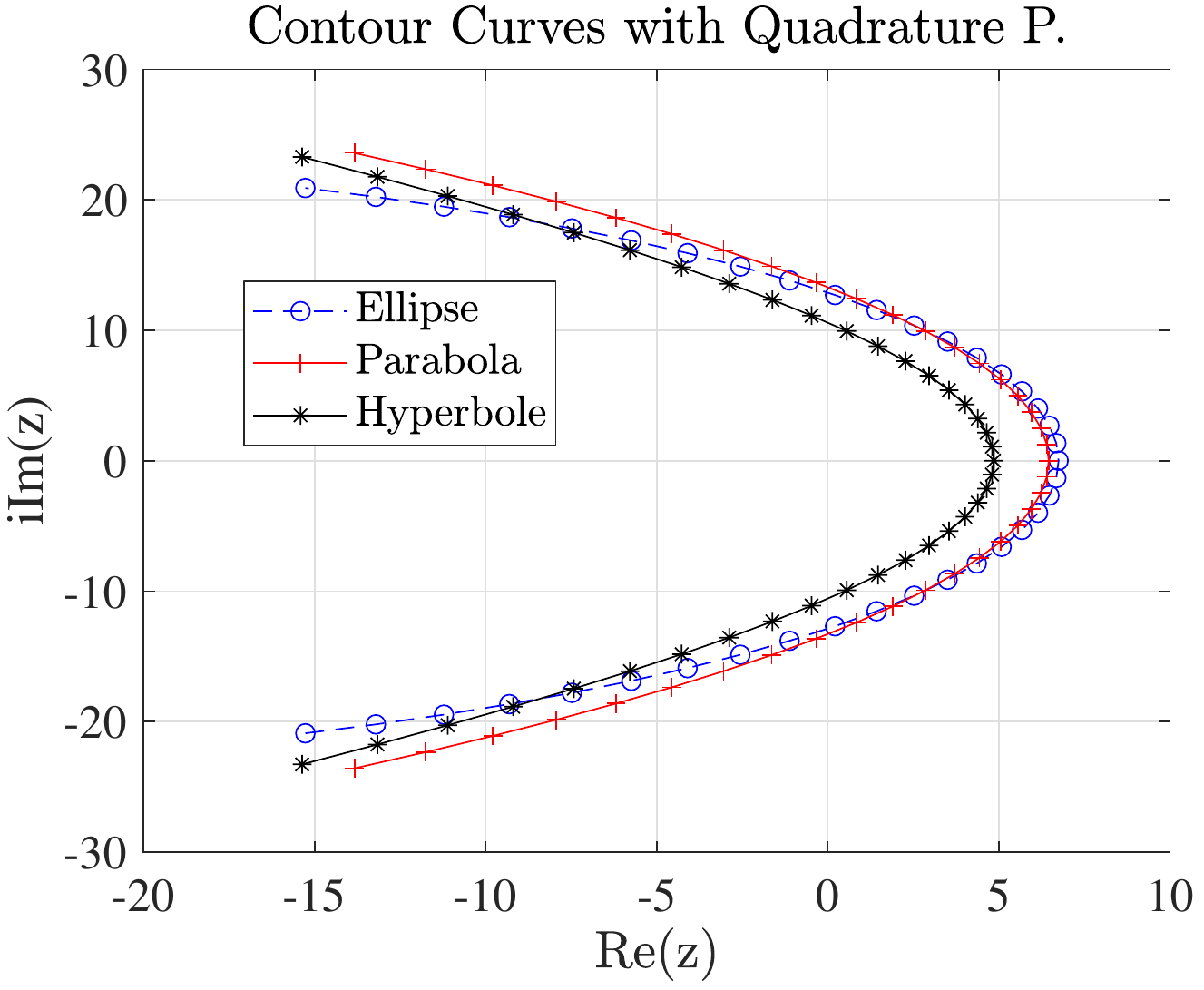}}
	\subfigure{
		\includegraphics[width=0.48\textwidth]{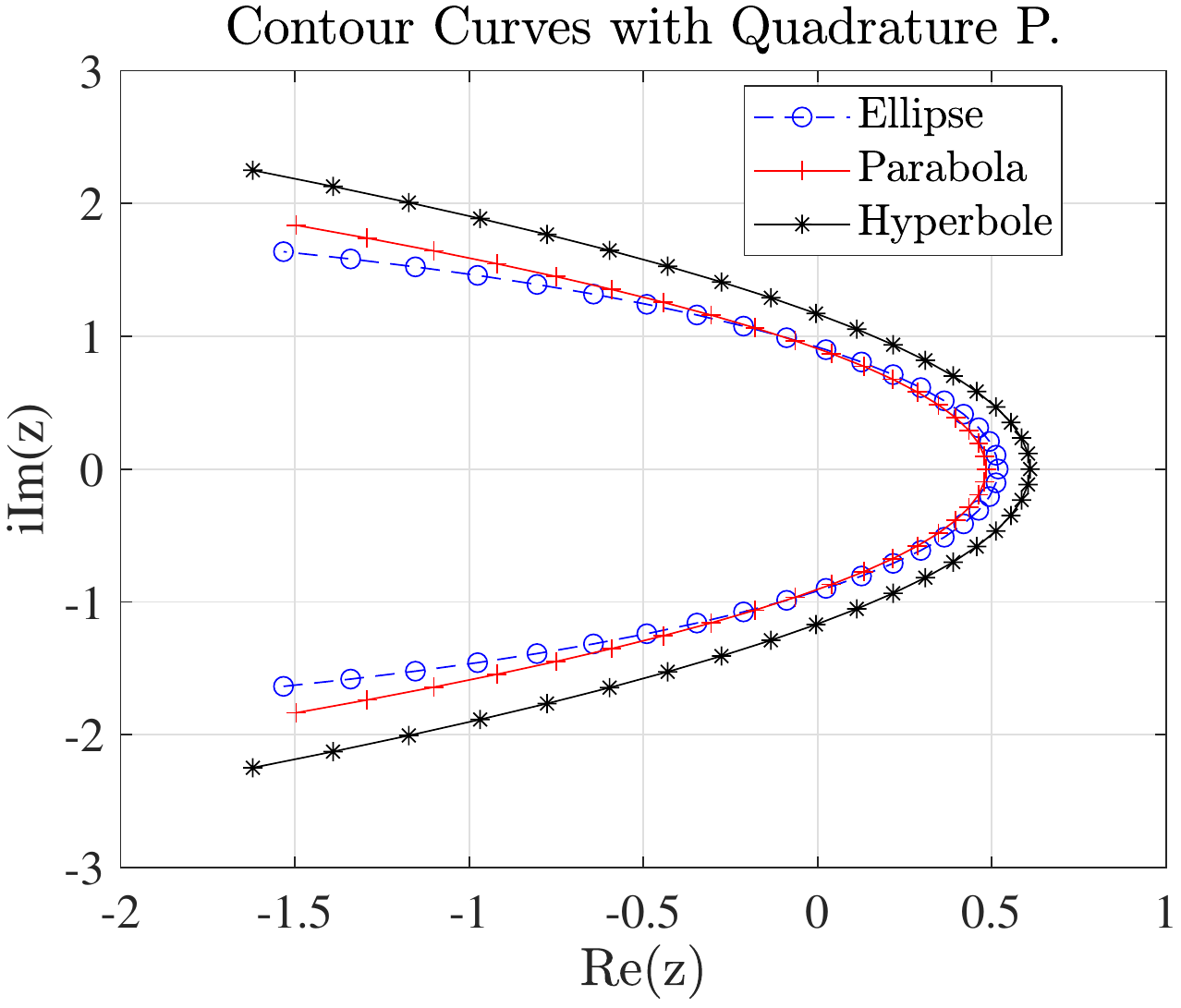}}
    }

    \caption{Example of integration profiles for the Heston problem for tolerance $\tol=5\cdot10^{-5}$ at time $t=1$ (left) and $t=10$ (right).}
	\label{fig5.6}
\end{figure}

We obtain similar results for the Heston test problem to the ones reported for the Black-Scholes problem. The hyperbolic contour is again the fastest one in reaching the target accuracy for $t=1$. {For $t=10$ elliptic, parabolic and hyperbolic contours show almost the same behaviour}. In Table \ref{T2} we report the values of $N$ estimated by (\ref{eq:N}) and again we do not observe remarkable differences for the three types of contour.
\begin{table}
	\centering
	\begin{tabular}{|c|ccc|ccc|}
		\hline
		\multirow{2}{*}{$\tol$}&\multicolumn{3}{c|}{$N$ at $t=1$}&\multicolumn{3}{c|}{$N$ at $t=10$}\\
		\cline{2-7}
		& Ellipse & Parabola  & Hyperbola & Ellipse & Parabola & Hyperbola\\
		\hline
		\hline
		$5\cdot10^{-3}$   & $13$    &$12$& $12$&$11$& $10$&$11$ \\
		
		$5\cdot10^{-5}$    & $17$ &$16$& $16$&$15$&$14$&$15$\\
		
		$5\cdot10^{-7}$  & $21$ &$19$& $19$&$18$&$17$&$18$\\

		$5\cdot10^{-9}$  & $24$ &$23$& $23$&$23$&$21$&$21$\\
		
		\hline
	\end{tabular}
	\caption{Estimate (\ref{eq:N}) for the Heston test problem at time $t=1$ and $t=10$. Results for the elliptic, parabolic and hyperbolic profiles.}
	\label{T2}
\end{table}
\subsection{Extension to the case of time intervals}\label{subsec:TW}
This situation is particularly important when the time $T$ at which the solution is required is only known approximately.
In this subsection we extend the results in \cite[Section 7]{GLN20} to determine the solution on an entire time window $[t_0,t_1]$,
with \[t_1=\Lambda t_0,\;\Lambda>1.\]
Most part of the computational effort in the evaluation of (\ref{eq:tr}) is devoted to the solution of linear systems with matrices $zI-A$, for $z$ the quadrature nodes. However this inversion does not involve the time $t$, that only appears in the exponential term. Thus, it is of high interest being able to use a unique integration contour on a whole time interval $[t_0,t_1]$. This requires a suitable modification of our strategy to construct the integration contour. We also aim to determine an acceptable amplitude of the time window.

Concerning the profile of integration, which is uniquely defined by the parameter $a$, we start by constructing $\Gamma_{left}$
as explained in Section \ref{sec:pseudospec}, for $t=t_0$ (lower end of the time interval). The choice of $t_0$ reflects into
the setting of the parameter $z^L$ as explained in Section \ref{sec:param}.
An application of the construction described in Section \ref{sec:pseudospec} gives the interpolation point $d+ir$, which uniquely
defines $\Gamma_{left}$.
Then we determine $a$ by following the procedure described in Section~\ref{sec:param} with $t=t_1$, the upper end of the time interval.
We use the profile determined in this way for all $t\in[t_0,t_1]$.

Despite the fact that, theoretically, the amplitude of the window can be arbitrarily large,
we need to keep into account the role of the exponential term in amplifying the error introduced
by the conditioning of $zI-A$.
We approximate the exact solution $u(t)$ by the linear combination
\begin{equation}
	\tilde{I}_N=\frac{c}{Ni}\sum_{j=1}^{N-1}\e^{z(x_i)t}\hat{u}_jz'(x_j),
\end{equation}
where $\hat{u}_j=\hat{u}(z(x_j))+\rho_j$ and $\rho_j$ is the error in the numerical solution of the linear system
\begin{equation}
	((z(x_j))I-A)\hat{u}=u_0+\hat{b}(z(x_j)),
\end{equation}
for $x_j$ our quadrature nodes and with the assumption that the nodes $x_j$, the parametrization $z(x)$, and its
derivative $z'(x)$ are computed exactly. The actual error in our computation is given by
\begin{equation}
	\tilde{err}_N=\big|u(t)-\tilde{I}_N\big|,
\end{equation}
that we can estimate in the following way:
\begin{equation}
	\widetilde{err}_N=\Big|u(t)-\frac{c}{Ni}\sum_{j=1}^{N-1}\e^{z(x_i)t}\hat{u}_jz'(x_j)\Big|\le err_T+err_N+err^{num}_N,
\end{equation}
where $err_T$ is the component of the total error due to the truncation of the integral, $err_N$ is the component
due to the approximation of the integral by the quadrature rule and finally $err^{num}_N$ is the component due to
the fact that we operate with finite precision arithmetic.
This last component reads as
\begin{align}
	{err}_N^{num}=&\Bigg|\frac{c}{Ni}\sum_{j=1}^{N-1}\e^{z(x_i)t}(\hat{u}(z(x_j))-\hat{u}_j)z'(x_j)\Bigg|\nonumber\\
	\le&\frac{c_{\max}}{N}\sum_{j=1}^{N-1}\e^{\real{(z(x_i))}t}|\rho_j||z'(x_j)|.\label{6.1}
\end{align}
Therefore, once the integration contour for a time window $[t_0,\;t_1]$ is fixed, we can observe that $err_N^{num}$
does not remain constant but instead changes exponentially with respect to time.
Since we want that $\widetilde{err}_N \le \tol$ for all $t\in[t_0,\;t_1]$, we need to check that (\ref{6.1}) is below
$\tol$ at $t_0$ and $t_1$.
If so, we proceed with the computation, otherwise we halve the amplitude of the window by increasing $t_0$ or
decreasing $t_1$, depending on which of the two time instants fails to satisfy the inequality (\ref{6.1}) smaller
than $\tol$. If for both $t_0$ and $t_1$ (\ref{6.1}) is greater than $\tol$ we deduce that the problem is possibly ill
conditioned and we suggest to increase $\tol$.

About the truncation parameter $c$ we observe that this value decreases as time increases, therefore we run Algorithm
\ref{A1} for $t=t_0$ and we use the computed value of $c$ in $t_0$ for every $t\in[t_0, t_1]$.
Doing so we expect to have the final error proportional to $\tol$ for $t=t_0$ and smaller than $\tol$ when $t>t_0$.

We show few numerical experiments for both Black-Scholes and Heston equations. In particular, we make the experiments
on the intervals $[0.1, \Lambda0.1]$ and $[1, \Lambda1]$ for $\Lambda=10$.
In the plots of Figures \ref{fig5.7}, \ref{fig5.8}, \ref{fig5.9}, \ref{fig5.10} we show the numerical results for
Black-Scholes and Heston equations. The target tolerance chosen is $\tol=5\cdot10^{-8}$ for Black-Scholes and
$\tol=5\cdot10^{-4}$ for Heston. We also fix $z^R = 0.06$ for both problems.
We note that for the Heston problem the time window $[1, 10]$ has been found to be too large by the procedure
previously described, therefore the initial time was automatically increased.
For the Black-Scholes problem a slow-down in convergence is observed as $t$ increases, we see this phenomena also
in the Heston problem even if in a less remarkable way. This because, as $t$ increases, the truncation value used
$c$ is larger than what should be required to achieve the prescribed tolerance $\tol$.
This results into a smaller error for $N$ large enough but also produces a slower convergence.
\begin{figure}
	\centering{
		\subfigure{
			\includegraphics[width=0.48\textwidth]{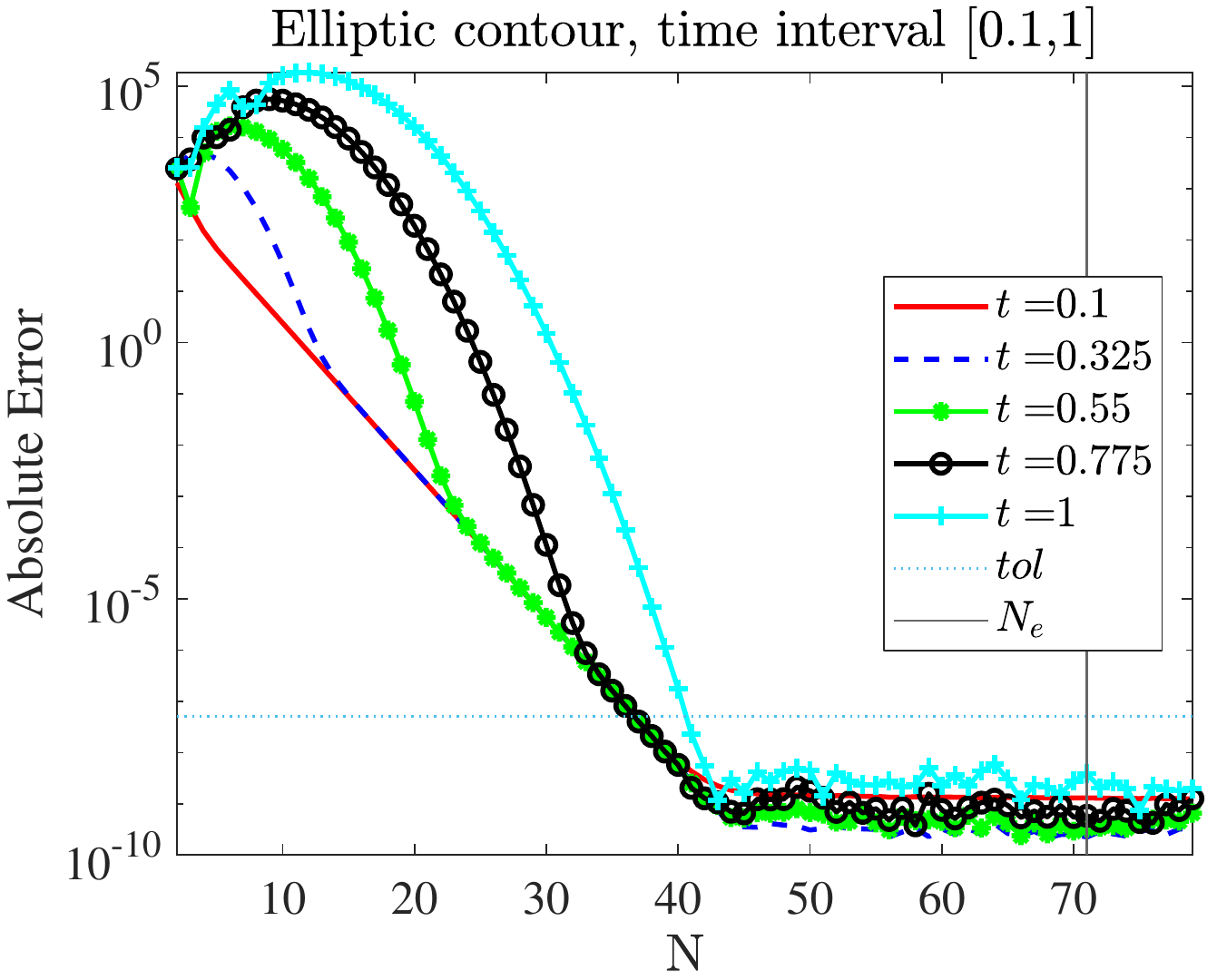}}
		\subfigure{
			\includegraphics[width=0.48\textwidth]{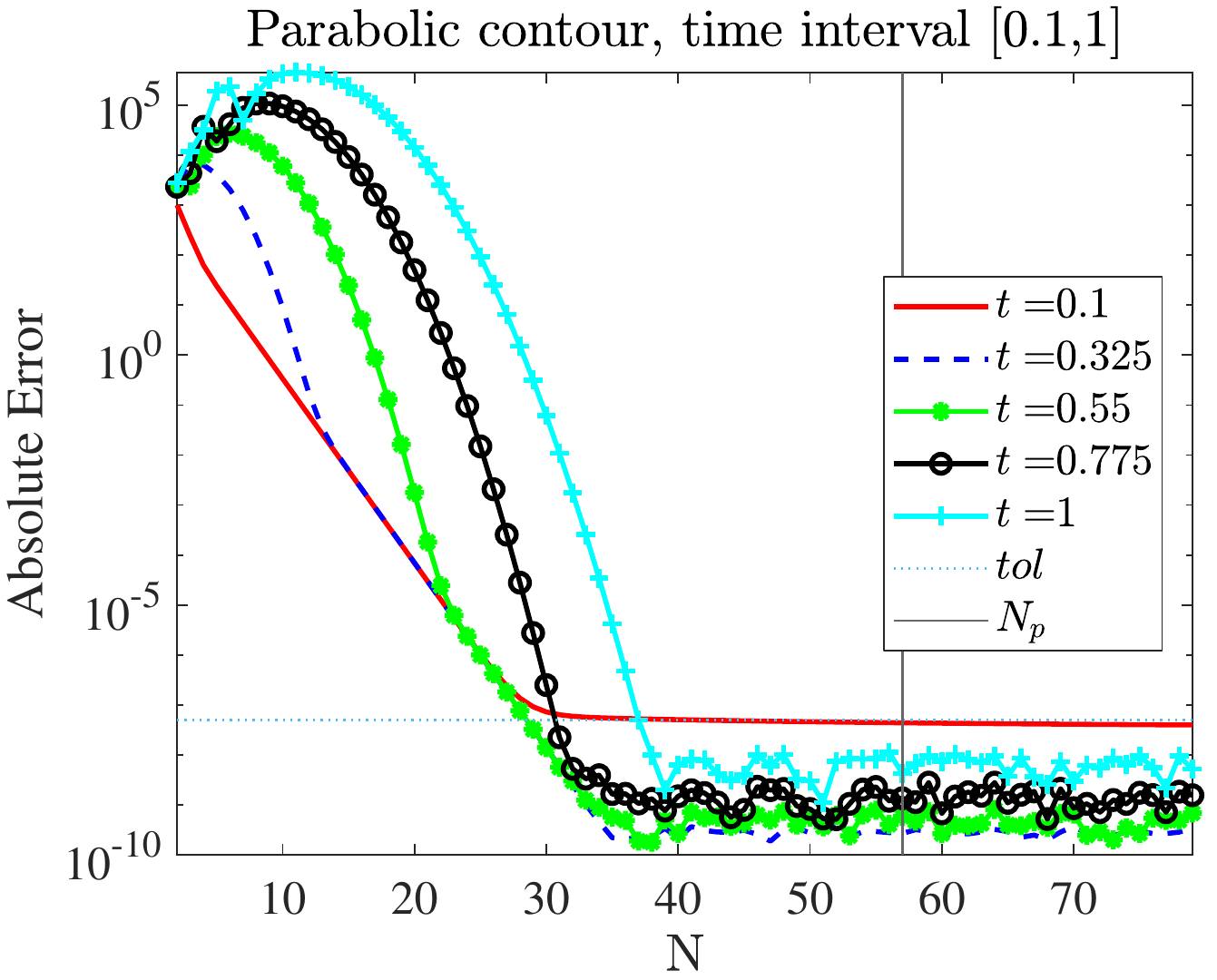}}
	}	
	\centering{
		\subfigure{
			\includegraphics[width=0.48\textwidth]{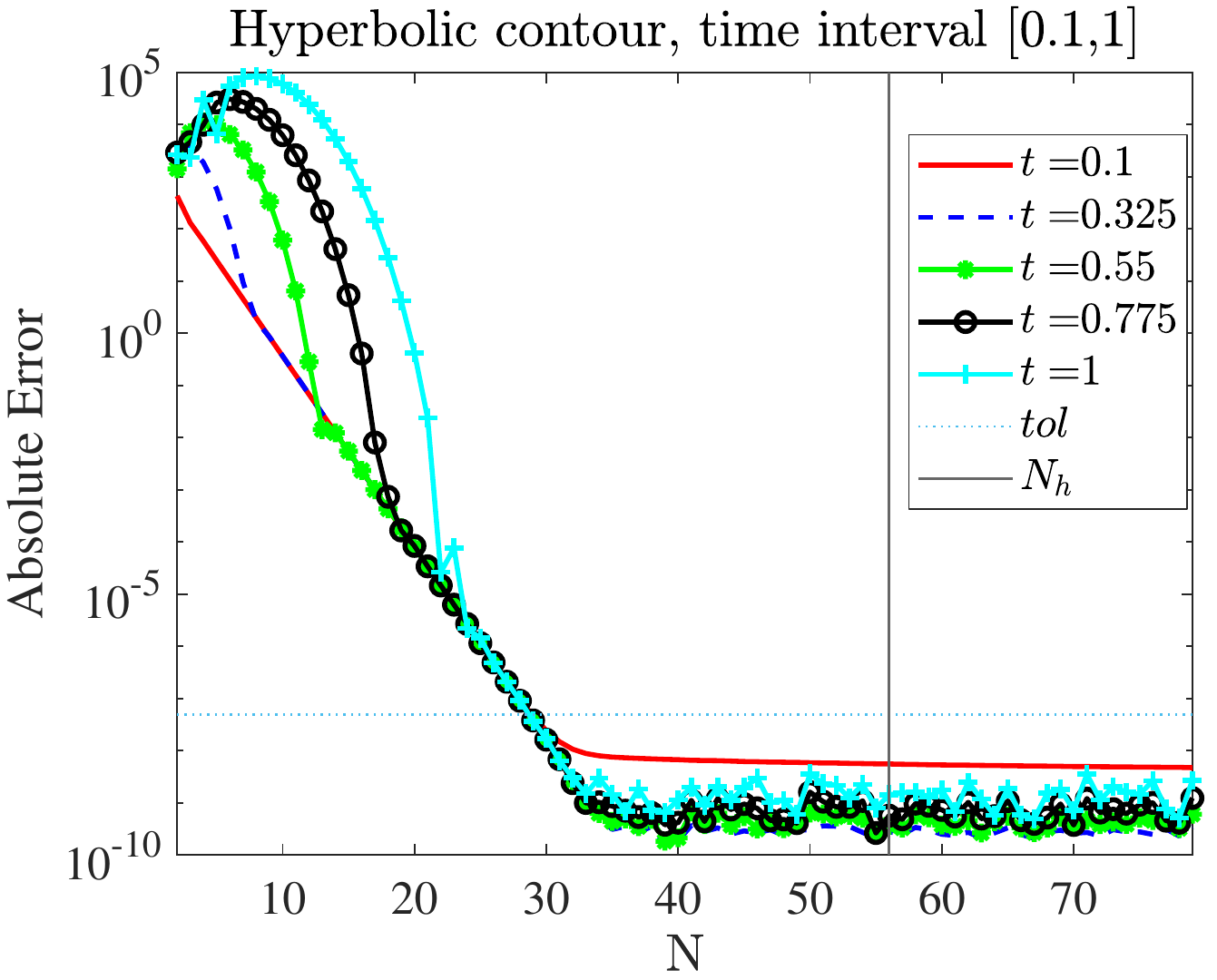}}
		
	}			
	
	\caption{Black-Scholes equation in time interval $[0.1,1]$, $\tol=5\cdot10^{-8}$, $z^L=-400$, $z^R=0.01$.}
	\label{fig5.7}
\end{figure}
\begin{figure}
	\centering{
		\subfigure{
			\includegraphics[width=0.48\textwidth]{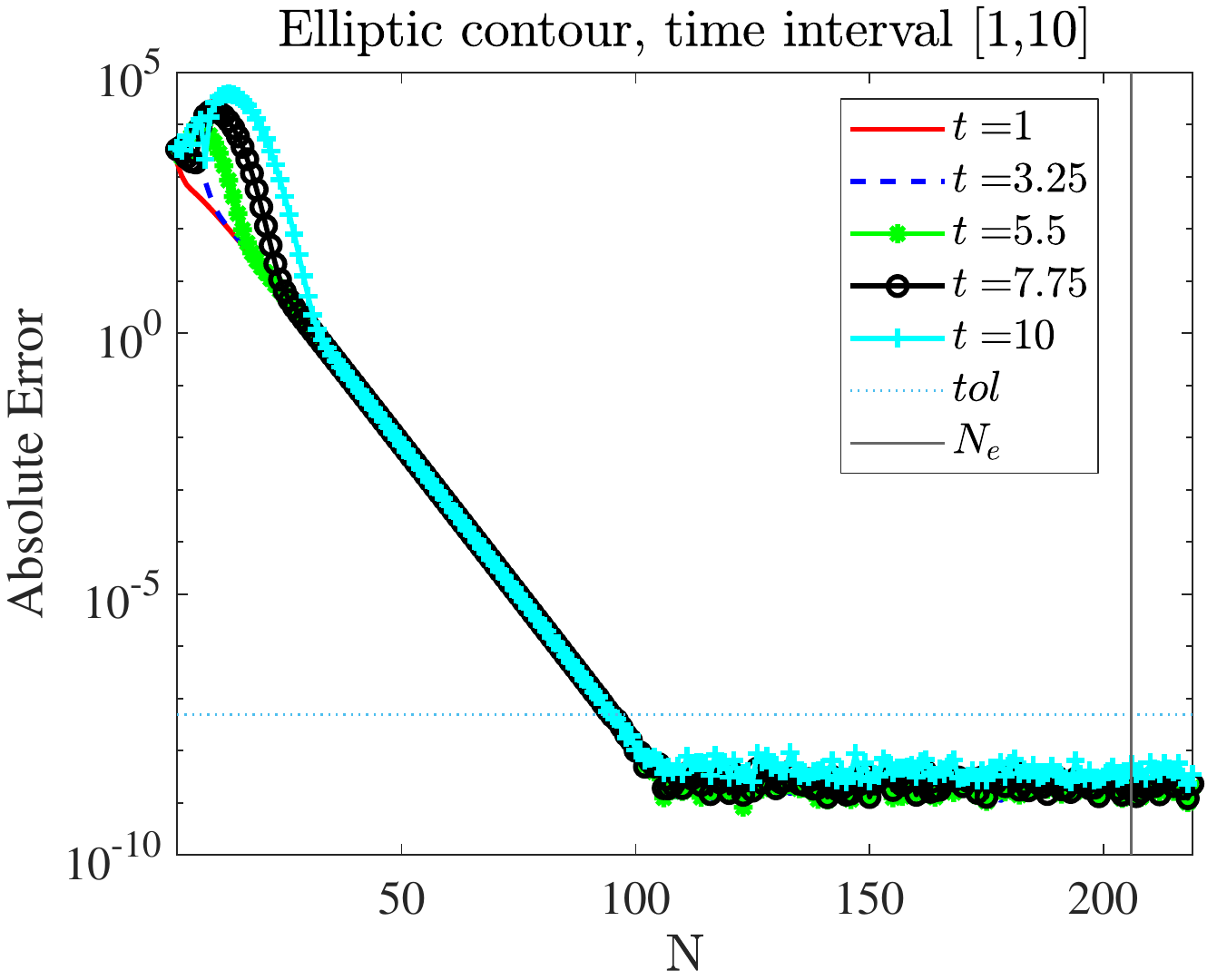}}
		\subfigure{
			\includegraphics[width=0.48\textwidth]{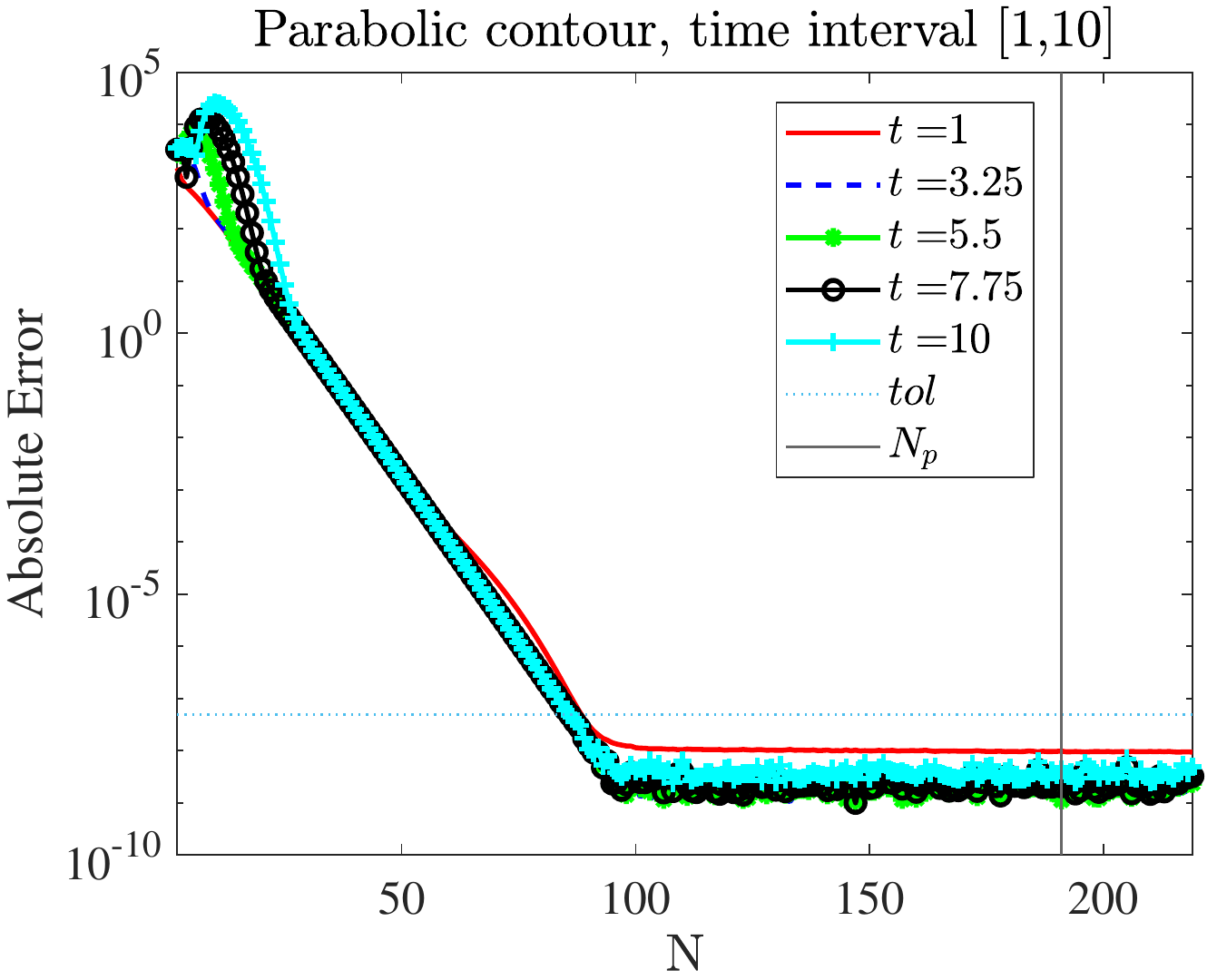}}
	}	
	\centering{
		\subfigure{
			\includegraphics[width=0.48\textwidth]{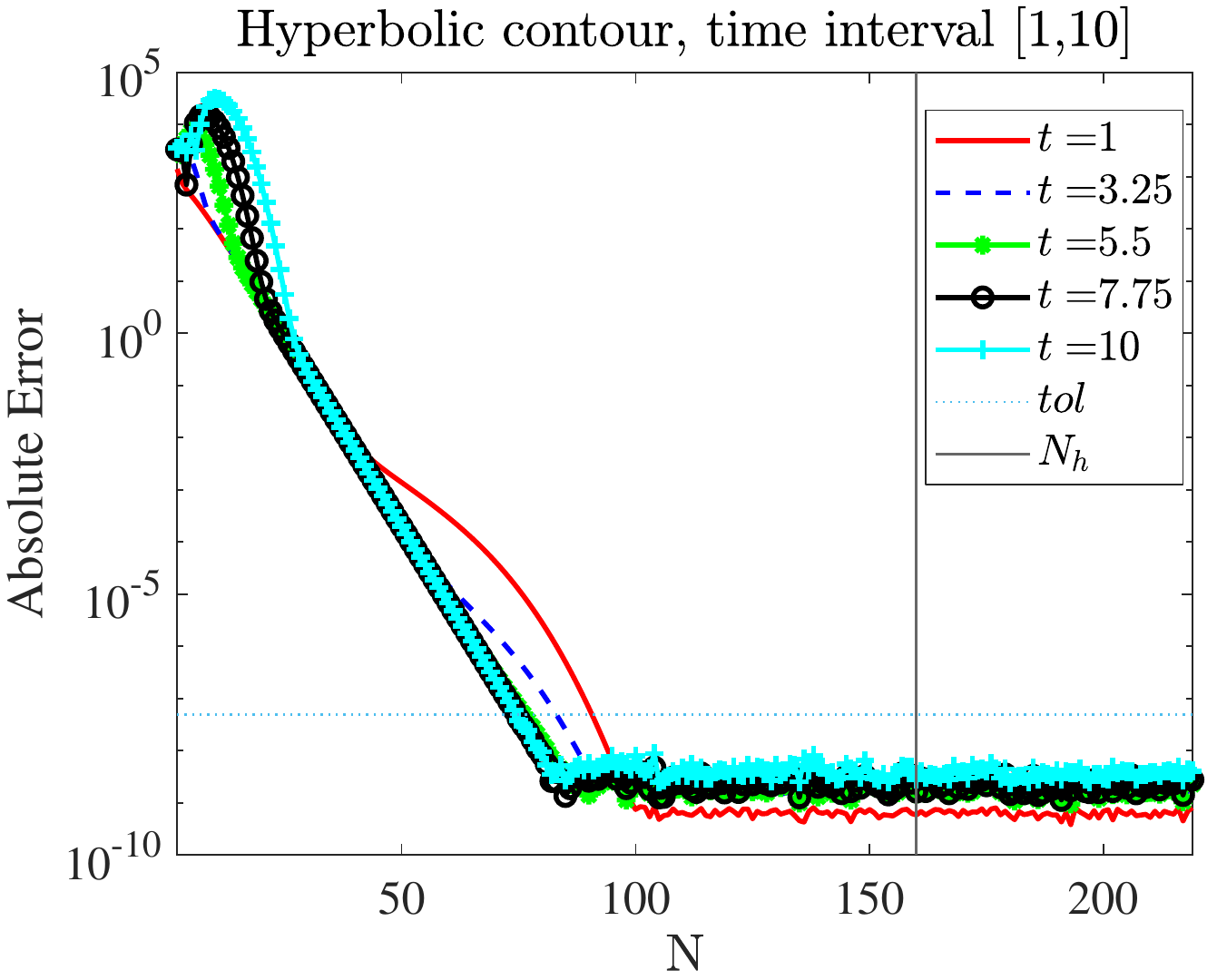}}
		
	}			
	
	\caption{Black-Scholes equation in time interval $[1,10]$, $\tol=5\cdot10^{-8}$, $z^L=-40$, $z^R=0.01$.}
	\label{fig5.8}
\end{figure}
\begin{figure}
	\centering{
		\subfigure{
			\includegraphics[width=0.48\textwidth]{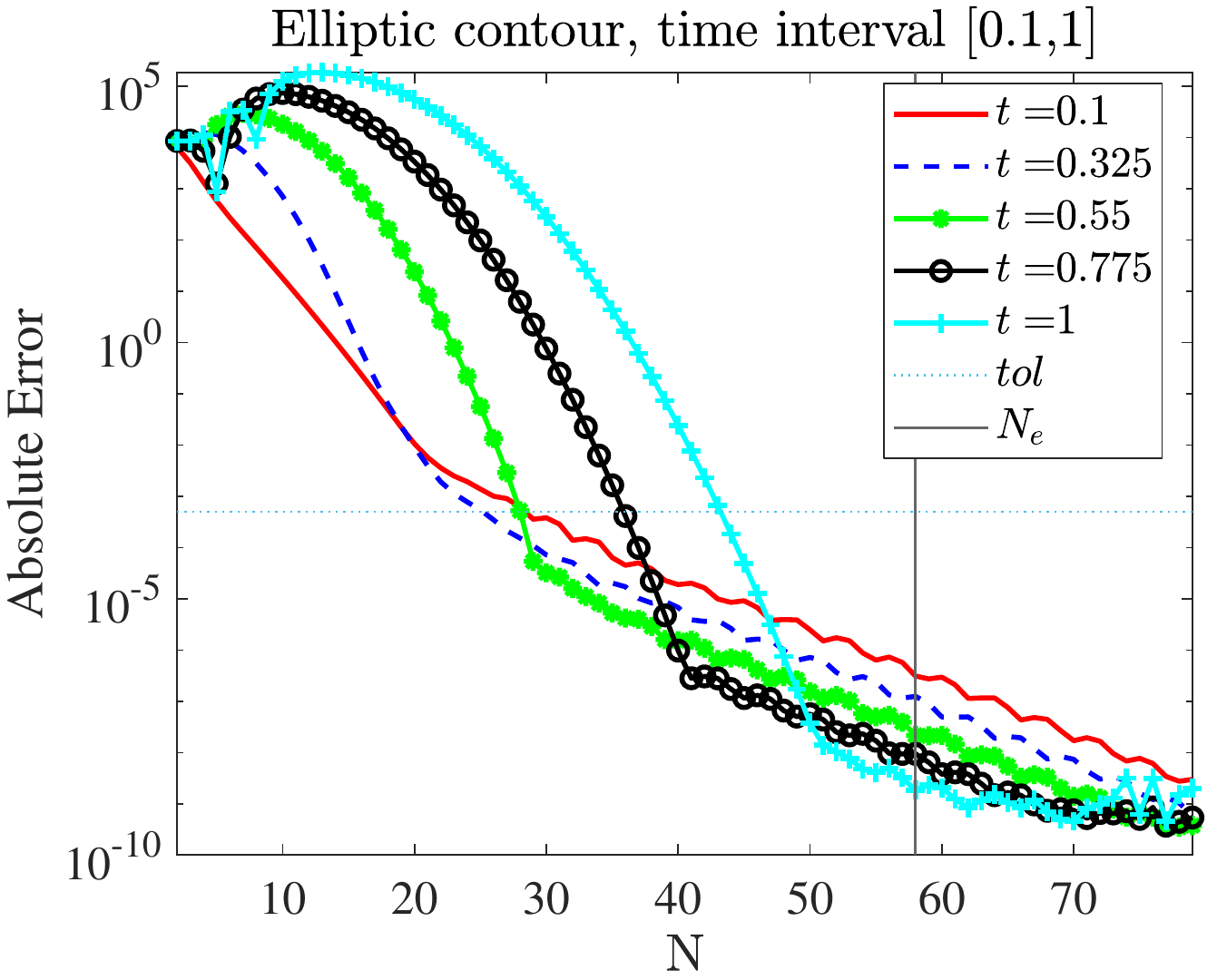}}
		\subfigure{
			\includegraphics[width=0.48\textwidth]{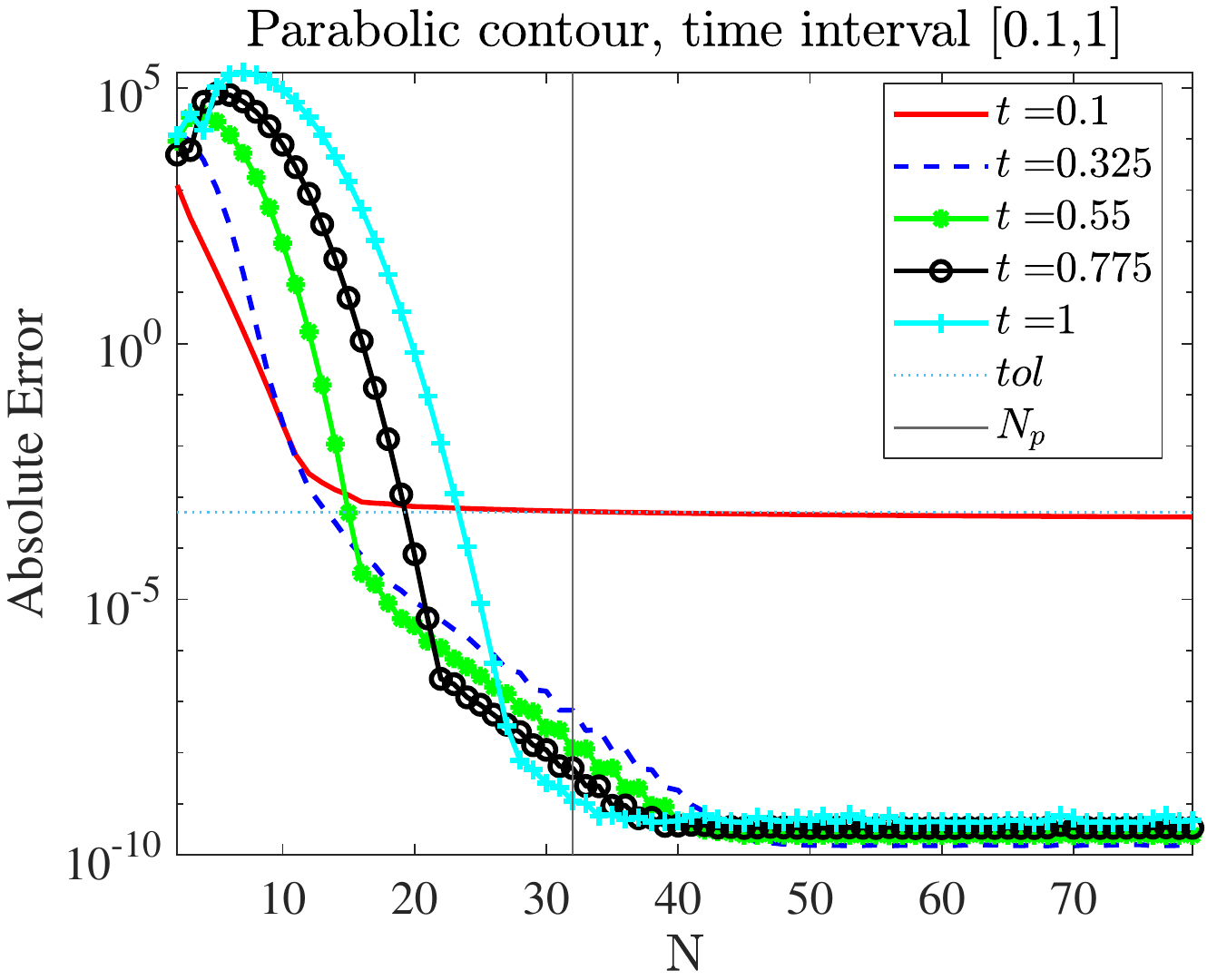}}
	}	
	\centering{
		\subfigure{
			\includegraphics[width=0.48\textwidth]{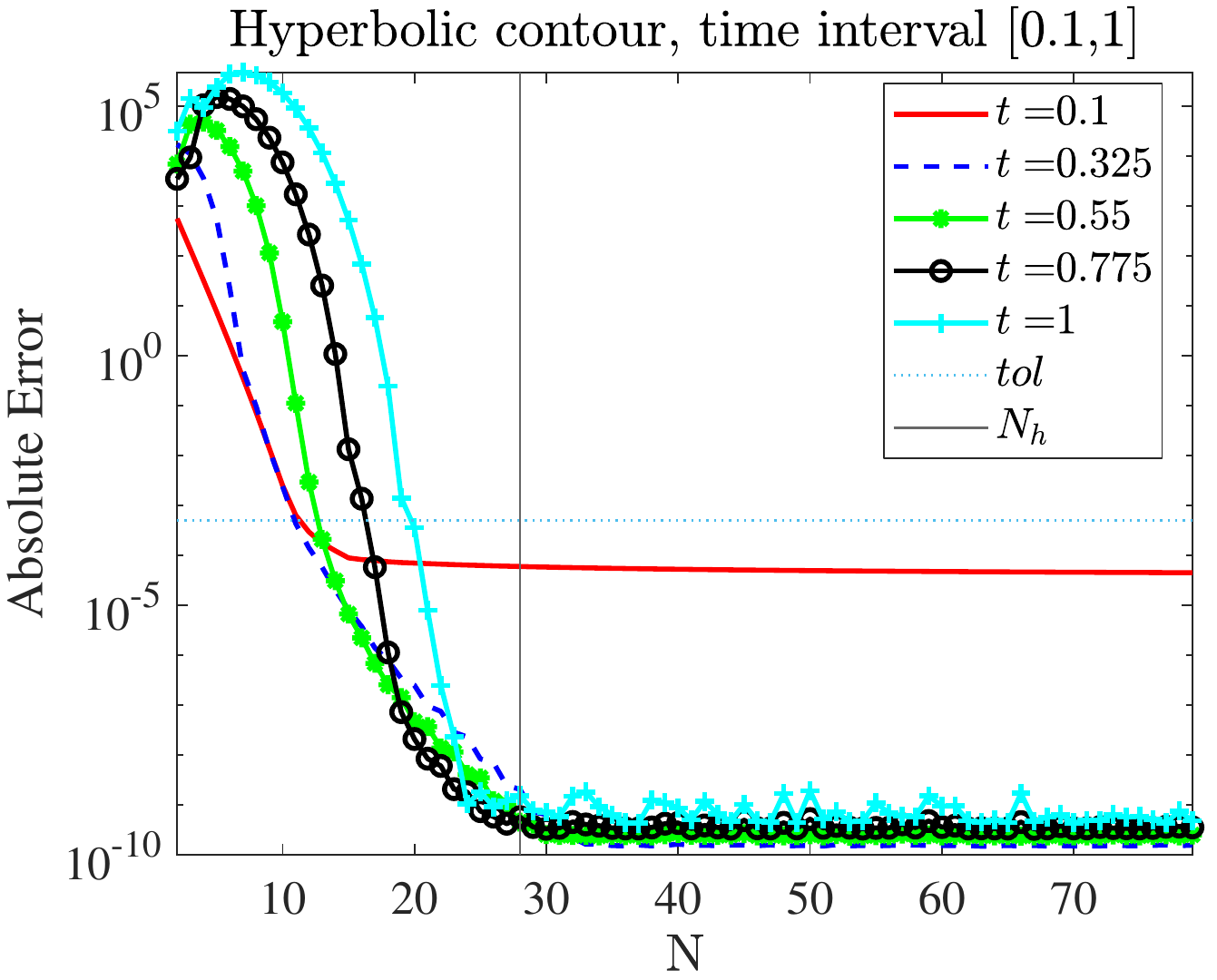}}
		
	}			
	
	\caption{Heston equation in time interval $[0.1,1]$, $\tol=5\cdot10^{-4}$, $z^L=-400$, $z^R=0.06$.}
	\label{fig5.9}
\end{figure}
\begin{figure}[ht]
	\centering{
		\subfigure{
			\includegraphics[width=0.48\textwidth]{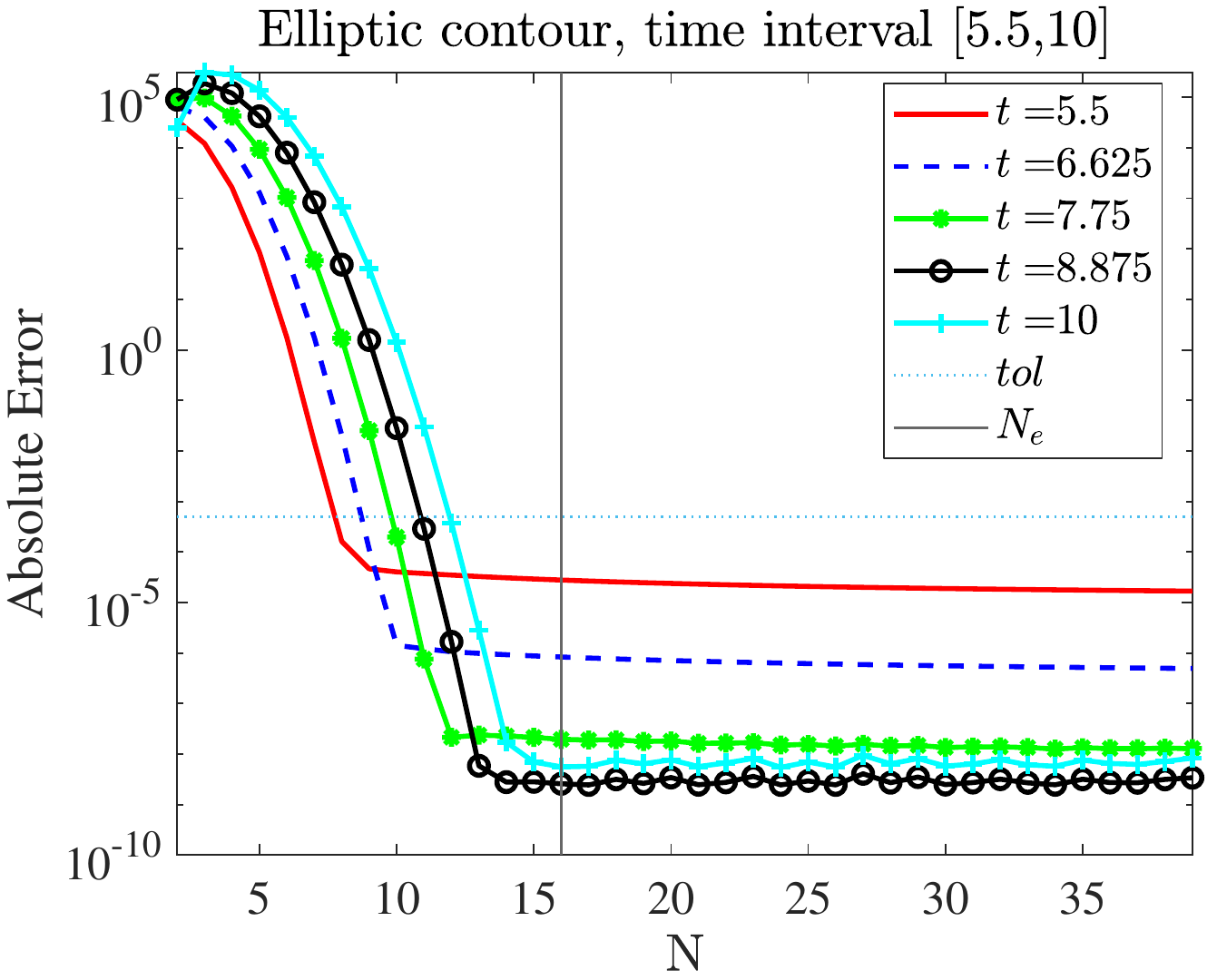}}
		\subfigure{
			\includegraphics[width=0.48\textwidth]{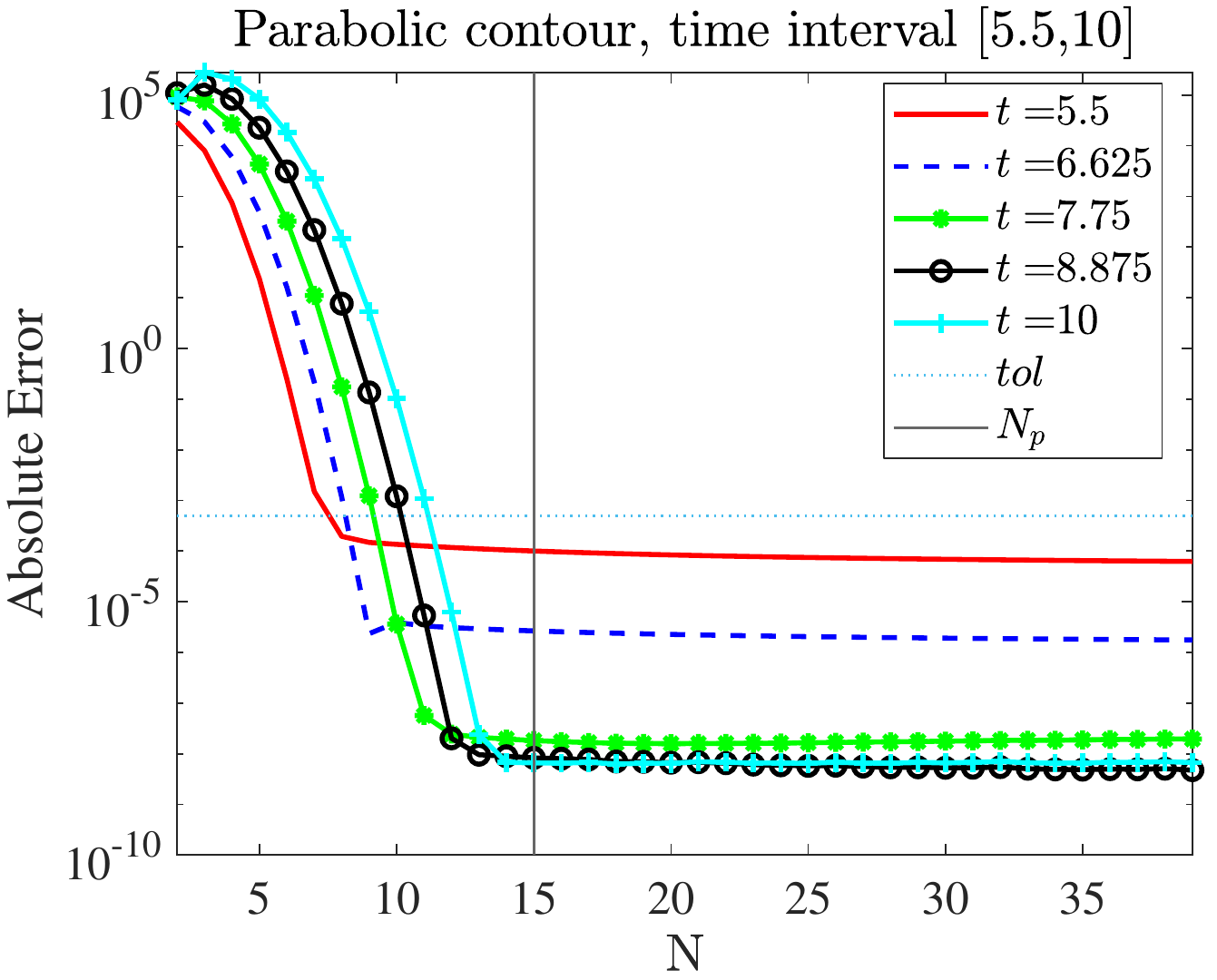}}
	}	
	\centering{
		\subfigure{
			\includegraphics[width=0.48\textwidth]{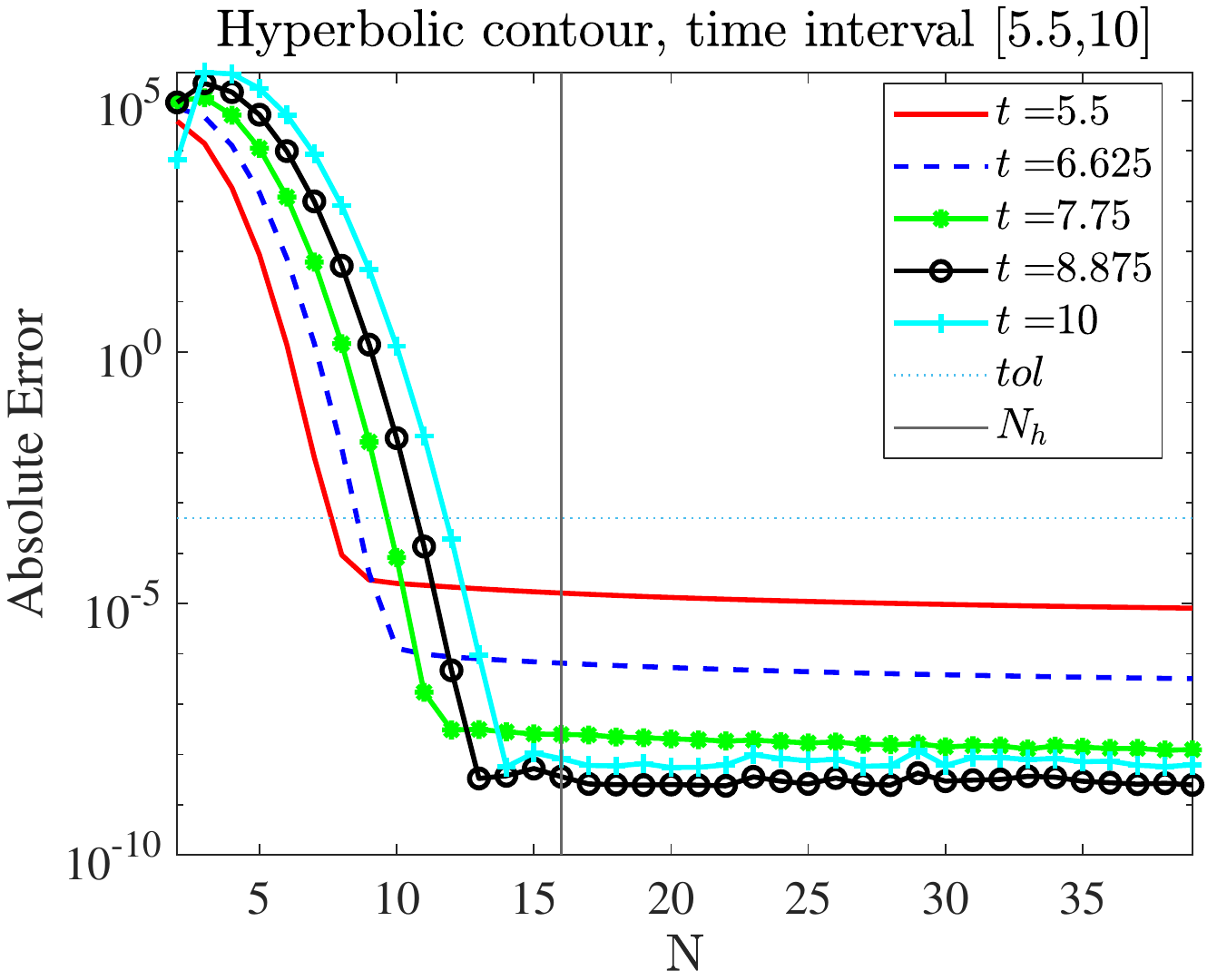}}
		
	}			
	
	\caption{Heston equation in time interval $[5.5,10]$, $\tol=5\cdot10^{-4}$, $z^L=-40$, $z^R=0.06$.}
	\label{fig5.10}
\end{figure}
\section{Implementation and codes} \label{sec:code}
The aim of this Section is to explain the structure and the main features of the code available at \cite{GLMcode}.
We assume that a space discretization is already available and thus we have the operator $A$, the Laplace transform
of the known term $\hat{b}$ and the initial solution $u_0$. The code provided includes three test problems:
a classical convection-diffusion problem, the Black and Scholes model and the Heston model.
\subsection{Construction of the inner curve and integration map}
The first step is the computation of the inner curve $\Gamma_{left}$, as described in Section \ref{sec:pseudospec};
then we determine the bandwidth of the map $a$ and the truncation parameter $c$.
Concerning $\Gamma_{left}$ there are some practical implementation remarks we want to mention.

First, given a discrete problem, we can use operators arising from coarser discretization to save computational effort in
evaluating the pseudospectra. Not so much is known about the behavior of the pseudospectra as the size of the operator increases.
It is reasonable to expect a behavior similar to the one of the condition number, i.e. that it increases as size increases.
Therefore the  magnitude of pseudospectra is most likely to be underestimated considering operators of smaller size. Anyway,
since we have exponential convergence with respect to $N$, it is sufficient to add few nodes to get a final solution under
the required accuracy.

Second, we consider the Newton iteration \eqref{eqNI}. At the beginning of the process it often happens that the perturbation
due to the gradient is such that $r^{\ell+1}\gg r^{\ell}$. If this is the case, we do the following update:
$r^{\ell+1}=r^{\ell}+pr^{\ell}$ with $p$ a suitable positive real number chosen by the user ($0.5$ in our problems).

Finally let us comment the choice of the set of points $\{ z_k \}$.
We define two grids of points, one finer than the other. We start by computing $\delta \epsilon$ in the coarser grid;
if $\delta \epsilon\le 0$ we consider the successive point in the same coarse grid. Otherwise we pass to the finer one and
we consider points there until $\delta \epsilon>0$.
This is all implemented in the functions
\begin{enumerate}
	\item \verb|Elliptic_Map|,
	\item \verb|Parabolic_Map|,
	\item \verb|Hyperbolic_Map|.
\end{enumerate}
The three functions have the same structure, it only changes the type of contour, therefore we limit our description to \verb|Parabolic_Map|.
\subsubsection{An example of implementation: the function \texttt{Parabolic\_Map}}
\addcontentsline{toc}{section}{6.1.1 An example of implementation: the function \texttt{PS parabolic profile}}
The following are the input and output arguments:
\begin{itemize}
	\item Input:
	\begin{enumerate}
		\item $A$ the discrete operator;
		\item $\hat{b}(z)$ Laplace transform of the known term;
		\item $u_0$ initial solution;
		\item $A_r$ operator of smaller size than $A$ used to compute the internal curve $\Gamma_{left}$;
		\item $\epsilon_1$ target value of the weighted pseudospectral level curve;
		\item $T$ time where to evaluate the solution;
		\item $n_X$ maximum number of points where I compute the pseudospectra;
		\item $z^L$ minimum real value;
		\item $z^R$ the internal parabola vertex position;
		\item $\tol$ accuracy required.
	\end{enumerate}
	\item Output:
	\begin{enumerate}
		\item $a_p$ uniquely defines the map, see Section \ref{sec:contour};
		\item $c$ truncation value of the integral, see Section \ref{sec:param};
		\item $a_1$ and $a_2$ defined in (\ref{a1para}) and (\ref{a2para}) respectively;
		\item $N$ number of nodes required to integrate with accuracy $\tol$.
	\end{enumerate}
\end{itemize}
The structure of the function is reported in Algorithm \ref{pp}.
\begin{algorithm}[t]
	\caption{\texttt{Parabolic\_Map}}\label{pp}
	\hspace*{\algorithmicindent} \textbf{Input:} $A,\;\hat{b},\;u_0,\;T,\;A_r,\;\varepsilon_1,\;t,\;nX,\;z^L,\;z^R,\;\tol$\\
	\hspace*{\algorithmicindent} \textbf{Output:} $a_p,\;c,\;N,\;a_1,\;a_2$
	\begin{algorithmic}[1]
		
		\STATE Construct $\Gamma_{left}$ according procedure described in Section \ref{sec:pseudospec};
		\STATE Compute $a$ by applying Algorithm~\ref{A2};
		\STATE Compute $c$ by applying Algorithm~\ref{A1};

	\end{algorithmic}
\end{algorithm}
\subsection{The integral approximation}
Once the map and the truncation value $c$ are known, the integration related to the numerical inversion of the Laplace transform can be performed. This is implemented in function \verb|InLa_Quadrature| whose structure is reported in Algorithm \ref{ppp}.
\begin{algorithm}[t]
	\caption{\texttt{InLa\_Quadrature}}\label{ppp}
	\hspace*{\algorithmicindent} \textbf{Input:} $A,\;\hat{b},\;u_0,\;T,\;c,\;z^L,\;N,\;a_1,\;a_2,\;flag$\\
	\hspace*{\algorithmicindent} \textbf{Output:} $u$
	\begin{algorithmic}[2]
		\STATE Use elliptic, parabolic or hyperbolic map $z(x)$ according to $flag$;
		\STATE Initialize $u=0$;
		\FOR {$j=\lceil N/2 \rceil\;\textit{to} \;N-1$}
		\STATE $x_j=-c\pi+j\frac{2c\pi}{N}$ ;
		\STATE $\hat{u}=(z(x_j)I-A)\backslash(\hat{b}(z(x_i)+u_0))$ ;
		\STATE $u=u+\e^{z(x_j)T}\hat{u}(x_j) z'(x_j)$ ;
		\ENDFOR
		\STATE $u=\frac{2c}{N}\imag(u)$;
	\end{algorithmic}
\end{algorithm}
The new variables introduced are: $u$ the output vector that contains the nodal values of the solution at time $t$ and $flag$, an input variable that specify which profile of integration has to be used: $flag=1$ to the elliptic profile, $flag=2$ the parabolic one and finally $flag=3$ the hyperbolic.

\subsection{Amplitude of the time window}
In Subsection \ref{subsec:TW} we have established that the amplitude for a time window $[t_0, t_1]$ is
acceptable if (\ref{6.1}) is smaller than $\tol$ when evaluated at $t_0$ and $t_1$.
To compute (\ref{6.1}) we first need to approximate the error in the numerical solution $\rho_j$, which
is not simple since it depends on many factors: the condition number of matrix $A_j=z(x_j)I-A$, the
perturbations in the matrix and in the r.h.s., the stability of the algorithm to solve the linear system
and machine precision of the solver.
Assuming the matrix and the r.h.s. are exact, we can approximate $\rho_j$ by means of the residual
associated to the solution of the linear system, i.e.
\begin{equation}
	A_j\rho_j= A_j(\hat{u}(z(x_j))-\hat{u}_j)=b-A_j(\hat{u}_j)=r_j,
\end{equation}
so that
\begin{equation}\label{7.1}
	\|\rho_j\|\le\|A_j^{-1}\|\|r_j\|.
\end{equation}
Expression (\ref{7.1}) needs the evaluation of the numerical solution at the nodes $x_j$ and the computation
of the smallest non zero singular value of $A_j$.
Since $\kappa(A_j)$ (and $\|A_j^{-1}\|$) may change significantly when evaluated at different nodes, we can
expect that (\ref{6.1}) depends only weakly on $N$. This fact allows us to use only few nodes to evaluate (\ref{6.1}), which results into a smaller computational
effort with respect Algorithm \ref{ppp}. However, the extra computational effort is justified by the fact that the solution is made available on a
continuous time window with an error below the prescribed accuracy for all $t\in[t_0, t_1]$.

\section{Conclusions}

In this article we present significant algorithmic developments of the method in \cite{GLN20} for the
approximation of convection-diffusion problems by means of the inverse Laplace transform.
The main achievements are the extension to more general quadratic contour curves, the
setting of a novel method to roam pseudospectral level sets - which makes the whole algorithm faster and independent of
the code Eigtool- and the extension of the method to approximate the solutions to the PDEs at time windows of suitable length to achieve a given target accuracy.
\subsection*{Acknowledgements}
NG acknowledges that his research was supported by funds from the
Italian MUR (Ministero dell'Universit\`a e della Ricerca) within the PRIN 2017
Project ``Discontinuous dynamical systems: theory, numerics and applications''
and by the INdAM Research group GNCS (Gruppo Nazionale di Calcolo Scientifico).

MLF acknowledges partial support by INdAM-GNCS and the Spanish grant MTM2016-75465-P.

\subsection*{Data availability}

The codes implementing the algorithms discussed in this article are publicly available at:\\
 \url{https://github.com/MattiaManucci/Contour\_Integral\_Methods.git}.

No other data are associated to the manuscript.

\end{document}